\documentclass[sn-mathphys,pdflatex]{sn-jnl}

\usepackage{mathtools}
\usepackage{amssymb}
\usepackage{color}
\usepackage{csquotes}
\MakeOuterQuote{"}
\usepackage{comment}
\usepackage{hyperref}
\usepackage{cleveref}

\jyear{2021}

\theoremstyle{thmstyleone}
\newtheorem{theorem}{Theorem}[section]
\newtheorem{proposition}[theorem]{Proposition}
\newtheorem{observation}[theorem]{Observation}
\newtheorem{lemma}[theorem]{Lemma}
\newtheorem{corollary}[theorem]{Corollary}

\theoremstyle{thmstyletwo}
\newtheorem{example}[theorem]{Example}
\newtheorem{remark}[theorem]{Remark}

\theoremstyle{thmstylethree}
\newtheorem{definition}[theorem]{Definition}
\newtheorem{notation}[theorem]{Notation}

\newcommand{\biLip}{\operatorname{biLip}}
\newcommand{\Tan}{\operatorname{Tan}}
\newcommand{\dom}{\operatorname{dom}}

\newcommand{\rad}{\operatorname{rad}}

\newcommand{\spt}{\operatorname{spt}}
\newcommand{\diam}{\operatorname{diam}}

\newcommand{\muae}{$\mu$-a.e.\ }
\renewcommand{\H}{\mathcal H}

\newcommand{\Mloc}{\mathcal M_{\mathrm{loc}}}

\newcommand{\skel}{\operatorname{skel}}
\newcommand{\cor}{\operatorname{cor}}
\newcommand{\child}{\operatorname{child}}

\renewcommand{\d}{\, \mathrm{d}}

\newcommand{\M}{\mathbb M_*}
\newcommand{\Mp}{\mathbb M_p}
\newcommand{\Mpm}{\mathbb M_{pm}}
\newcommand{\dH}{d_{\mathrm{H}}}
\newcommand{\dGH}{d_{\mathrm{GH}}}
\newcommand{\dGHs}{d_{*}}
\newcommand{\dpGH}{d_{\mathrm{pGH}}}
\newcommand{\dpmGH}{d_{\mathrm{pmGH}}}
\newcommand{\dKR}{F}
\newcommand{\dHL}{H}
\newcommand{\dG}{G}

\newcommand{\N}{\mathbb N}
\newcommand{\D}{\mathcal D}
\renewcommand{\H}{\mathcal H}

\begin{document}
	
	\author*{\fnm{David} \sur{Bate}\email{david.bate@warwick.ac.uk}}

	\affil{\orgdiv{Zeeman Building}, \orgname{University of Warwick}, \orgaddress{\city{Coventry}, \postcode{CV4 7AL}, \country{UK}}, ORCiD: \href{https://orcid.org/0000-0003-0808-2453}{0000-0003-0808-2453}}
	
	\date{8 October 2021}

	\abstract{
	We characterise rectifiable subsets of a complete metric space $X$ in terms of local approximation, with respect to the Gromov--Hausdorff distance, by an $n$-dimensional Banach space.
	In fact, if $E\subset X$ with $\mathcal{H}^n(E)<\infty$ and has positive lower density almost everywhere, we prove that it is sufficient that, at almost every point and each sufficiently small scale, $E$ is approximated by a bi-Lipschitz image of Euclidean space.\\
	
	We also introduce a generalisation of Preiss's tangent measures that is suitable for the setting of arbitrary metric spaces and formulate our characterisation in terms of tangent measures.
	This definition is equivalent to that of Preiss when the ambient space is Euclidean, and equivalent to the measured Gromov--Hausdorff tangent space when the measure is doubling.
	}	
	
	\title[Rectifiability via tangents]{Characterising rectifiable metric spaces using tangent spaces}
	
	
	\maketitle

	\subsection*{Acknowledgements}
	This work was supported by the European Union's Horizon 2020 research and innovation programme (Grant agreement No. 948021).
	
	I would like to thank the referee for a careful reading and many helpful comments that improved the exposition.

	\section{Introduction}
	
	A set $E\subset \mathbb{R}^m$ is \emph{$n$-rectifiable} if it can be covered, up to a set of $\H^n$ measure zero, by countably many Lipschitz images of $\mathbb{R}^n$ (throughout this paper, $\H^n$ denotes the $n$-dimensional Hausdorff measure).
	Rademacher's theorem implies that an $n$-rectifiable $E\subset \mathbb{R}^m$ with $\H^n(E)<\infty$ has a unique $n$-dimensional \emph{approximate tangent plane} at $\H^n$-a.e.\ $x\in E$:
	an $n$-dimensional affine subspace $W\ni x$ such that, for each $r>0$, there exists $E_r\subset E\cap B(x,r)$ with
	\[\frac{\H^n(E\cap B(x,r)\setminus E_r)}{r^n} \to 0\]
	and
	\begin{equation*}\frac{\dH(W\cap B(x,r), E_r)}{r} \to 0\end{equation*}
	as $r\to 0$.
	Here $\dH$ denotes the Hausdorff distance between two subsets of $\mathbb{R}^m$.
	A central concept of geometric measure theory is to understand sufficient conditions for rectifiability in terms of local approximations by tangent planes.
	
	The simplest such result is the converse to the above statement: if $E\subset \mathbb{R}^m$ with $\H^n(E)<\infty$, the existence of an $n$-dimensional approximate tangent plane at $\H^n$-a.e.\ point implies that $E$ is $n$-rectifiable (see \cite[Theorem 15.19]{Mattila_1995}).
	A result of Marstrand \cite{marstrand} (for $n=2$ and $m=3$) significantly strengthens this by allowing the approximating affine subspace $W$ to depend on $r$ --- the approximating tangent plane may ``rotate'' as one zooms into $E$ --- under the additional assumption of the positive lower Hausdorff density of $E$ (see \eqref{lower_density} below).
	This was later generalised by Mattila \cite{mattila} to all higher dimensions.
	The case of $n=1$ is addressed in the work of Besicovitch \cite{besicovitchII}.
	This much stronger result is crucial for many results in geometric measure theory, such as the density theorems of Marstrand \cite{marstrand}, Mattila \cite{mattila} and Preiss \cite{MR890162}.
	
	Analogously, the work of David and Semmes \cite{davidsemmes} characterises \emph{uniform} rectifiability in terms of good approximation by (rotating) tangents in the form of Jones's beta numbers \cite{jones}.
	Tolsa and Azzam \cite{ta,t} provide similar results for measures.
	
	Recently there has been significant interest in studying analysis on metric spaces.
	Inevitably, such study leads to questions regarding the geometric nature of rectifiable subsets of a metric space.
	For a metric space $X$, $E\subset X$ is \emph{$n$-rectifiable} if it can be covered, up to a set of $\H^n$ measure zero, by countably many Lipschitz images of \emph{subsets} of $\mathbb{R}^n$.
	Kirchheim \cite{MR1189747} gives a precise description of the local structure of rectifiable subsets of a metric space analogous to the description granted by Rademacher's theorem.
	From this description it is possible to deduce the existence of a tangent structure of a rectifiable metric space akin to an approximate tangent plane after two modifications.
	
	Firstly, one must replace the Hausdorff distance $\dH$ with the \emph{Gromov--Hausdorff} distance $\dGH$.
	Two metric spaces $M_1$ and $M_2$ satisfy
	\[\dGH(M_1,M_2)<\epsilon\]
	if there exist a metric space $Z$ and isometric embeddings $\iota_i\colon M_i\to Z$, $i=1,2$, such that
	\[\dH(\iota_1(M_1),\iota_2(M_2))<\epsilon\]
	(see \Cref{pGH} for a modification that is more suitable for unbounded metric spaces).
	Secondly, the affine subspace must be replaced by an $n$-dimensional Banach space; the tangent space may not be equipped with the Euclidean norm.
	
	The question of whether an $E\subset X$ with approximate tangent planes defined in this way is rectifiable has existed since the work of Kirchheim.
	More recently, this question appeared in the work of Ambrosio, Bru\'e and Semola \cite{MR3990192} when studying sets of finite perimeter in RCD metric spaces.
 	To date the question is unknown for any ambient metric space that is not a Euclidean space (in this case $\dGH$ can be replaced by $\dH$ in \eqref{conv_to_tan_metric} and so the statement is equivalent to the results of Marstrand and Mattila), even if all of the tangent spaces are assumed to be a fixed $n$-dimensional Banach space.
 		
	In this paper we answer this question affirmatively.
	In fact we prove the following stronger result.
	For $K\geq 1$ let $\biLip(K)$ be the set of metric spaces $Y=(\mathbb{R}^n,\rho)$ such that $\rho$ is $K$-bi-Lipschitz equivalent to the Euclidean norm, that is
	\[\frac{1}{K}\|x-y\|_2 \leq \rho(x,y) \leq K\|x-y\|_2 \quad \forall x,y\in\mathbb{R}^n.\]
	Rather than requiring local approximation by a fixed $n$-dimensional Banach space, we allow approximation by an element of $\biLip(K)$ that is allowed to depend on the scale.
\begin{theorem}\label{main_marstrand}
	Let $(X,d)$ be a complete metric space, $n\in\N$ and let $E\subset X$ be $\H^n$-measurable with $\H^n(E)<\infty$ and
	\begin{equation}\label{lower_density}\Theta_*^n(E,x):= \liminf_{r\to 0}\frac{\H^n(B(x,r)\cap E)}{r^n}>0 \quad \text{for } \H^n\text{-a.e. } x\in E.\end{equation}
	Suppose that, for $\H^n$-a.e.\ $x\in E$, there exist $K_x\geq 1$ and, for each $r>0$, a $Y_r\in \biLip(K_x)$  and $E_r\subset E \cap B(x,r)$ such that
	\begin{equation}\label{exclude_small_metric}\frac{\H^n(E\cap B(x,r)\setminus E_r)}{r^n} \to 0\end{equation}
	and
	\begin{equation}\label{conv_to_tan_metric}\frac{\dGH(Y_r\cap B(0,r),E_r)}{r}\to 0\end{equation}
	as $r\to 0$.
	Then $E$ is $n$-rectifiable.
\end{theorem}
Note that even in the case that $X$ is a Euclidean space, \Cref{main_marstrand} is new since it allows for local approximation by bi-Lipschitz images of $\mathbb{R}^n$.

There is a more subtle aspect in which \Cref{main_marstrand} is stronger than the results of Marstrand and Mattila.
Since our notion of a tangent space is defined via Gromov--Hausdorff convergence, we do not rely on any structure of the ambient metric space $X$.
On the other hand, the proof of Marstrand crucially relies on tangents that are affine subspaces of $\mathbb{R}^m$.
In particular, Marstrand's proof fails for tangents that are isometric copies of $\mathbb{R}^2$ when $X=(\mathbb{R}^3,\|\cdot\|_\infty)$.
However, we note that for the special case $n=1$, the idea of Besicovitch can be modified to work in any metric space;
an account of this will appear in \cite{regular}.

The use of Gromov--Hausdorff tangents may be considered as a further weakening of the concept of the rotating tangent plane.
Such a framework is necessary when working in this generality;
Even if we assume $X$ is a Banach space, it is not possible to define an approximate tangent plane using ambient linear structure.
For example consider
\[\{\chi_{[0,t]} \in L^1([0,1]) : t\in [0,1]\},\]
which is an isometric copy of $[0,1]$ but does not have an approximate tangent plane at any point.

The results of Marstrand and Mattila are a starting point for the work of Preiss \cite{MR890162}.
It is natural to consider generalisations of Preiss's tangent measures when  investigating results like \Cref{main_marstrand}.
	Recall that a non-zero Radon measure $\nu$ on $\mathbb{R}^m$ is a \emph{tangent measure} to another Radon measure $\mu$ at $x\in\spt \mu$ if there exist $r_i\to 0$ such that the pushforwards of $\mu$ under the maps $y\mapsto (y-x)/r$, scaled by $1/\mu(B(x,r_i))$, weak* converge to $\nu$.
	
	To define a tangent measure of a measure $\mu$ on a metric space $X$, one requires a definition of a limit of a sequence of pointed metric measure spaces.
	In this paper, a \emph{pointed metric measure space} $(X,d,\mu,x)$ will consist of a metric space $(X,d)$, a Borel measure $\mu\in \Mloc(X)$ (that is, $\mu$ is finite on bounded sets), and a distinguished point $x\in\spt\mu$.
	Given some notion of convergence of pointed metric measure spaces, we can define a \emph{tangent} of a metric measure space $(X,d,\mu,x)$ as any metric measure space $(Y,\rho,\nu,y)$ for which there exist $r_i\to 0$ such that
	\begin{equation}
		\label{intro_tangent}
		\left(X,\frac{d}{r_i},\frac{\mu}{\mu(B(x,r_i))}, x\right) \to (Y,\rho,\nu,y).
	\end{equation}
	
	One such notion of convergence is \emph{measured Gromov--Hausdorff convergence} and the resulting tangent spaces are known as \emph{measured Gromov--Hausdorff tangents}.
	This convergence can be metrised similarly to the Gromov--Hausdorff distance:
	In addition to taking the Hausdorff distance between the metric spaces embedded in $Z$, one also considers the distance between the pushforwards of the measures, using some canonical metric that metrises weak* convergence in $Z$.
	We construct one such metrisation of measured Gromov--Hausdorff convergence, $\dpmGH$, in \Cref{pmGH}.
		
	However, the requirement that the underlying metric spaces must Gromov--Hausdorff converge is too rigid to study rectifiable sets.
	Indeed, it is easy to find examples of rectifiable metric spaces with no tangent measures according to this definition (see \Cref{gh_example}).
	Instead, we define a metric, $\dGHs$, that only considers the distance between the pushforwards of the measures in $Z$ and disregards the Hausdorff distance between the embedded metric spaces (this will be discussed further below).
	We write $\Tan(X,d,\mu,x)$ for the set of all tangent measures to $(X,d,\mu,x)$ defined using $\dGHs$ and \eqref{intro_tangent}.
			
Our main result on tangent measures is the following.
For $K\geq 1$ let $\biLip(K)^*$ be the set of all pointed metric measure spaces that are supported on an element of $\biLip(K)$.
\begin{theorem}\label{main_thm}
	Let $(X,d)$ be a complete metric space, $n\in\N$ and let $E\subset X$ be $\H^n$-measurable with $\H^n(E)<\infty$. The following are equivalent:
	\begin{enumerate}
		\item \label{main_rect} $E$ is $n$-rectifiable;
		\item \label{main_utan} For $\H^n$-a.e.\ $x\in E$, $\Theta_*^n(E,x)>0$ and there exists an $n$-dimensional Banach space $(\mathbb{R}^n,\|.\|_x)$ such that
		\[\Tan(X,d,\H^n\vert_E,x)=\{(\mathbb{R}^n,\|.\|_x,\H^n/2^n,0)\}.\]
		\item \label{main_bliptan} For $\H^n$-a.e.\ $x\in E$, $\Theta^n_*(E,x)>0$ and there exists a $K_x\geq 1$ such that $\Tan(X,d,\H^n \vert_E,x) \subset \biLip(K_x)^*$.
	\end{enumerate}
\end{theorem}

Once \Cref{main_marstrand} is established, the main steps to proving \Cref{main_thm} are to develop the properties of $\dGHs$ and the properties of tangent measures.

The proof of \Cref{main_marstrand} follows from the combination of two results, \Cref{main_construction} and \Cref{useful_corollary}.
Roughly speaking, \Cref{main_construction} shows that, under the hypotheses of \Cref{main_marstrand}, for $\H^n$-a.e.\ $x\in E$, the following is true:
for any $\epsilon>0$ and any sufficiently small $r>0$, there exists a metric space $\tilde E \supset E$ and a continuous (in fact H\"older) map $\iota\colon [0,r]^n\to \tilde E$ such that $\iota\vert_{\partial [0,r]^n}$ is close to having Lipschitz inverse and
\begin{equation}\label{small_content_intro}\H^n_\infty (\iota([0,r]^n)\setminus E) < \epsilon r^n.\end{equation}
Here $\H^n_\infty$ denotes the $n$-dimensional Hausdorff content.

\Cref{useful_corollary} is a consequence of a theorem of \cite{perturbations} (see \Cref{perturbations}).
This theorem can be viewed as a replacement for the Besicovitch--Federer projection theorem that is valid in any complete metric space.
It considers 1-Lipschitz "non-linear" projections in place of orthogonal projections.
\Cref{useful_corollary} states that any $E\subset X$ for which each $E'\subset E$ satisfies the conclusion of \Cref{main_construction} is $n$-rectifiable.
Note that the hypotheses of \Cref{main_marstrand} are inherited by any subset (by \Cref{hd_density}).

The construction of \Cref{main_construction} is the central result of the paper.
Naturally, the idea is to take a set $G\subset E$ for which the approximations given by \eqref{conv_to_tan_metric} are in some sense uniform and "glue" together the approximating tangent planes to construct the H\"older surface.
Of course, following this approach, one will eventually encounter points in $E\setminus G$ where we halt the gluing process and we begin to define $\tilde E\setminus G$.
Therefore, at every step of the construction, it is crucial to control the size of the set of points we encounter in $E\setminus G$ to establish \eqref{small_content_intro}.
The details of this construction are discussed further in \Cref{sec_construction}.

We will see that the basic theory of tangent measures in our setting follows analogously to Preiss's tangent measures, once the properties of $\dGHs$ are established.

There are several equivalent ways to define convergence of metric measure spaces that considers only the convergence of measures and not the Gromov--Hausdorff convergence of the metric spaces: one simply has to decide upon a way to metrise weak* convergence in the metric space $Z$.
Sturm \cite{sturm} first considered the $L_2$ transportation distance between probability measures with finite variance.
Greven, Pfaffelhuber, and Winter \cite{MR2520129} consider the Prokhorov metric between probability measures on compact metric spaces.
Gigli, Mondino and Savar\'e \cite{MR2401600} simply define $(X_i,d_i,\mu_i,x_i)\to (X,d,\mu,x)$, for $\mu_i,\mu\neq 0$, if there exist isometric embeddings into $Z$ such that $x_i\to x$ and $\mu_i$ weak* converges to $\mu$.

To define $\dGHs$, we first define a metric $\dKR$ between $\mu,\nu\in\Mloc(Z)$ by duality with Lipschitz functions.
A common approach for \emph{finite} measures is to consider the metric $\dKR^1$ defined as the dual norm to the set of 1-Lipschitz functions on $Z$.
It is well known that $\dKR^1$ contains geometric information on the support of $\mu,\nu$.
One can define a metric between $\mu$ and $\nu$ by taking the infimal $\epsilon>0$ for which
\begin{equation}\label{dKR_intro}\dKR^1(\mu\vert_{B(z,1/\epsilon)},\nu\vert_{B(z,1/\epsilon)})<\epsilon,\end{equation}
for some fixed $z\in Z$.
Taking this idea further, we let $\dKR^{\epsilon}$ be the dual norm of the set of $1/\epsilon$-Lipschitz functions, and define $\dKR$ by replacing $\dKR^1$ with $\dKR^\epsilon$ in \eqref{dKR_intro} (see \Cref{dKR}).
Then from $\dKR$ one precisely obtains the Hausdorff distance between two large subsets of the supports of $\mu$ and $\nu$ (see \Cref{KR_implies_GH}).
Consequently, by defining $\dGHs$ using $\dKR$ to metrise weak* convergence in $\Mloc(Z)$, we obtain an explicit relationship between $\dGHs$ and the $\dpmGH$ distance between two large subsets (see \Cref{d*_equiv_GH}).
This allows us to use standard Gromov--Hausdorff techniques and Prokhorov's theorem to develop the theory of $\dGHs$.

However, we do not introduce $\dGHs$ as a new notion of convergence.
Indeed, it is equivalent to any of the previously mentioned notions whenever the latter is defined: they all correspond to weak* convergence in $Z$ (see \Cref{gigli_rmk}).
Rather, we use it as another, convenient description of the established notions.

Note that $\Tan(\mathbb{R}^n,\|\cdot\|_2,\mu,x)$ agrees with (isometry classes of) Preiss's tangent measures.
Also, for a sequence of uniformly doubling spaces, convergence in $\dGHs$ is equivalent to measured Gromov--Hausdorff convergence (by \cite[Theorem 3.33]{MR2401600} or \Cref{cor_doub}).
Consequently, tangent measures of doubling metric measure spaces are precisely measured Gromov--Hausdorff tangents.

To conclude the introduction we discuss results related to this work.

Assuming a set $E\subset \mathbb{R}^m$ has \emph{uniform} approximation by tangents that are isomorphic copies of $\mathbb{R}^n$, a classical construction of Reifenberg \cite{reifenberg} constructs a locally bi-H\"older parametrisation of $E$ by an $n$-dimensional ball.
This has been generalised by David and Toro \cite{toro_metric}, using Gromov--Hausdorff tangents, to any metric space.
Such a parametrisation can be used to obtain the map $\iota$ in \Cref{main_construction} and in fact one may take $\tilde E=E$.
However, these assumptions are too strong for our situation: they do not allow the exclusion of sets of small measure as in \eqref{exclude_small_metric} and the convergence in \eqref{conv_to_tan_metric} is not uniform in $x$.
Moreover, one cannot deduce that such a map exists for subsets of $E$ and so cannot use \Cref{useful_corollary} to deduce rectifiability.

Relationships between tangent spaces and parametrising maps have been explored in specific non-Euclidean metric spaces such as the Heisenberg group (or other Carnot groups) \cite{merlo,donne}.
In this case, the additional structure of a particular ambient metric space such as the Heisenberg group enables one to define a much stronger notion of a tangent than a Gromov--Hausdorff tangent.

The outline of the paper is as follows.

In \Cref{preliminaries} we recall preliminary notions of geometric measure theory and properties of doubling measures.
In \Cref{prok} we recall facts about weak* convergence of measures in metric spaces and define $\dKR$.
In \Cref{grom} we recall several constructions related to the Gromov--Hausdorff distance, incorporating both $\dH$ and $\dKR$, in a way that the definitions and properties of $\dpGH$, $\dpmGH$ and $\dGHs$ all follow from the same construction.

In \Cref{sec_construction} we prove \Cref{main_construction}: the construction of a H\"older surface under the hypotheses of sets $X\supset C\supset G$ such that, for each $x\in G$ and each sufficiently small $r>0$, $B(x,r)\cap C$ is approximated (up to sets of small measure) by bi-Lipschitz images of $\mathbb{R}^n$ (see \Cref{GTA}).

In \Cref{sec:convergence} we define $\dpmGH$ and $\dGHs$ and prove various properties, in particular the relationship between the two in \Cref{d*_equiv_GH}.
This section may be of independent interest.

\Cref{sec_tangents} contains the definition of a tangent measure using $\dGHs$ and develops the theory of tangent measures in parallel to the basic theory of Preiss \cite{MR890162}.

In \Cref{proof_of_main} we conclude the proof of \Cref{main_marstrand,main_thm}.

	\section{Preliminaries}\label{preliminaries}
	Throughout this article $n$ will denote a fixed positive integer.
	For $m\in\N$, $\mathbb{R}^m$ will denote the $m$-dimensional real vector space, $\ell_2^m=(\mathbb{R}^m,\|\cdot\|_2)$ and $\ell_\infty^m=(\mathbb{R}^m,\|\cdot\|_\infty)$.
	We also write $\ell_\infty$ for the Banach space of all bounded real sequences equipped with the supremum norm.
	
	Let $(X,d)$ be a metric space.
	If $S\subset X$ and $x\in X$ we will write
	\[d(x,S) = \inf\{d(x,s) : s\in S\}\]
	and
	\[B(S,r) = \{x\in X : d(x,S) \leq r\},\]
	the \emph{closed} neighbourhood of $S$ of radius $r$.
	We write $B(x,r)$ for $B(\{x\},r)$, the closed ball of radius $r$ centred at $x$.
	Similarly, we define
	\[U(S,r) = \{x\in X : d(x,S) < r\}\]
	and $U(x,r)=U(\{x\},r)$, the open neighbourhood of $S$ and open ball centred on $x$, respectively.
	We will often omit notation for the metric $d$ when it is not necessary.
	
	By a \emph{measure} on $X$ we mean a non-negative countably sub-additive function $\mu$ defined on the power set of $X$ with $\mu(\emptyset)=0$.
	A set $A\subset X$ is \emph{$\mu$-measurable} if
	\[\mu(F)=\mu(F\cap E)+\mu(F\setminus E)\quad \forall F\subset X.\]
	The set of all $\mu$-measurable subsets of $X$ is a $\sigma$-algebra and $\mu$ is countably additive when restricted to the set of $\mu$-measurable sets.
	A \emph{Borel measure} on $X$ is a measure for which all Borel subsets of $X$ are $\mu$-measurable.
	A Borel measure is \emph{Borel regular} if, for any $S\subset X$, there exists a Borel $B\supset S$ with $\mu(B)=\mu(S)$.
	Note that, for any $Y\subset X$, the monotonicity of $\mu$ implies $\mu(Y\cap S) = \mu(Y\cap B)$.
	The \emph{support} of $\mu$, denoted $\spt\mu$, is the smallest closed set $C\subset X$ with $\mu(X\setminus C)=0$.
	
	We write $\mathcal M(X)$ for the set of all Borel regular measures on $X$ and define $\Mloc(X)$ to be those $\mu\in \mathcal M(X)$ that are finite on all balls in $X$.
	By \cite[Theorem 2.2.2]{federer}, for any $\mu\in\Mloc(X)$, and any Borel $B\subset X$,
	\[\mu(B)=\sup\{\mu(C):C\subset B \text{ closed}\} = \inf\{\mu(U):U\supset B \text{ open}\}.\]
	Also note that, for any $\mu\in\Mloc(X)$, $\spt\mu$ is separable.
	Consequently, if $X$ is complete, any $\mu\in\Mloc(X)$ is a Radon measure.
	That is, for every Borel $B\subset X$,
	\[\mu(B)= \sup\{\mu(K): K\subset B \text{ compact}\}.\]
			
	\subsection{Rectifiable subsets of a metric space}
	For $s\geq 0$, $0\leq \delta \leq \infty$ and $A\subset X$, define
	\begin{equation*}
    \mathcal{H}_\delta^s(A) = \inf\left\{\sum_{i\in \mathbb{N}} \diam(A_i)^s : A \subset \bigcup_{i\in \mathbb{N}} A_i,\ \diam(A_i)\leq \delta\right\}
    \end{equation*}
    and
    \begin{equation*}
    \mathcal{H}^s(A) = \lim_{\delta\to 0} \mathcal{H}_\delta^s(A),
    \end{equation*}
    the $s$-dimensional \emph{Hausdorff measure} of $A$.
    For any $s\geq 0$ and any metric space $X$, $ \mathcal{H}^s$ is a Borel regular measure on $X$, see \cite[Section 4]{Mattila_1995}.
	We note that $\H^s_\infty(A)$, the \emph{$s$-dimensional Hausdorff content} of $A$, is bounded above by $\diam(A)^s$.

	For $L\geq 0$, a function $f\colon (X,d)\to (Y,\rho)$ between two metric spaces is $L$-Lipschitz if
	\[\rho(f(x),f(y)) \leq L d(x,y) \quad \forall x,y\in X.\]
	Note that, for any $s\geq 0$, $0\leq \delta\leq \infty$ and $A\subset Y$,
	\[\H_{L\delta}^s (f(A)) \leq L^s \H_\delta^s(A).\]
	In particular, Lipschitz functions increase $\H^s$ and $\H^s_\infty$ by at most a multiplicative factor of $L^s$.
	A function $f\colon (X,d)\to (Y,\rho)$ is \emph{$L$-bi-Lipschitz} if it is $L$-Lipschitz and invertible onto its image with $L$-Lipschitz inverse.
	That is, for any $x,y\in Y$,
	\[\frac{1}{L}\rho(x,y) \leq \rho'(f(x),f(y)) \leq L\rho(x,y) \quad \forall x,y\in X.\]
	
    For a $\H^s$-measurable $S\subset X$ and $x\in X$ define the \emph{upper} and \emph{lower Hausdorff densities} of $S$ at $x$ by
    \[\Theta^{*,s}(S,x):=\limsup_{r\to 0} \frac{\H^n(B(x,r))}{(2r)^s}\]
    and
    \[\Theta_{*}^s(S,x):=\liminf_{r\to 0} \frac{\H^n(B(x,r))}{(2r)^s}\]
    respectively.

    \begin{definition}\label{rectifiable}
    A $\mathcal{H}^n$-measurable $E\subset X$ is \emph{$n$-rectifiable} if there exist a countable number of Lipschitz $f_i \colon A_i \subset \mathbb{R}^n \to X$ such that
    \begin{equation*}
    \mathcal{H}^n \left( E \setminus \bigcup_{i\in \mathbb{N}}f_i(A_i)\right) = 0.    
    \end{equation*}
    A $ \mathcal{H}^n$-measurable $S\subset X$ is \emph{purely $n$-unrectifiable} if every $n$-rectifiable $E\subset X$ satisfies $\mathcal{H}^n(S\cap E)=0$.
    \end{definition}

    By the same proof as for rectifiable subsets of Euclidean space (see \cite[Theorem 15.6]{Mattila_1995}), we have the following decomposition result.
    \begin{lemma}\label{rect_decomp}
    Let $Y\subset X$ be $ \mathcal{H}^n$-measurable with $ \mathcal{H}^n(Y)<\infty$.
    There exists a decomposition $Y=E\cup S$ where $E$ is $n$-rectifiable and $S$ is purely $n$-unrectifiable.
    \end{lemma}

	Recall the following classical result regarding Hausdorff densities.
	\begin{lemma}[\cite{federer} 2.10.18]
		\label{hd_density}
		Let $s>0$ and $S\subset X$ be $\mathcal H^s$-measurable with $\mathcal H^s(S)<\infty$.
		Then
		\begin{enumerate}
			\item \label{upper_density} For $\mathcal H^s$-a.e.\ $x\in S$
			\[\Theta^{*,s}(S,x)\leq 1;\]
			\item \label{density_zero} For $\mathcal H^s$-a.e.\ $x\in X\setminus S$
			\[\Theta^{*,s}(S,x)=0.\]
		\end{enumerate}
	\end{lemma}
	
	The following theorem of Kirchheim precisely describes the local structure of rectifiable subsets of a metric space.
	Note the statements of \cref{kirchheim,hd_meas_ball} are slightly different to those in \cite{MR1189747} due to the choice of normalisation of Hausdorff measure.
	\begin{theorem}[Theorem 9 \cite{MR1189747}]
	 \label{kirchheim}
	 Let $E\subset X$ be $n$-rectifiable.
	 For $\mathcal H^n$-a.e.\ $x\in E$ there exists a norm $\|\cdot\|_x$ on $\mathbb R^n$, a map $f_x\colon X\to\mathbb R^n$ and a closed set $C_x\subset E$ such that $f_x(x)=0$,
	 \begin{equation}\label{k_density}
	 \lim_{r\to0} \frac{\mathcal H^n(B(x,r)\cap C_x)}{(2r)^n}=1
	 \end{equation}
	 and
	 \begin{equation*}
	  \limsup_{r\to 0}\left\{\left\vert1-\frac{\|f_x(y)- f_x(z)\|_x}{d(y,z)}\right\vert : y\neq z\in C_x\cap B(x,r)\right\}=0.
	 \end{equation*}
	 
	 Note that, if $\H^n(E)<\infty$, \eqref{k_density} and \Cref{hd_density} \cref{upper_density} imply
	\begin{equation}\label{density_1}
		\frac{\H^n(B(x,r))}{(2r)^n} \to 1
	\end{equation}
	and
    \begin{equation}\label{k_density_2}
     \frac{\mathcal H^n(B(x,r)\setminus C_x)}{(2r)^n} \to 0.
    \end{equation}
	\end{theorem}
	
	We also note the following.
	\begin{lemma}[Lemma 6 \cite{MR1189747}]\label{hd_meas_ball}
		If $\|\cdot\|$ is a norm on $\mathbb{R}^n$ then $\H^n(B(0,r))=(2r)^n$.
		Here $\H^n$ and $B(0,1)$ are defined with respect to $\|\cdot\|$.
	\end{lemma}

	We also require the McShane extension theorem.
	\begin{theorem}
		\label{holder-extension}
		Let $S\subset X$ and $0<\alpha\leq 1$.
		Suppose that $f\colon S\to \ell_\infty^m$ (respectively into $\ell_\infty$) is an $\alpha$-H\"older map with H\"older constant $H$.
		Then there exists an $\alpha$-H\"older extension $F\colon X\to \ell_\infty^m$ (respectively into $\ell_\infty)$ of $f$ with H\"older constant $H$.
		In particular, $F$ is $\alpha$-H\"older with H\"older constant $\sqrt{m}H$ as a map into $\ell_2^m$.

		Moreover, for each $1\leq i\leq m$ (respectively, for each $i\in\N$), $F$ may be chosen such that
		\[\inf\{f_i(x): x\in S\} \leq F_i \leq \sup\{f_i(x): x\in S\}.\]
		
		Finally, if $S\subset B(x,r)$ for some ball $B(x,r)\subset X$ then $F$ may be chosen with $\spt F\subset B(x,r')$, for $r'=r+(H^{-1}\|f\|_\infty)^{1/\alpha}$.
	\end{theorem}
	Here and throughout, $\spt F$ denotes the closure of the set \[\{x\in X: F(x)\neq 0\}.\]

	\begin{proof}
		The proof of the first part is standard, see \cite[Section 6]{MR1800917}.
		For the second part, one simply replaces the $i$th coordinate of $F$ by
		\[\max\{\inf\{f_i(x): x\in S\}, \min\{F_i, \sup\{f_i(x): x\in S\}\}\}\]
		for each $1\leq i \leq m$ (respectively for each $i\in\N$).
		
		Finally, if $S\subset B(x,r)$ then $f$ is bounded. We first extend $f$ to a map $f'$ defined on
		\[S':=S \cup (X\setminus B(x,r'))\]
		by defining $f'$ to equal zero on $X\setminus B(x,r')$. Then $f'$ is also $\alpha$-H\"older with H\"older constant $H$. Applying the second part of the theorem to $f'$ gives the final statement.
	\end{proof}

	\subsection{Doubling measures}
	Much of the theory of tangent measures requires the original measure to be asymptotically doubling.
	\begin{definition}\label{doubling}	
		For $M\geq 1$, $\mu\in \mathcal M(X)$ is \emph{$M$-doubling} if 
		\[0<\mu(B(x,2r)) \leq M \mu(B(x,r))<\infty\]
		for each $x\in X$ and each $r>0$.
		A $\mu\in\mathcal M(X)$ is \emph{asymptotically doubling} if
		\[\limsup_{r\to 0} \frac{\mu(B(x,2r))}{\mu(B(x,r))}<\infty \quad \text{for } \mu \text{-a.e. } x\in X.\]
	\end{definition}
	
	To proceed we need the following standard facts.
	
	\begin{lemma}\label{doubling_decomposition}
		Let $\mu\in\Mloc(X)$.
		For any $r>0$,
		\begin{equation}\label{lsc}x \mapsto \mu(U(x,r)) \text{ is lower semicontinuous}\end{equation}
		and
		\begin{equation}\label{usc}x \mapsto \mu(B(x,r)) \text{ is upper semicontinuous}.\end{equation}
		Consequently, for any $R,M,s >0$ the set
		\[A:=\{x\in X: \mu(B(x,2r)) \leq M \mu(B(x,r)) \ \forall 0<r<R\}\]
		is a Borel set, as are
		\[B :=\{x\in X: \mu(B(x,r)) < M r^s\ \forall 0<r<R\}\]
		and
		\[C:=\{x\in X: \mu(B(x,r)) > M r^s\ \forall 0<r<R\}.\]
	\end{lemma}
	
	\begin{proof}
	If $x_i\to x$ with $d(x_i,x)$ monotonically decreasing then
	\[U(x,r-d(x,x_i))\]
	monotonically increases to $U(x,r)$ and is contained in $U(x_i,r)$, proving \eqref{lsc}.
	Similarly,
	\[B(x,r+d(x,x_i))\]
	monotonically decreases to $B(x,r)$ and contains in $B(x_i,r)$, proving \eqref{usc}.
	Consequently, for any $R,M>0$ and $\delta\geq 0$, the set
	\[A_{R,M,\delta}:=\{x\in X: \mu(U(x,(2+\delta)r)) \leq M \mu(B(x,r)) \ \forall 0<r<R\}\]
	is closed, as is
	\[A= \bigcap_{\delta\in \mathbb Q^{+}} A_{R,M,\delta}.\]
	
	The proof that $B$ and $C$ are Borel is similar.
	\end{proof}
	
	\begin{lemma}\label{lem:locally-doubling}
		Let $\mu\in \mathcal M(X)$.
		Suppose that for some $M,R>0$ a set $Y\subset X$ satisfies
		\begin{equation}\label{local_doub}
			0<\mu(B(x,2r)) \leq M \mu(B(x,r)) < \infty \quad \forall 0<r<R,\ \forall x\in Y.
		\end{equation}
		Then for any $x\in Y$ and $r<R/2$, $B(x,r)\cap Y$ is contained in $M^4$ balls of radius $r/2$ centred at points of $B(x,r)\cap Y$.
		In particular, for any $\epsilon>0$, $B(x,r)\cap Y$ is contained in $M^{-4\log_2 \epsilon}$ balls of radius $\epsilon r$.
	\end{lemma}
	
	\begin{proof}
		Let $\mathcal N$ be a maximal disjoint $r/4$-net of $B(x,r)\cap Y$.
		Then for any $y\in \mathcal N$, \eqref{local_doub} gives
		\begin{equation}\label{this_doub}
			\mu(B(y,r/4)) \geq \frac{\mu(B(y,r/2))}{M} \geq \frac{\mu(B(y,r))}{M^2} \geq \frac{\mu(B(y,2r))}{M^3} \geq \frac{\mu(B(x,r))}{M^3}.
		\end{equation}
		For $y\in\mathcal N$ the $B(y,r/4)$ are disjoint subsets of $B(x,2r)$.
		Therefore, if $y_1,\ldots,y_N\in \mathcal N$,
		\begin{multline*}\mu(B(x,2r)) \geq \mu\left(\bigcup_{i=1}^N B(y_i,r/4)\right) = \sum_{i=1}^N \mu(B(y_i,r/4)) \\
			\geq \sum_{i=1}^N \frac{\mu(B(x,r))}{M^3} \geq \sum_{i=1}^N \frac{\mu(B(x,2r))}{M^4}
		\end{multline*}
		using \eqref{this_doub} for the penultimate inequality and \eqref{local_doub} for the final inequality.
		Thus $N\leq M^4$.

		The statement about covering by balls of radius $\epsilon r$ follows by induction.	\end{proof}
	
	\begin{theorem}\label{thm:lebesgue-density}
		Any asymptotically doubling $\mu\in \mathcal M(X)$ satisfies the Lebesgue density theorem.
		That is, for any $S\subset X$ and for \muae $x\in S$,
		\[\lim_{r\to 0} \frac{\mu(S\cap B(x,r))}{\mu(B(x,r))}=1.\]
		Such an $x\in S$ is a \emph{Lebesgue density point} of $S$.
	\end{theorem}
	
	\begin{proof}
		This is a standard result for doubling measures.
		First one proves the Vitali covering theorem \cite[Theorem 1.6]{MR1800917} and uses it to deduce the Lebesgue differentiation theorem \cite[Theorem 1.8]{MR1800917} and hence the Lebesgue density theorem for Borel sets $S \subset X$.

		The proof for asymptotically doubling measures is the same since the proof of the Vitali covering theorem works, with minor modifications, only assuming an almost everywhere countable decomposition into Borel sets, each satisfying \eqref{local_doub} for some $M,R>0$.
		Assuming $\mu$ is asymptotically doubling, \Cref{doubling_decomposition} gives such a decomposition of $\spt\mu$.

		The statement for arbitrary $S\subset X$ follows by considering a Borel $S'\supset S$ with $\mu(S')=\mu(S)$.
	\end{proof}
	
	\subsection{Prokhorov's theorem}\label{prok}
	
	A sequence $\mu_i\in \Mloc(X)$ \emph{weak*} converges to $\mu\in\Mloc(X)$ if, for every bounded and continuous $g\colon X\to \mathbb{R}$ with bounded support,
	\[\int g\d\mu_i \to \int g \d\mu.\]
	We write $\mu_i\to\mu$ to denote this convergence.
	
	Recall that, if $\mu_i\to\mu \in \Mloc(X)$ and $x_i\to x\in X$, then for every $r>0$,
		\begin{equation}\label{inf_open_ball}
			\mu(U(x,r)) \leq \liminf_{i\to\infty}\mu_i(U(x_i,r))
		\end{equation}
		and
		\begin{equation}\label{sup_closed_ball}
			\mu(B(x,r)) \geq \limsup_{i\to\infty}\mu_i(B(x_i,r)).
		\end{equation}
	Indeed, these follow by combining the standard proof of lower/upper semicontinuity for a fixed open/closed set and the proof of \cref{lsc,usc}.
	
	Recall Prokhorov's theorem, reformulated to describe convergence in $\Mloc(X)$.
	\begin{theorem}[Theorem 2.3.4 \cite{bogachev}]
		\label{prokhorov}
		Let $X$ be a complete and separable metric space and $x\in X$.
		A set $\mathcal S\subset \Mloc(X)$ is pre-compact (with respect to weak* convergence) if and only if, for every $r,\epsilon>0$:
		\begin{itemize}\item $\{\mu(B(x,r)) : \mu\in\mathcal S\}$ is bounded and
		\item There exists a compact $K\subset X$ such that
		\begin{equation}\label{tight} \mu(B(x,r)\setminus K) \leq \epsilon \quad \forall \mu\in\mathcal S.\end{equation}
		\end{itemize}
	\end{theorem}
	
	\begin{definition}
	\label{dKR}
	Let $X$ be a metric space and $x\in X$.
	For $\mu,\nu\in \Mloc(X)$, and $L,r>0$, define
	\[\dKR^{L,r}_x(\mu,\nu) = \sup\left\{\int g \d(\mu-\nu) : g\colon X\to [-1,1],\ L\text{-Lipschitz},\ \spt g\subset B(x,r)\right\}.\]
	Define $\dKR_x^L(\mu,\nu)$ to be the infimum, over all $0< \epsilon <1/2$, for which
    \[\dKR_x^{L,1/\epsilon}(\mu,\nu)<\epsilon.\]
    If no such $0<\epsilon<1/2$ exists, set $\dKR_x^L(\mu,\nu)=1/2$.
    Also define $\dKR_x(\mu,\nu)$ to be the infimum, over all $0< \epsilon <1/2$, for which
    \[\dKR_x^{1/\epsilon,1/\epsilon}(\mu,\nu)<\epsilon.\]
    If no such $0<\epsilon<1/2$ exists, set $\dKR_x(\mu,\nu)=1/2$.
\end{definition}

Note that $\dKR_x^{L,r}$ is increasing in $L,r$ and
\begin{equation}\label{dKR_L}
	\dKR_x^{L',r} \leq \frac{L'}{L}\dKR_x^{L,r}.
\end{equation}
Note that, for any Borel $S\subset X$,
\begin{equation}
	\label{dKR_subset}
	\dKR_x^{L,r}(\mu,\mu\vert_S) \leq \mu(B(x,r)\setminus S).
\end{equation}
The reason for allowing the Lipschitz constants of the functions to increase when defining $\dKR_x$ will become clear in \Cref{KR_implies_GH}.

\begin{proposition}
	\label{dKR_weak}
	Let $X$ be a metric space, $x\in X$ and $L>0$.
	Convergence in $\dKR_x$, equivalently $\dKR_x^L$, implies weak* convergence in $\Mloc(X)$ and, if $X$ is complete and separable, the converse is true.
\end{proposition}

\begin{proof}
	Equivalence of convergence in $\dKR_x$ and $\dKR_x^L$ follows from \eqref{dKR_L}.
	
	Suppose $\dKR_x(\mu_i,\mu)\to 0$ and that $g\colon X \to \mathbb{R}$ is bounded and continuous with $\spt g\subset B(x,r)$.
	The fact that
	\[\int g \d\mu_i \to \int g\d\mu\]
	follows from a standard approximation argument \cite[Section 5.1]{MR2401600}.
	Indeed, observe that the sequence
	\[g_k(y) := \inf\{g(w) + kd(w,y):w\in X \}\]
	monotonically increases to $g$ and that each $g_k$ is $k$-Lipschitz with
	\[\spt g_k \subset B(x,r+\|g\|_\infty).\]
	Thus,
	\begin{align*}\liminf_{i\to\infty}\int g\d\mu_i &\geq \limsup_{k\to\infty}\liminf_{i\to\infty} \int g_k\d\mu_i\\
	&= \limsup_{k\to\infty} \int g_k\d\mu\\
	&= \int g\d\mu,
	\end{align*}
	using monotonicity of the integral for the inequality, the fact that $\dKR_x(\mu_i,\mu)\to 0$ for the first equality and the monotone convergence theorem for the final equality.
	The reverse "limsup" equality follows from a similar argument that involves approximating $g$ from above.
	
	Now suppose that $X$ is complete and separable, that $\mu_i\to\mu$ in $\Mloc(X)$ but $\dKR_x(\mu_i,\mu)\not\to 0$.
	By taking a (non-relabelled) subsequence, there exist an $r>0$ such that $\dKR_x(\mu_i,\mu)>1/r$ for all $i\in\N$.
	That is, for each $i\in \N$ there exists an $r$-Lipschitz $g_i \colon X \to [-1,1]$ with $\spt g_i \subset B(x,r)$ such that
	\begin{equation*}
		\left\vert\int g_i \d(\mu_i -\mu) \right\vert \geq 1/r.
	\end{equation*}
	Set $r'=r+1/r$. By \Cref{prokhorov}, there exist an $M>0$ and a compact $K\subset B(x,r')$ such that, for all $i\in\N$, $\mu_i(B(x,r')) \leq M$ and
	\begin{equation}\label{phere}\mu_i(B(x,r')\setminus K) < 1/8r.\end{equation}
	By increasing $K$ if necessary, we may also suppose $\mu(B(x,r')\setminus K)<1/8r$.
	
	The Arzel\`a--Ascoli theorem implies that, after taking a further subsequence, the $g_i$ converge uniformly on $K$ to an $r$-Lipschitz $g\colon K\to [-1,1]$ with $\spt g \subset B(x,r)$.
	By \Cref{holder-extension}, we may extend $g$ to an $r$-Lipschitz function defined on the whole of $X$ with $\spt g\subset B(x,r')$.
	Then, for any sufficiently large $i$ such that
	\[\sup\{\vert g(y)-g_i(y)\vert : y\in K\}(\mu_i+\mu)(B(x,r')) \leq 1/4r,\]
	which exists by \eqref{sup_closed_ball},
	\begin{align*}
		\left\vert \int g\d(\mu_i-\mu)\right\vert &\geq \left\vert \int_{K} g \d(\mu_i-\mu)\right\vert
		 - 1/4r\\
		&\geq \left\vert \int_{K} g_i \d(\mu_i-\mu)\right\vert - 1/2r\\
		&\geq \left\vert \int g_i \d(\mu_i-\mu)\right\vert - 3/4r\geq 1/4r.
	\end{align*}
	This contradicts $\mu_i\to\mu$.	
\end{proof}

\begin{lemma}
	\label{dKR_ms}
	For any metric space $X$, $x\in X$ and $L>0$,
	$\dKR_x$ and $\dKR_x^L$ are metrics on $\Mloc(X)$.
	If $X$ is complete and separable then so are $(\Mloc(X),\dKR_x)$ and $(\Mloc(X),\dKR_x^L)$.
\end{lemma}

\begin{proof}
	Certainly $\dKR_x$ are positive, finite on $\Mloc(X)$ and symmetric.
	If $\dKR_x(\mu,\nu)=0$, then the monotonicity of $\dKR^{L,r}_x$ implies $\dKR^{L,r}_x(\mu,\nu)=0$ for every $L,r>0$.
	Consequently, the monotone convergence theorem implies $\mu(C)=\nu(C)$ for all closed and bounded $C\subset X$ and hence $\mu(B)=\nu(B)$ for all Borel $B\subset X$.
	Since $\mu$ and $\nu$ are both Borel regular, this implies $\mu=\nu$.
	
	To see that $\dKR_x$ satisfies the triangle inequality, let $\mu,\nu,\lambda\in \Mloc(X)$.
	If either of $\dKR_x(\mu,\lambda)$ or $\dKR_x(\lambda,\nu)$ equal 1/2 then the triangle inequality holds since $\dKR_x(\mu,\nu)\leq 1/2$.
	Otherwise, for $\delta>0$, let $0<\epsilon_1<\dKR_x(\mu,\lambda)+\delta$ and $0<\epsilon_2< \dKR_x(\lambda,\nu)+\delta$ satisfy
	\[\dKR_x^{\frac{1}{\epsilon_1},\frac{1}{\epsilon_1}}(\mu,\lambda) < \epsilon_1 \quad \text{and}\quad \dKR_x^{\frac{1}{\epsilon_2},\frac{1}{\epsilon_2}}(\lambda,\nu)< \epsilon_2.\]
	Then the triangle inequality and monotonicity of $\dKR_x^{L,r}$ imply
	\begin{align*}
		\dKR_x^{\frac{1}{\epsilon_1+\epsilon_2},\frac{1}{\epsilon_1+\epsilon_2}}(\mu,\nu) &\leq \dKR_x^{\frac{1}{\epsilon_1+\epsilon_2},\frac{1}{\epsilon_1+\epsilon_2}}(\mu,\lambda) + \dKR_x^{\frac{1}{\epsilon_1+\epsilon_2},\frac{1}{\epsilon_1+\epsilon_2}}(\lambda,\nu)\\
		&\leq \dKR_x^{\frac{1}{\epsilon_1},\frac{1}{\epsilon_1}}(\mu,\lambda) + \dKR_x^{\frac{1}{\epsilon_2},\frac{1}{\epsilon_2}}(\lambda,\nu)\\
		&\leq \epsilon_1+\epsilon_2\\
		&\leq \dKR_x(\mu,\lambda)+ \dKR_x(\lambda,\nu)+2\delta.
	\end{align*}
	Since $\delta>0$ is arbitrary, this shows that $\dKR_x$ is a metric.
	
	If $x_i$, $i\in \N$ is a dense subset of $X$, one easily checks that the set
	\[\{q_1\delta_{x_1}+\dots q_m \delta_{x_m}: m\in\N,\ q_i\in \mathbb Q\}\]
	is dense in $\Mloc(X)$.
	
	Finally, the completeness of $\dKR_x$ immediately follows from \cite[Corollary 2.3.5]{bogachev} (that convergence of $\int g\d\mu_i$ for every bounded Lipschitz $g$ implies \eqref{tight}), \Cref{prokhorov} and \Cref{dKR_weak}.
	The proof for $\dKR_x^L$ is similar.
\end{proof}

\subsection{Gromov distances, embeddings and convergence}\label{grom}
Pointed Gromov--Hausdorff \emph{convergence} of metric spaces is prevalent in the literature.
However, the metric defining this convergence is less commonly used.
In this subsection we recall Gromov's metric for pointed Gromov--Hausdorff convergence and prove some basic facts about it that are missing from the literature (namely that it is a complete metric on the isometry classes of pointed metric spaces).
Our presentation of this is sufficiently general so that the corresponding statements for the pointed measured Gromov--Hausdorff distance, and our metric $\dGHs$, easily follow (see \Cref{sec:convergence}).

We use the following modification of the Hausdorff distance.
\begin{definition}[Section 6 \cite{gromov2}]
	\label{pGH}
	Let $Z$ be a metric space, $z\in Z$ and $X,Y\subset Z$.
	Define $\dHL_z(X,Y)$ to be the infimum over all $0<\epsilon<1/2$ for which
	\begin{equation}
		\label{pGH_def2} X\cap B(z,1/\epsilon) \subset B(Y,\epsilon) \quad \text{and} \quad Y\cap B(z,1/\epsilon) \subset B(X,\epsilon).
	\end{equation}
	If no such $0<\epsilon<1/2$ exists, set $\dHL_z(X,Y)=1/2$.
	Note that, if \eqref{pGH_def2} is satisfied for some $\epsilon>0$ then it is satisfied for all $\epsilon'\geq \epsilon$.
\end{definition}

\begin{lemma}
	\label{dHL_pseudo}
	For any metric space $Z$ and $z\in Z$, $\dHL_z$ is a pseudometric on the power set of $Z$ and is a metric on the closed subsets of $Z$.
	If $Z$ is complete then so is the set of closed subsets of $Z$ equipped with $\dHL_z$.
\end{lemma}

\begin{proof}
	Let $\zeta$ be the metric on $Z$, let $X,Y,W\subset Z$ and let
	\[d_{XW}= \dHL_z(X,W) \quad \text{and} \quad d_{WY}=\dHL_z(W,Y).\]
	
	If one of $d_{XW}$ or $d_{WY}$ equals $1/2$ there is nothing to prove.
	Otherwise, let $0<\delta<1/2-\min\{d_{XZ},d_{ZY}\}$.
	If
	\begin{equation}\label{Hcond1}x\in X\cap B(z,1/(d_{XZ}+d_{ZY}+2\delta))\end{equation}
	then \eqref{pGH_def2} for $X$ and $W$ gives a $w\in W$ with
	\begin{equation}\label{zeta1}
		\zeta(w,x)\leq d_{XW}+\delta.
	\end{equation}
	In particular, since $d_{XW},d_{WZ}<1/2-\delta$,
	\begin{equation*}
		\zeta(w,z) \leq \zeta(w,x)+\zeta(x,z) \leq d_{XW}+\delta + 1/(d_{XW}+d_{WY}+2\delta)\leq 1/(d_{WY}+\delta).
	\end{equation*}
	Therefore \eqref{pGH_def2} for $W$ and $Y$ gives a $y\in Y$ with
	\begin{equation}
	\label{zeta2}
	\zeta(w,y)\leq d_{WY} +\delta.
	\end{equation}
	In particular,
	\begin{equation}\label{Hcond2}\zeta(x,w)\leq d_{XW} + d_{WY}+2\delta.\end{equation}
	Thus \cref{Hcond1,Hcond2} show that \eqref{pGH_def2} holds between $X$ and $W$ for $\epsilon=d_{XW} + d_{WY}+2\delta$.
	Since $\delta>0$ is arbitrary, this completes the proof.
	
	If $C\subset Z$ is closed and $\dHL_z(C,S)=0$, then for any $x\in S$ there exist $x_i\in C$ with $x_i\to x$.
	Since $C$ is closed, $x\in C$ and so $C\subset S$.
	Consequently, $\dHL_z$ is a metric on the closed subsets of $Z$.
	
	Completeness follows from the same argument as the regular Hausdorff metric \cite[Proposition 7.3.7]{burago}.
\end{proof}

In this section we consider isometric embeddings $X \to Z$ of one metric space into another.
We identify $X$, its elements and any measure on $X$, with their isometric images in $Z$.

We require the following standard construction.
\begin{lemma}
	\label{admissible_unions}
	Let $X_i$ be a sequence of metric spaces and, for each $i\leq j\in \N$, suppose that there exist isometric embeddings $X_i,X_j \to (Z_{i,j},\zeta_{i,j})$ into some metric space.
	There exists a metric space $(Z,\zeta)$ and, for each $i\in\N$, isometric embeddings $X_i\to Z$ such that, for each $i\leq j\in\N$, $w_i\in X_i$ and $w_j\in X_j$, $\zeta(w_i,w_j) \leq \zeta_{i,j}(w_i,w_j)$, with equality if $j=i+1$.
\end{lemma}

\begin{proof}
	Define a function $\zeta$ on
	\[\tilde Z=\bigsqcup_{i\in\N} X_i\]
	as follows.
	For $i\leq j\in\N$, $w_i\in X_i$ and $w_j\in X_j$, define $\zeta(w_i,w_j)$ to be the infimum of
	\begin{equation*}
	\sum_{p=1}^m\zeta_{k_{p-1},k_p}(w_{p-1},w_p)
	\end{equation*}
	over all $m\in\N$, $i=k_0 \leq k_1\leq \ldots \leq k_m = j$ and $w_k\in X_k$ for each $i < k < j$.
	Then $\zeta$ is a pseudometric on $Z$ such that $\zeta(w_i,w_j) \leq \zeta_{i,j}(w_i,w_j)$ for all $i\leq j\in\N$, $w_i\in X_i$ and $w_j\in X_j$.
	If $j=i+1$, the triangle inequality for $\zeta_{i,i+1}$ gives the reverse inequality.
	Finally, let $(Z,\zeta)$ be the metric space obtained from $(\tilde Z,\zeta)$.
	Since each $X_i\to \tilde Z$, the projection to $Z$ gives an isometric embedding $X_i\to Z$.
\end{proof}

\begin{definition}
	\label{mms}
	A \emph{pointed} metric space $(X,x)$ consists of a metric space $X$ and a distinguished point $x\in X$.
	An \emph{isometric embedding} $(X,x)\to (Y,y)$ of pointed metric spaces is an isometric embedding of metric spaces with $x=z$.
	An \emph{isometry} $(X,x)\to(Y,y)$ is a surjective isometric embedding.
	 
	A \emph{pointed metric measure space} $(X,\mu,x)$ consists of a complete and separable pointed metric space $(X,x)$ and a $\mu\in\Mloc(X)$.
\end{definition}

\begin{definition}\label{dG}
	Let $(X,\mu,x)$ and $(Y,\nu,y)$ be pointed metric measure spaces and $a,b\geq 0$.
	Define
	\[\dG_{a,b}((X,\mu,x),(Y,\nu,y))\]
	to be the infimum, over all pointed metric spaces $(Z,z)$ and all isometric embeddings $(X,x),(Y,y)\to (Z,z)$ of
	\[a \dHL_{z}(X,Y) + b \dKR_{z}(\mu,\nu).\]
\end{definition}

\begin{proposition}
	\label{G_completeness}
	For any $a,b\geq 0$, $G_{a,b}$ is a complete pseudometric on the class of pointed metric measure spaces.
	Moreover,
	\[G_{a,b}((X_i,\mu_i,x_i),(X,\mu,x))\to 0\]
	if and only if there exists a complete and separable pointed metric space $(Z,z)$ and isometric embeddings $(X_i,x_i) \to (Z,z)$ and $(X,x)\to (Z,z)$ such that
	\[a\dHL_z(X_i,X) +b\dKR_z(\mu_i, \mu)\to 0.\]
\end{proposition}

\begin{proof}

First observe the following.
Let $\zeta,\zeta'$ be metrics on a set $Z$ with $\zeta\leq \zeta'$.
Let $z\in Z$ and suppose that $X,Y\subset Z$.
Then
	\begin{equation}\label{hl_new}\dHL^{\zeta}_z(X,Y) \leq \dHL^{\zeta'}_{z}(X,Y),\end{equation}
	where the superscript denotes the metric used to define $\dHL_z$.
	Also, since any Lipschitz function on $X\cup Y\subset Z$  with respect to $\zeta$ is Lipschitz with respect to $\zeta'$, if $\mu\in \Mloc(X)$ and $\nu\in\Mloc(Y)$,
	\begin{equation}\label{kr_new}\dKR^\zeta_{z}(\mu,\nu) \leq \dKR^{\zeta'}_{z}(\mu,\nu).\end{equation}
	
	To prove the triangle inequality for $\dG_{a,b}$, for $i=1,2,3$ let $(X_i,d_1,\mu_i,x_i)$
	be a pointed metric measure space and suppose there exist isometric embeddings
	\[(X_1,d_1,x_1),(X_2,d_2,x_2)\to (Z_{1,2},\zeta_{1,2},z_{1,2})\]
	and
	\[(X_2,d_2,x_2),(X_3,d_3,x_3) \to(Z_{2,3},\zeta_{2,3},z_{2,3}).\]
	Let $\zeta_{1,3}$ be the metric on $X_1\sqcup X_3$ that equals $d_1$ and $d_3$ on $X_1$ and $X_3$ respectively and $\zeta(w_1,w_3)=1+d_1(w_1,x_1)+d_3(w_3,x_3)$ whenever $w_1\in X_1$ and $w_3\in X_3$.
	Let $(Z,\zeta)$ be the metric space obtained from \Cref{admissible_unions} and note that necessarily $x_1=x_2=x_3=:z$ in $Z$.
	The triangle inequality for $\dHL_z$ and $\dKR_z$ in $Z$ imply
	\begin{multline}\label{this_triq}a\dHL_z(X_1,X_3)+b\dKR_z(\mu_1,\mu_3) \leq a\dHL_z(X_1,X_2)+b\dKR_z(\mu_1,\mu_2)\\ + a\dHL_z(X_2,X_3)+b\dKR_z(\mu_2,\mu_3).\end{multline}
	Combining \cref{this_triq,hl_new,kr_new} gives the triangle inequality for $\dG_{a,b}$.
	
	Now let $(X_i,\mu_i,x_i)$
	be a Cauchy sequence with respect to $\dG_{a,b}$ and, for each $i\leq j\in\N$, let $(Z_{i,j},\zeta_{i,j},z_{i,j})$ be a pointed metric space for which there exist isometric embeddings
	\[(X_i,x_i),(X_j,x_j) \to (Z_{i,j},\zeta_{i,j},z_{i,j})\]
	with
	\begin{equation}\label{are_cauchy}
		a \dHL_{z_{i,j}}(X_i,X_j) + b \dKR_{z_{i,j}}(\mu_i,\mu_j) \leq 2 \dG_{a,b}((X_i,\mu_i,x_i),(X_j,\mu_j,x_j)).
	\end{equation}
	Let $(Z,\zeta)$ be the completion of the metric space given by \Cref{admissible_unions} and note that necessarily $x_1=x_2=\ldots=:z$ in $Z$.
	Since each $X_i$ is separable, so is $Z$.
	Since each $X_i$ is complete, they are closed subsets of $Z$ and so the $\mu_i\in\Mloc(Z)$.
	\Cref{hl_new,kr_new,are_cauchy} imply that the $X_i$ and $\mu_i$ are Cauchy sequences with respect to $\dHL_z$ and $\dKR_z$ in $Z$.
	Since $Z$ is complete, $X_i$ and $\mu_i$ converge to some $X$ and $\mu$.
	Consequently
	\[\dG_{a,b}((X_i,\mu_i,x_i),(X,\mu,x))\to 0.\]
	Since any convergent sequence is Cauchy, this also proves the moreover statement.
\end{proof}

\begin{definition}[Section 6 \cite{gromov2}]
	The \emph{Gromov--Hausdorff distance} between pointed metric spaces $(X,x)$ and $(Y,y)$ is defined to be
	\[\dpGH((X,x),(Y,y)) = \dG_{1,0}((X,0,x),(Y,0,y)).\]
\end{definition}

The definition of $\dpGH$ in \cite{gromov2} only requires isometric embeddings into $(Z,\zeta)$ satisfying $\zeta(x,y)\leq \epsilon$, not necessarily $x=y$.
However, the next lemma shows that two definitions are bi-Lipschitz equivalent.
We will see that fixing the distinguished point simplifies several results.
\begin{lemma}
	\label{same_basepoint}
	Let $(Z,\zeta)$ be a metric space, $X,Y\subset Z$, $x\in X$ and $y\in Y$ with $x\neq y$.
	\begin{enumerate}
		\item \label{bp1}The metric $\tilde \zeta$ on $X\sqcup Y$ that equals $\zeta$ on each of $X$ and $Y$ and
	$\tilde\zeta(w,w')=\zeta(w,w')+\zeta(x,y)$ for any $w\in X$ and $w'\in Y$ satisfies
	\begin{equation}
		\label{pseudo_greater}
		\tilde\zeta(w,x) \leq \tilde\zeta(w,y)\ \forall w\in X \quad\text{and}\quad \tilde\zeta(w,y) \leq \tilde \zeta(w,x)\ \forall w\in Y.
	\end{equation}
	\item \label{bp2} Let $\tilde\zeta$ be \emph{any} metric on $X\sqcup Y$ that equals $\zeta$ on each of $X$ and $Y$ that satisfies \eqref{pseudo_greater}.
	If $(Z',\zeta')$ is the quotient metric space of $(X\sqcup Y,\tilde\zeta)$ with respect to the relation $x\sim y$, there exist isometric embeddings $X,Y\to Z'$ with $x=y$.
	\item \label{bp3} There exists a metric space $(Z',\zeta')$ and isometric embeddings $X,Y\to Z'$ with $x=y$ such that $\zeta'\leq \zeta +\zeta(x,y)$ on $X\cup Y$.
	\end{enumerate}
\end{lemma}
\begin{proof}
The triangle inequality for $\zeta$ implies the triangle inequality for $\tilde \zeta$.
\Cref{pseudo_greater} is equivalent to the triangle inequality for $\zeta$.
This proves \cref{bp1}.

The quotient metric of any metric $\tilde \zeta$ with respect to the relation $x\sim y$ is given by
	\begin{equation*}\zeta'(w,w')=\min\{\tilde\zeta(w,w'),\tilde\zeta(w,x) + \tilde\zeta(y,w'), \tilde\zeta(w',x) + \tilde\zeta(y,z)\}.\end{equation*}
	If $\tilde\zeta$ satisfies \eqref{pseudo_greater} then $\zeta'$ equals $\zeta$ on each of $X$ and $Y$.
	Indeed, if $w,w'\in X$,
	\[\tilde \zeta(w,w') \leq \tilde \zeta(w,x) + \tilde\zeta(x,w') \leq \tilde \zeta(w,x) + \tilde\zeta(y,w'),\]
	by the triangle inequality and \eqref{pseudo_greater}.
	Thus $\zeta'(w,w')=\tilde\zeta(w,w')$.
	The corresponding statement for elements of $Y$ implies \cref{bp2}.
	
	Combining \cref{bp1,bp2} gives \cref{bp3}.
\end{proof}

The following Lemma is an adaptation of the corresponding concepts of an $\epsilon$-isometry of bounded metric spaces, see \cite[Corollary 7.3.28]{burago}.
\begin{definition}
	\label{eps_isom}
	Let $(X,d,x)$ and $(Y,\rho,y)$ be pointed metric spaces and $\epsilon>0$.
	An \emph{$\epsilon$-isometry} $f\colon (X,d,x)\to (Y,\rho,y)$ is a Borel map $f\colon B(x,1/\epsilon) \to Y$ with $f(x)=y$ such that
	\begin{equation}
	\label{eps_inj}	
	\vert\rho(f(z),f(z')) - d(z,z')\vert \leq \epsilon \quad \forall z,z'\in B(x,1/\epsilon)
	\end{equation}
	and
	\begin{equation}
	\label{eps_surj}
	B(y,1/\epsilon-\epsilon) \subset B(f(B(x,1/\epsilon)),\epsilon).	
	\end{equation}
\end{definition}

\begin{lemma}
	\label{eps_isom_pGH}
	Let $(Z,\zeta,z)$ be a pointed metric space, $X,Y$ separable subsets of $Z$ and $0<\epsilon<1/2$.
    If $\dHL_z(X,Y)<\epsilon$ then there exists a Borel map
    \[f\colon X\cap B(z,1/\epsilon) \to Y\]
    with $\zeta(f(x),x) \leq \epsilon$ for all $x\in X\cap B(z,1/\epsilon)$ and
	\begin{equation}\label{eps_surj3}Y\cap B\left(z,1/\epsilon-\epsilon\right)\subset B\left(f\left(X\cap B\left(z,1/\epsilon\right)\right),2\epsilon\right).\end{equation}
	
	In particular, if $\dpGH((X,x),(Y,y))<\epsilon$, there exists a $2\epsilon$-isometry from $(X,x)$ to $(Y,y)$.

	Conversely, if there exists an $\epsilon$-isometry from $(X,x)$ to $(Y,y)$ then there exists a pointed metric space $(Z,\zeta,z)$ and isometric embeddings
	\[(X,x),(Y,y) \to (Z,z)\]
	such that, for all $x'\in B(x,1/\epsilon)$, \[\zeta(x',f(x'))\leq 2\epsilon.\]
	In particular,
	\[\dpGH((X,x),(Y,y)) < 2\epsilon.\]
\end{lemma}

\begin{proof}
	Let $\delta>0$ be such that $\dHL_z(X,Y)<\epsilon-\delta$.
	Since $X$ is separable, there exist by countably many disjoint and non-empty Borel sets $X_j\subset X$, $j\in\N$, of diameter at most $\delta$ with $B(z,1/\epsilon)\cap X=\cup_j X_j$.
	For each $j\in\N$, pick $x_j\in X_j$.
	Since $\dHL_z(X,Y)<\epsilon-\delta$, there exists $y_j\in Y$ with $\zeta(x_j,y_j) \leq \epsilon-\delta$.
	Define $f(w)=y_j$ for each $w\in X_j$.
	Then $f$ is a Borel map and, by the triangle inequality, for $x$ in any $X_j$,
	\begin{equation}
		\label{this_eps_isom}
		\zeta(f(x),x) \leq \zeta(y_j,x_j) + \zeta(x_j,x) \leq \epsilon-\delta+\delta \leq \epsilon.
	\end{equation}
	
	To see that \eqref{eps_surj3} holds, let
	\[y\in B\left(z,\frac{1}{\epsilon}-\epsilon\right)\]
	By assumption there exists $x\in X$ with $\zeta(x,y)\leq \epsilon$.
	Then by the triangle inequality,
	\[\zeta(x,z) \leq \zeta(x,y) + \zeta(y,z) \leq \epsilon + (1/\epsilon-\epsilon) \leq 1/\epsilon.\]
	Consequently $f(x)$ is defined and, by \eqref{this_eps_isom} and the triangle inequality,
	\[\zeta(f(x),y) \leq \zeta(f(x),x) + \zeta(x,y) \leq 2\epsilon,\]
	proving \eqref{eps_surj3}.
	
	For the converse statement, let $f$ be the assumed $\epsilon$-isometry and define $\zeta$ on $X\sqcup Y$ by $\zeta=d$ on $X$, $\zeta=\rho$ on $Y$ and
	\[\zeta(z,t) = \inf\{d(z,w) + \rho(f(w),t)+\epsilon: w\in B(x,1/\epsilon)\}\]
	whenever $z\in X$ and $t\in Y$.
	Then $\zeta$ is a metric:
	if $z,z''\in X$ and $w,w''\in B(x,1/\epsilon)$, the triangle inequality for $d,\rho$ and the fact that $f$ is an $\epsilon$-isometry imply
	\begin{align*}
		d(z,z') &\leq d(z,w) +d(w,w') + d(w',z')\\
		&\leq d(z,w) + \rho(f(w),f(w')) + \epsilon + d(w',z')\\
		&\leq d(z,w) + \rho(f(w),t) + \epsilon + d(w',z')+ \rho(t,f(w')) + \epsilon
	\end{align*}
	for any $t\in Y$.
	Taking the infimum over all $w,w'\in B(x,1/\epsilon)$ shows $\zeta(z,z')\leq \zeta(z,t)+\zeta(t,z')$.
	A similar argument for $t,t'\in Y$ and $z\in X$ shows that $\zeta$ satisfies the triangle inequality.
	Note that $\zeta$ satisfies $\zeta(x,y)=\epsilon$ and \eqref{pseudo_greater}.
	Indeed, for any $w\in B(x,1/\epsilon)$, since $f$ is an $\epsilon$-isometry,
	\[d(z,x) \leq d(z,w)+ d(w,x) \leq d(w,z) + \rho(f(w),y) + \epsilon\]
	Taking the infimum over all $w\in B(x,1/\epsilon)$ gives \eqref{pseudo_greater}.
	
	Let $(Z',\zeta')$ be the metric space obtained from \Cref{same_basepoint} \cref{bp2}.
	Observe that for any $x'\in B(x,1/\epsilon)$, $\zeta'(x',f(x'))\leq \epsilon$, proving and the first containment in \eqref{pGH_def2}.
	If $y'\in Y\cap B(y,1/\epsilon -\epsilon)$ then there exists $x'\in X\cap B(x,1/\epsilon)$ with $\rho(y',f(x'))\leq \epsilon$, so that $\zeta'(x',y')\leq 2\epsilon$, proving the second containment in \eqref{pGH_def2}, for $\epsilon$ replaced by $2\epsilon$.
\end{proof}

Recall that a metric space $X$ is \emph{proper} if all closed balls in $X$ are compact.
\begin{corollary}
\label{dpGH_metric}
	The Gromov--Hausdorff distance is a complete and separable metric on the set $\Mp$ of isometry classes of proper pointed metric spaces.
	That is, if $(X,x)$ and $(Y,y)$ are proper with
	\begin{equation}\label{this_proper}\dpGH((X,x),(Y,y))=0\end{equation}
	then there exists an isometry $(X,x)\to(Y,y)$.
\end{corollary}

\begin{remark}
	Here and below, the issue of considering the set of "all" separable metric spaces can be easily avoided by using the Kuratowski (isometric) embedding of any separable metric space into $\ell_\infty$.

	We also make no distinction in notation between spaces and their equivalence classes.
	It can be easily verified that the concepts we consider to not depend on particular representatives of an equivalence class.
\end{remark}

\begin{proof}
By \Cref{G_completeness} we know that $\dpGH$ is a complete pseudometric.
Moreover, if
\[\dpGH((X,x),(Y,y))<\epsilon\]
and $B(x,1/\epsilon)$ is covered by $N$ balls of radius $\epsilon$, then $B(y,1/2\epsilon)$ is covered by $N$ balls of radius $2\epsilon$.
Thus, if $(Y,y)$ is a complete pointed metric space and is the $\dpGH$ limit of a sequence of proper pointed metric spaces, then $Y$ is proper too.
It is easily verified that the set of finite pointed metric spaces with rational distances is dense in $(\Mp,\dpGH)$.

Now let $(X,x)$ and $(Y,y)$ be proper satisfying \eqref{this_proper}
and fix $k\in\N$.
For each $j\in\N$ let $\mathcal N_j$ be a finite $1/j$-net of $B(x,k)$.
For each fixed $j\in \N$, by taking a convergent subsequence, we may suppose that the $f_i$ converge to an isometry when restricted to $\mathcal N_j$ whose image is a $2/j$-net of $B(y,k)$.
By taking a diagonal subsequence, we may suppose the $f_i$ converge to an isometry when restricted to each $\mathcal N_{j}$.
Let $\iota\colon \cup_j \mathcal N_j \to B(y,k)$ be the limiting isometry.
Since $\cup_j \mathcal N_j$ is dense in $B(x,k)$, $\iota(\cup_j\mathcal N_j)$ is dense in $B(y,k)$ and since $B(y,k)$ is complete, $\iota$ extends to an isometry $\iota\colon B(x,k)\to B(y,k)$.
Taking a convergent diagonal subsequence over $k\to\infty$ completes the proof.
\end{proof}

The Gromov compactness theorem is an immediate corollary of the completeness of $(\Mp,\dpGH)$.

Recall that a set of metric spaces $\mathcal S$ is \emph{uniformly totally bounded} if, for every $\epsilon>0$ there exists $N_\epsilon\in\N$ such that each $X\in \mathcal S$ is contained in $N_\epsilon$ balls of radius $\epsilon$.
\begin{theorem}[Section 6 \cite{gromov2}]
	\label{gromov_compactness}
	A set $\mathcal S \subset (\Mp,\dpGH)$ is pre-compact if and only if, for every $r>0$,
	\[\{B(x,r) \subset X: (X,x)\in \mathcal S\}\]
	is uniformly totally bounded.
\end{theorem}

\begin{proof}
	The given condition is equivalent to $\mathcal S$ being totally bounded in $(\Mp,\dpGH)$.
	Indeed, for $\epsilon>0$ and $(X,x)\in \mathcal S$, suppose that $B(x,1/\epsilon)$ is covered by $N_\epsilon$ balls of radius $\epsilon$.
	Let $\mathcal N$ be the set of metric spaces of the form
	\[(\{1,2,\ldots,N_\epsilon\},\rho)\]
	with
	\[\rho(x,y)\in \{\epsilon j: j\in\mathbb N,\ 0 <j\leq 1/\epsilon^2\}\]
	for each $x\neq y\in Y$ (and such that $\rho$ satisfies the triangle inequality).
	Then $\mathcal S$ is contained in the $2\epsilon$ neighbourhood of $\mathcal N$.
	
	Conversely, suppose that $\mathcal S$ is totally bounded and $\mathcal N\subset \mathcal S$ is a finite set such that $\mathcal S\subset B(\mathcal N,\epsilon)$.
	Since each $(Y,y)\in \mathcal N$ is proper and $\mathcal N$ is finite, there exists $N_\epsilon\in \N$ such that, for each $(Y,y)\in\mathcal N$, $B(y,1/\epsilon)$ is covered by $N_\epsilon$ balls of radius $\epsilon$.
	Then for any $(X,x)\in\mathcal S$, $B(x,1/2\epsilon)$ is contained in $N_\epsilon$ balls of radius $2\epsilon$.
\end{proof}

\section{Construction of a H\"older surface}\label{sec_construction}

Recall that throughout the paper we work with a fixed $n\in\mathbb N$.
For notational convenience, we define $\biLip(K)$ as the set of $K$-bi-Lipschitz images of $\ell_\infty^n$, rather than images of $\ell_2^n$.
\begin{definition}
	\label{biLip}
	For $K\geq 1$ let $\biLip(K)$ be the set of isometry classes of all pointed proper metric spaces $(X,x)$ for which there exists a surjective and $K$-bi-Lipschitz $\psi\colon \ell_\infty^n \to X$.
\end{definition}

\begin{definition}
	\label{GTA}
	Let $(X,d)$ be a complete metric space, let $G\subset C\subset X$ be closed subsets and $K\geq 1$, $\eta,R_0>0$ and $0<\delta<1/2$.

	The triple $(X,C,G)$ has \emph{good tangential approximation}, written $GTA(\eta,K,\delta,R_0)$, if both of the following conditions hold:
	\begin{enumerate}
	\item \label{GTA1} For each $x \in C$ and each $0 < r \leq R_0$,
	\begin{equation}\label{GTA_density}
		\mathcal{H} ^n(B(x,r)) \geq \eta r^n;
	\end{equation}
	\item \label{GTA2} For each $x\in G$ and each $0<r \leq R_0$, there exists a closed
	\[C'\subset B(x,r)\cap C\]
	containing $x$ with
	\begin{equation}\label{pre_fill_ball}
		\mathcal{H} ^n(B(x,r)\setminus C') < \eta (\delta r)^n
	\end{equation}
	and
	\begin{equation}
		\label{constr_close_pGH}
		\dpGH((C',d/r,x),\biLip(K))<\min\{\delta,1/K(1+2\delta)\}.
	\end{equation}
	\end{enumerate}
\end{definition}

\begin{observation}\label{bi-lip_param}
For $K\geq 1$ and $0<\delta<1/2$ let
	\[\Phi = [-K(1+2\delta), K(1+2\delta)]^n \subset \ell_\infty^n.\]
If $(X,C,G)$ satisfies $GTA(\eta,K,\delta,R_0)$ then for any $0<r\leq R_0$ there exists a Borel $\xi\colon \Phi\to C$ with $\xi(0)=x$ such that
	\begin{equation}\label{fill_ball}
		\mathcal{H}^n(B(x,r)\setminus B(\xi(\Phi),\delta r)) < \eta (\delta r)^n
	\end{equation}
	and
	\begin{equation}\label{near_bilip}
		\frac{\|q-q'\|_\infty}{K} -\delta \leq \frac{d(\xi(q),\xi(q'))}{r} \leq K\|q-q'\|_\infty + \delta
	\end{equation}
	for all $q,q'\in \Phi$.
	Indeed, let $(Y,\rho,y)\in\biLip(K)$ with
	\[\dpGH((C',d/r,x),(Y,\rho,y))<\min\{\delta,1/K(1+2\delta)\}=:\delta'\]
	and let $p\colon \ell_\infty^n \to Y$ be $K$-bi-Lipschitz and surjective with $p(0)=x$.
	Define $\xi$ by composing $p$ and the $\delta'$-isometry from $(Y,\rho,y)$ to $(C',d/r,x)$ granted by \Cref{eps_isom_pGH}.
\end{observation}

For $m\in\mathbb N$ and $R>0$ let
	\[\mathcal D(R,m) = \{R2^{-m}(j_1,\ldots,j_n) \in [0,R]^n : j_1,\ldots, j_n \in\mathbb Z \}.\]
The main result of this section is the following theorem.
\begin{theorem}\label{main_construction}
		For any $K \geq 1$, $0<\gamma<1$ and $\eta>0$ there exists $m_0\in \N$ such that the following is true.
		Suppose that $m\geq m_0$ and $G\subset C\subset X \subset \ell_\infty$ are closed sets such that $(X,C,G)$ satisfies
		\begin{equation}\label{gta_MN}GTA(\eta,K,4^{-m},R_0).\end{equation}
		Then for any $x\in G$ and $0<r\leq R_0$, there exists a $\gamma$-H\"older map
		\[\iota\colon [0,r]^n\subset \ell_\infty^n \to B(x,20Kr) \subset \ell_\infty\]
		with $\iota(0)=x$ such that:
	\begin{enumerate}
		\item \label{holder_bilip} $\iota\vert_{\mathcal D(r,m)}\colon \mathcal D(r,m) \to C$ is $(K+2^{-m})$-bi-Lipschitz onto its image;
		\item \label{holder_perturb} For any $y,z\in [0,r]^n$ with $\|y-z\|_\infty \leq 2^{-m}r$,
		\[\|\iota(y)-\iota(z)\|_\infty \leq 10K2^{-m\gamma/2} r;\]
		\item \label{holder_full_meas}\[\H^n_\infty(\iota([0,r]^n) \setminus C) \leq \Lambda \H^n(B(\iota(0),20K r) \cap X\setminus G),\]
		for some $\Lambda>0$ depending only upon $K,\gamma,m,n,\eta$.
	\end{enumerate}
\end{theorem}

Before discussing the proof of \Cref{main_construction}, we mention how the properties of $\iota$ are used in the proof of \Cref{main_thm} (the complete details are given in \Cref{3_implies_1}).

\Cref{holder_bilip} implies that the map $f:= \iota\vert_{\mathcal D(r,m)}^{-1}$ exists and is $L$-Lipschitz for some $L$ independent of $m$ and $r$.
After extending $f$ to all of $\ell_\infty$ using \Cref{holder-extension}, \Cref{main_construction} \cref{holder_bilip,holder_perturb} imply that
\begin{equation}
	\label{hhhhpert}
	\|f(\iota(y))-y\|_\infty \leq A 2^{-m} r
\end{equation}
for some constant $A$ independent of $m$ and $r$.
Provided $m$ is sufficiently large, this implies that $f(\iota([0,r]^n))$ contains $[r/4,3r/4]^n$ and hence
\begin{equation}
	\label{hhhjjjpert}\H^n(f(\iota([0,r]^n))) \geq cr^n
\end{equation}
for an absolute constant $c>0$ (see \Cref{useful_corollary}).
Provided
\[\frac{\H^n(B(\iota(0),20K r) \cap X\setminus G)}{r^n}\]
is sufficiently small (which can be guaranteed if $\iota(0)$ is a density point of $G$ and $r$ is sufficiently small), \cref{holder_full_meas} and \eqref{hhhjjjpert} imply that
\begin{equation}\label{kkkkkpert}\H^n(f(\iota([0,r]^n)\cap C))>0.\end{equation}
Moreover, if $f$ is replaced by any $L$-Lipschitz $g$ with $\|f-g\|_\infty< A2^mr/4$, then \eqref{hhhhpert} and hence \eqref{hhhjjjpert} and \eqref{kkkkkpert} also hold for $g$.
The main result of \cite{perturbations} (see \Cref{perturbations}) then implies that $C$ cannot be purely $n$-unrectifiable.
In \Cref{AR_tangent_decomp} it is shown that if $X$ satisfies the hypotheses of \Cref{main_thm}, any subset of $X$ can be covered (up to a $\H^n$ null set) by a countable union of sets which satisfy the hypotheses of $C$.
Consequently, $X$ must be $n$-rectifiable.

We now discuss the proof of \Cref{main_construction}. First note that, by scaling the norm in $\ell_\infty$, it suffices to consider the case $R_0=r=1$.
The construction of the map $\iota$ involves a multi-scale iteration over sub-cubes of $[0,1]^n$. To begin we define $\iota_1$ on $\mathcal D(1,m)$ to equal the map $\xi$ given by \Cref{bi-lip_param}. By choosing $m_0$ sufficiently large, \eqref{near_bilip} ensures that \cref{holder_bilip} of \Cref{main_construction} holds.

We now wish to extend $\iota_1$ to a function $\iota_2$ defined on (a large subset of) $\mathcal D(1,2m)$.
The basic idea is the following:
\begin{itemize}
	\item \label{iterttt} Suppose that $p\in \mathcal D(1,m)$ is such that $\iota_1(p)$ is sufficiently close to a point of $G$. Then the map $\xi$ from \Cref{bi-lip_param} provides a rich structure to the neighbourhood of $\iota_1(p)$. We wish to use $\xi$ to locally define $\iota_2$ on points in $\mathcal D(1,2m)$ that are close to $p$, whilst maintaining some control on the local Lipschitz constant of $\iota_2$. (See \Cref{basic_ext} for the fundamental building block of this extension.)
 \item \label{itertttt} On the other hand, since $\xi$ takes values in $C$, there may well exist $p\in \mathcal D(1,m)$ for which $\iota_1(p)$ lies relatively far away from $G$. At this stage we do not extend $\iota_2$ near to $p$ and this will produce a hole in the domain of $\iota$. Later we will patch over such holes, but we must control the total size of the holes we need to patch. The lower density bound \eqref{GTA_density} allows us to bound the measure of those points that lie far away from $G$ and consequently on the amount of patching required. (This is addressed much later in \Cref{prop:surface}.)
\end{itemize}
We then iterate this extension process for each $i\geq 2$, defining $\iota_i$ on a large subset of $\mathcal D(1,im)$ as an extension of $\iota_{i-1}$ and defining $\iota$ as the limit of the $\iota_i$.
The control on the local Lipschitz constant of each $\iota_i$ leads to a H\"older bound on $\iota$ as in \Cref{main_construction} \cref{holder_perturb}; The control on the total amount of patching over all scales leads to \cref{holder_full_meas}.

Before proceeding with more details, we fix notation for this section.

\begin{notation}\label{notation}
	We fix $K\geq 1$, $0<\gamma<1$ and $\eta>0$.

	Let $N\in \mathbb{N}$ be such that
	\begin{equation*}
		(5K^2)^{n} \leq 2^{N(1-\gamma)}
   \end{equation*}
   and choose $\gamma\leq\alpha<1$ such that
	\begin{equation}\label{def_N}
		 (5K^2)^{n} = 2^{N(1-\alpha)}.
	\end{equation}
	Write $l=2^{-N}$ and $\sigma = (5K^2)^n$, so that
	\begin{equation}
		\label{sigmal_obs}
		\sigma l = (5K^2)^n 2^{-N} = 2^{-N\alpha} = l^\alpha.
	\end{equation}

	Fix $(X,d)$ a complete metric space and $C,G\subset X$ such that $(X,C,G)$ satisfies $GTA(\eta,K,l/20,1)$.
\end{notation}

We note the following.
If $i\in\N$ with $i> 2/\alpha$, \cref{sigmal_obs,def_N} give
\[\sigma (\sigma l)^i = l^{\alpha-1} l^{i\alpha} < l^1 \leq 1/2.\]
That is,
\begin{equation}
		\label{beta_i_bound}
		\sigma^{i} < \frac{1}{2\sigma l^i}.
\end{equation}

\begin{definition}\label{cubes}
	For $i \in \mathbb{N}$ let
	\begin{equation*}
		\mathcal{D}(i) = \{l^{i}(j_1,\ldots,j_n)\in [0,1]^n: j_1,\ldots,j_n\in \mathbb{Z}\}.
	\end{equation*}
	For an integer $0\leq m\leq n$, an $m$-dimensional \emph{face} of side length $l^{i}$ is any set of the form
	\begin{equation*}
		\{x\in [0,1]^n : p_j \leq x_j \leq p_j+ b_j l^{i}\ \forall 1\leq j\leq n\},
	\end{equation*}
	with $p\in \mathcal{D}(i)$, $b\in \{0,1\}^n$ and exactly $m$ of the $b_i=1$.
	We denote by $\mathcal{F}(m,i)$ the set of all such faces.
	If $1\leq m \leq n$ and $F\in \mathcal{F}(m,i)$, define the \emph{boundary} of $F$ by
	\begin{equation*}
		\partial F := \bigcup \{F'\in \mathcal{F}(m-1,i) : F' \subset F\}.
              \end{equation*}
              
	We also write $\mathcal{Q}(i) = \mathcal{F}(n,i)$, the set of \emph{cubes} of side length $l^{i}$.
	For $Q\in \mathcal{Q}(i)$ the $m$-dimensional \emph{skeleton} of $Q$ is
	\begin{equation*}
		\skel(Q,m) := \bigcup \{F \in \mathcal{F}(m,i) : F \subset Q\}.
	\end{equation*}
	For $\mathcal{Q} \subset \mathcal{Q}(i)$, we define the \emph{corners} of $\mathcal{Q}$ by
\begin{equation*}
	\cor(\mathcal{Q}) = \bigcup_{Q\in \mathcal{Q}} \skel(Q,0).
\end{equation*}
We also define the \emph{children} of $\mathcal{Q}$ to be
\begin{equation*}
		\child(\mathcal{Q}) := \{Q' \in \mathcal{Q}(i+1) : Q' \subset Q, \text{ some } Q\in \mathcal{Q}\}.
\end{equation*}
If $Q\in \mathcal Q_i$ we write $\child(Q)$ for $\child(\{Q\})$.
      \end{definition}
      
To begin the construction, we define a decomposition of a given collection of cubes in order to implement the idea sketched above.

\begin{lemma}\label{choice_of_good}
	Let $i \in \mathbb{N}$, $\mathcal Q\subset \mathcal{Q}(i)$ and, for some $\beta>0$, suppose that
	\begin{equation*}
		\iota \colon \cor(\mathcal Q) \to C
	\end{equation*}
	satisfies
	\begin{equation}\label{split_iota_ass}
		d(\iota(p),\iota(p')) \leq \beta\|p-p'\|_\infty
	\end{equation}
	for each $Q\in\mathcal Q$ and each $p,p'\in \cor(Q)$.
	Then there exists a decomposition $\mathcal Q = \mathcal G\cup \mathcal B$ such that:
	\begin{enumerate}
		\item \label{choice_good_1} For each $Q\in \mathcal G$ there exists $x\in G$ with
		\begin{equation*}
		\iota(\cor(Q)) \subset B(x, 2\beta l^{i}).
		\end{equation*}
		\item \label{choice_good_2} For every $p\in\cor(\mathcal B)$,
		\begin{equation*}
		B(\iota(p), \beta l^{i})\cap G = \emptyset.
	\end{equation*}
	\end{enumerate}
\end{lemma}

\begin{proof}
	Let $Q\in \mathcal Q$.
	The triangle inequality and \eqref{split_iota_ass} imply that either \cref{choice_good_1} or \cref{choice_good_2} holds.
	Decompose $\mathcal Q$ accordingly.
\end{proof}

The proof of \Cref{main_construction} now consists of two main parts.
First, for $i\in\N$ and $\iota$ and $\mathcal Q$ satisfying the hypotheses of \Cref{choice_of_good}, we define an extension of $\iota$ to
\[D':=\bigcup \mathcal G \cap \mathcal D(i+1).\]
This extension satisfies the hypotheses of \Cref{choice_of_good} for $i+1$ with a larger value of $\beta$.
Note that we do not extend $\iota$ into
\[\bigcup \mathcal B \cap \mathcal D(i+1).\]

In the second part we iteratively construct such extensions for all $i\in \N$, using the resulting function from one iteration as the input to the next.
We show that the limiting function,
\[\iota\colon S\subset [0,1]^n\to C,\]
is $\alpha$-H\"older continuous.
A H\"older extension extends $\iota$ to
\[\iota\colon [0,1]^n \to \ell_\infty.\]
This extension simultaneously patches all holes that were created when we did not extend the function into the cubes in $\mathcal B$ after each application of \Cref{choice_of_good}.
Finally we show how the properties of $\mathcal B$ imply \Cref{main_construction} \cref{holder_full_meas}.

\subsection{Constructing the extension to the next scale}
The hypothesis that $(X,C,G)$ satisfies $GTA(\eta,K,l/20,1)$ is used to prove the following extension result.
This serves as the fundamental building block that is used construct the extension from $\D(i)$ to a large subset of $\D(i+1)$.
\begin{lemma}\label{basic_ext}
	For $0\leq m \leq n$ and $i\in \mathbb{N}$ let $F\in \mathcal{F}(m,i)$.
	For $0< \beta < 1/(2l^i)$ suppose that
	\begin{equation*}
		\iota \colon \partial F \cap \mathcal{D}(i+1) \to C
	\end{equation*}
	satisfies
	\begin{equation}\label{input_lip}
		d(\iota(p),\iota(p')) \leq \beta \|p-p'\|_{\infty}
	\end{equation}
	for each $p,p' \in \partial F \cap \mathcal{D}(i+1)$.
	Suppose that there exists $x\in G$ with
	\begin{equation}\label{req_G}
		\iota(\partial F \cap \mathcal{D}(i+1)) \subset B(x, 2 \beta l^{i}).
	\end{equation}
	Then there exists an extension of $\iota$ to
	\begin{equation*}
		\iota \colon F \cap \mathcal{D}(i+1) \to C
	\end{equation*}
	such that
	\begin{equation}\label{extension_lip}
		d(\iota(p),\iota(p')) \leq 2K^2 \beta \|p-p'\|_{\infty}
	\end{equation}
	for each $p,p'\in F \cap \mathcal{D}(i+1)$.
\end{lemma}

\begin{proof}
	By hypothesis, we have $r:=2 \beta l^{i} < 1$.
	Let $\xi \colon \Phi \to C$ be given by \Cref{bi-lip_param} for $x$ and this value of $r$.
	Let $\delta = l/20$.
	
	For a moment fix a point $p \in \partial F \cap \mathcal{D}(i+1)$.
	Since $\iota(p)\in C$, \Cref{GTA} \cref{GTA1} implies
	\begin{equation*}
		\mathcal{H} ^n(B(\iota(p),\delta r)) \geq \eta(\delta r)^n.
	\end{equation*}
	Since $\iota(p)\in B(x,2\beta l^i)$, \eqref{fill_ball} implies that there exists $z\in B(\xi(\Phi),\delta r)$ with
	\begin{equation*}
		d(z,\iota(p)) \leq \delta r.
	\end{equation*}
	That is, there exists a point $\psi(p)\in \Phi$ with
	\begin{equation}\label{ate_saw}
		d(\iota(p),\xi(\psi(p))) \leq 2\delta r.
	\end{equation}
	The function $\psi \colon \partial F \cap \mathcal{D}(i+1) \to \Phi$ is $\frac{3K\beta}{2r}$-Lipschitz.
	Indeed, for any $p,p'\in \partial F \cap \mathcal{D}(i+1)$, \cref{near_bilip} implies
	\begin{equation*}
	\|\psi(p)-\psi(p')\|_{\infty} \leq \frac{K}{r} d(\xi(\psi(p)), \xi(\psi(p'))) + \delta K.
	\end{equation*}
	Combining this with \cref{ate_saw} and using the triangle inequality gives
	\begin{equation*}
		\|\psi(p)-\psi(p')\|_{\infty} \leq \frac{K}{r} d(\iota(p), \iota(p')) + 5\delta K.
	\end{equation*}
	Finally, \cref{input_lip} and the choice of $r$ gives
	\begin{align*}
		\|\psi(p)-\psi(p')\|_{\infty} &\leq \frac{K \beta}{r} \|p - p'\|_\infty + 5\delta K\\
		&= \frac{K\beta}{r} (\|p-p'\|_\infty + 10\delta l^i)\\
		&= \frac{K\beta}{r} (\|p-p'\|_\infty +  l^{i+1}/2)\\
		&\leq \frac{3K\beta}{2r} \|p-p'\|_\infty.
	\end{align*}
	This implies that $\psi$ is $\frac{3K \beta}{2r}$-Lipschitz as claimed.
	Since $\psi$ takes values in $\Phi$, the second conclusion of \Cref{holder-extension} gives an extension of $\psi$ to a $\frac{3K \beta}{2r}$-Lipschitz function (with respect to $\|\cdot\|_\infty$ in both the domain and image)
	\begin{equation*}
		\psi \colon F \cap \mathcal{D}(i+1) \to \Phi.
	\end{equation*}

	For each $p\in F \cap \mathcal{D}(i+1)\setminus \partial F$, define $\iota(p)= \xi(\psi(p))$.
	To check \eqref{extension_lip}, let $p,p'\in F \cap \mathcal{D}(i+1)$.
	Note that, regardless of whether $p$ or $p' \in \partial F$, \cref{ate_saw} and the triangle inequality give
	\begin{equation*}
		d(\iota(p), \iota(p')) \leq d(\xi(\psi(p)), \xi(\psi(p'))) + 4\delta r.
	\end{equation*}
	Combining this with \eqref{near_bilip} gives
	\begin{equation*}
		d(\iota(p), \iota(p')) \leq Kr\|\psi(p)-\psi(p')\|_{\infty} + 5\delta r
	\end{equation*}
	Using the fact that $\psi$ is $\frac{3K \beta}{2r}$-Lipschitz gives
	\begin{equation*}
		d(\iota(p), \iota(p')) \leq \frac{3K^2\beta}{2} \|p-p'\|_{\infty} + 5\delta r.
	\end{equation*}
    Finally, substituting in for $r$ and $\delta$ gives
    \begin{equation*}
    	d(\iota(p), \iota(p')) \leq \frac{3K^2\beta}{2} \|p-p'\|_{\infty} + \frac{\beta l^{i+1}}{2} \leq 2K^2\beta\|p-p'\|_{\infty},
    \end{equation*}
	giving \eqref{extension_lip}.
\end{proof}

Suppose we are given $\iota \colon \mathcal{D}(i) \to C$ that we wish to extend to $\mathcal{D}(i+1)$.
One way to do this is to consider each $Q\in \mathcal{Q}(i)$ one at a time and apply \Cref{basic_ext} several times to extend $\iota$ to $Q \cap \mathcal{D}(i+1)$ (for the purposes of this discussion, ignore the requirement given in \eqref{req_G}).
Such an extension must agree with the value of $\iota$ at those points in $\partial Q\cap\mathcal{D}(i+1)$ that belong to the boundary of adjacent cubes to which we have previously extended $\iota$.
However, each time \Cref{basic_ext} is applied in this way, the Lipschitz constant of the extension, given by \eqref{extension_lip}, is a multiplicative factor larger than the input Lipschitz constant given in \eqref{input_lip}.
Consequently, if this is repeated for each $Q\in \mathcal{Q}(i)$, the Lipschitz constant of the extension to the whole of $\mathcal{D}(i+1)$ is far too large, and such an extension is useless.

The correct approach is to first extend $\iota$ to all points of
\[S:=\mathcal D(i+1) \cap \bigcup\mathcal F(m,1)\]
by applying \Cref{basic_ext} one face at a time.
For any $F\in\mathcal F(m,1)$,
\[\partial F \cap \mathcal D(i+1)\subset \mathcal D(i).\]
Since $\iota$ is defined on $\mathcal D(i)$, it can be shown that the Lipschitz constant increases by a fixed amount, independent of the number of $F\in\mathcal F(m,1)$ to which the extension is applied.
We then extend to
\[\mathcal D(i+1) \cap \bigcup \mathcal F(m,2)\]
one face at a time.
Similarly to the previous case, for any $F\in\mathcal F(m,2)$,
\[\partial F\cap \mathcal D(i+1)\subset S.\]
Since $\iota$ is now defined on $S$, as in the previous case, the Lipschitz constant only increases once, regardless of the number of $F\in\mathcal F(m,2)$ to which we apply \Cref{basic_ext}.
This is repeated $n$ times in total, so that we extended $\iota$ to all of $\mathcal D(i+1)$.

To facilitate this iteration we make the following simple observation.
\begin{observation}\label{observation}
  For $0\leq m\leq n$ and $i\in\N$, let $Q,Q'\in \mathcal{Q}(i)$, $p\in \skel(Q,m)\cap \mathcal{D}(i+1)$ and $p'\in \skel(Q',m)\cap \mathcal{D}(i+1)$.
  Suppose that, for some $F \in \mathcal F(m,i)$, $p\in F$ and $p' \not\in F$.
  Then there exists $q\in \partial F \cap \mathcal D(i+1)$ such that
  \begin{equation}\label{eq:1}
    \max\{\|p-q\|_{\infty}, \|q-p'\|_{\infty}\} \leq \|p-p'\|_{\infty}.
  \end{equation}
\end{observation}

\begin{proof}
  Without loss of generality, we may suppose that
  \begin{equation*}
    F = \{x \in [0,1]^{n} : 0 \leq x_{j} \leq l^{i} \ \forall 1\leq j \leq m\}.
  \end{equation*}
  If $p\in \partial F$ then we pick $q=p$.
  Otherwise, we know that
  \begin{equation*}
    0 < p_{1},\ldots,p_{m} < l^{i}.     
  \end{equation*}

  There are two cases:
  \begin{enumerate}
  \item For some $1\leq j \leq m$, $p'_{j}=0$ or $p'_{j} \geq l^{i}$;
  \item $0 < p'_{1},\ldots,p'_{m}< l^{i}$.
    In this case, since $p'\in \skel(Q',m)$ and hence belongs to some $m$-dimensional face, the other components of $p'$ must be integer multiples of $l^{i}$.
    Since $p'\not\in F$, one such component, say $p'_{j}$ differs from $p_{j}$.
    Consequently, $\vert p'_{j}-p_{j}\vert\geq l^{i}$.
  \end{enumerate}
  In the first case, define $q_{k}=p_{k}$ for each $k \neq j$ and let $q_{j}$ equal either $0$ or $l^{i}$, such that $q_{j}$ lies between $p_{j}$ and $p'_{j}$.
  Then $q\in \partial F$, $q_k=p_k$ for all $k\neq j$ and
  \begin{equation*}
    \vert p_{j}-p'_{j}\vert = \vert p_{j}-q_{j}\vert + \vert q_{j} - p'_{j}\vert,
  \end{equation*}
  which implies \eqref{eq:1}.
  In the second case, arbitrarily pick $q\in \partial F \cap \mathcal{D}(i+1)$.
  Then $\|p-q\|_{\infty}\leq l^{i} \leq \|p-p'\|_{\infty}$, implying one inequality in \eqref{eq:1}.
  Also, $\vert p'_{k}-q_{k}\vert\leq l^{i} \leq \vert p'_{j}-p_{j}\vert$ for all $1\leq k\leq m$ and $\vert p'_{k}-q_{k}\vert = \vert p'_{k}-p_{k}\vert$ for all $k>m$.
  Therefore, $\|p'-q\|_{\infty}
  \leq\|p'-p\|_{\infty}$, proving the other inequality in \eqref{eq:1}.
\end{proof}

With this observation we demonstrate how to construct the extension from the $m$-dimensional faces to the $m+1$-dimensional faces.
\begin{lemma}\label{skel_iter}
	Fix $i\in \mathbb{N}$, let $D\subset \mathcal{D}(i+1)$ and let $0\leq m <n$.
	Let $0<\beta < 1/(2 l^{i})$ and suppose that $\iota\colon D\to C$ satisfies
	\begin{equation}\label{skel_input_lip}
		d(\iota(p),\iota(p'))\leq \beta \|p-p'\|_{\infty}
	\end{equation}
	for each $p,p'\in D$ with $\|p-p'\|_\infty \leq l^i$.
	Suppose that $\mathcal{G}\subset \mathcal{Q}(i)$ is such that
	\begin{equation}\label{skel_ass}
		\forall Q \in \mathcal{G},\ D \cap Q = \skel(Q,m) \cap \mathcal{D}(i+1)
	\end{equation}
	and
	\begin{equation}\label{C_near_ass}
		\forall Q \in \mathcal{G},\ \exists x \in G \text{ with } \iota(D\cap Q) \subset B(x, 2\beta l^{i}).
	\end{equation}
	Then there exists an extension of $\iota$ to
	\begin{equation*}
		D' := D \cup \bigcup_{Q\in \mathcal{G}} \skel(Q,m+1) \cap \mathcal{D}(i+1)
	\end{equation*}
	such that
	\begin{equation}\label{skel_out_lip}
		d(\iota(p),\iota(p'))\leq 5 K^2 \beta \|p-p'\|_{\infty}
	\end{equation}
	for each $p,p'\in D'$ with $\|p-p'\|_\infty \leq l^{i}$.
\end{lemma}

\begin{proof}
	Let $Q\in \mathcal{G}$ and $F\in \mathcal{F}(m+1,i)$ with $F \subset Q$.
	By \eqref{skel_ass}, $\iota$ is defined on $\partial F \cap \mathcal{D}(i+1)$.
	By \eqref{C_near_ass} and the fact that $\beta< 1/(2l^i)$, the hypotheses of \Cref{basic_ext} are satisfied.
	An application of \Cref{basic_ext} extends $\iota\vert_{\partial F \cap \mathcal{D}(i+1)}$ to a map
	\begin{equation*}
		\iota \colon F \cap \mathcal{D}(i+1) \to C
	\end{equation*}
	such that
	\begin{equation}\label{local_control}
		d(\iota(p), \iota(p')) \leq 2K^2 \beta \|p-p'\|_{\infty}
	\end{equation}
	for each $p,p'\in F \cap \mathcal{D}(i+1)$.
	We combine each of these extensions to form an extension of $\iota$ to a map
	\begin{equation*}
		\iota \colon D' \to C,
	\end{equation*}
	ignoring any repetitions of faces that appear in multiple cubes.

	Now let $p,p' \in D'$ with $\|p-p'\|_{\infty} \leq l^i$.
        If $p,p'\in D$ then \eqref{skel_out_lip} holds for $p,p'$ by hypothesis.
	Therefore, we may suppose that $p \not\in D$.
        Let $Q\in \mathcal{G}$ with $p\in Q$ and $F\in \mathcal{F}(m+1,i)$ with $p\in F\subset \skel(Q,m+1)$.
        If $p'\in F$ then \eqref{local_control} implies \eqref{skel_out_lip} for $p,p'$.
        Otherwise, let $q\in \partial F\cap \mathcal{D}(i+1)$ be given by \Cref{observation} such that
        \begin{equation*}
          \max\{\|p-q\|_{\infty}, \|q-p'\|_{\infty}\} \leq \|p-p'\|_{\infty}
        \end{equation*}
        If $p'\not\in D$, let $Q'\in \mathcal G$ and $F'\in \mathcal{F}(m+1,i)$ with $p'\in F' \subset \skel(Q',m+1)$ and let $q'\in \partial F' \cap \mathcal{D}(i+1)$ be given by \Cref{observation} such that
        \begin{equation*}
          \max\{\|p'-q'\|_{\infty}, \|q-q'\|_{\infty}\} \leq \|p'-q\|_{\infty}.
        \end{equation*}
        If $p' \in D$ let $q'=p'$.
        In either case, we have $\|q-q'\|_{\infty},\|p-p'\|_\infty \leq l^{i}$ and so \eqref{skel_input_lip} implies
	\begin{equation*}
		d(\iota(q),\iota(q')) \leq \beta \|q-q'\|_{\infty}.
         \end{equation*}
        Further, \eqref{local_control} implies
         \begin{equation*}
           d(\iota(p), \iota(q)) \leq 2K^{2}\beta\|p-q\|_{\infty}
         \end{equation*}
        and
         \begin{equation*}
           d(\iota(q'), \iota(p')) \leq 2K^{2}\beta\|q'-p'\|_{\infty}.
         \end{equation*}
        Therefore, the triangle inequality gives
	\begin{align*}
		d(\iota(p), \iota(p')) &\leq 2K^2 \beta \|p-q\|_{\infty} + \beta \|q-q'\|_{\infty} \\
		&\quad + 2K^2 \beta \|q'-p'\|_{\infty} \\
		&\leq 5 K^2\beta \|p-p'\|_{\infty},
	\end{align*}
	as required.
\end{proof}

Recall $\sigma=(5K^2)^n$ is defined in \Cref{notation}.
We now iterate the construction in \Cref{skel_iter} to obtain an extension to $n$-dimensional cubes.
\begin{lemma}\label{skel_ext}
	Fix $i\in \mathbb{N}$, let $D\subset \mathcal{D}(i)$ and let
	\begin{equation}\label{beta_req}
		0<\beta < \frac{1}{2\sigma l^{i}}.
	\end{equation}
	Suppose that $\iota\colon D\to C$ satisfies
	\begin{equation}\label{skel_ext_input_lip}
		d(\iota(p),\iota(p'))\leq \beta \|p-p'\|_{\infty}
	\end{equation}
	for each $p,p'\in D$ with $\|p-p'\|_\infty \leq l^i$.
	Suppose also that $\mathcal{G}\subset \mathcal{Q}(i)$ is such that $\cor(\mathcal{G}) \subset D$ and
	\begin{equation}\label{C_near_ext_ass}
		\forall Q \in \mathcal{G},\ \exists x \in G \text{ with } \iota(D\cap Q) \subset B(x, 2\beta l^{i}).
	\end{equation}
	Then there exists an extension of $\iota$ to
	\begin{equation*}
		D' := \left( \bigcup \mathcal{G} \cap \mathcal{D}(i+1) \right) \cup D = \cor(\child(\mathcal{G})) \cup D
	\end{equation*}
	such that
	\begin{equation}\label{skel_ext_lip}
		d(\iota(p),\iota(p'))\leq \sigma \beta \|p-p'\|_{\infty}
	\end{equation}
	for each $p,p'\in D'$ with $\|p-p'\|_\infty \leq l^{i}$.
\end{lemma}

\begin{proof}
	First note that, since $\cor(\mathcal{G})\subset D$, \eqref{skel_ass} holds for $m=0$.
	Since \eqref{C_near_ext_ass} implies \eqref{C_near_ass}, we may apply \Cref{skel_iter} to obtain an extension of $\iota$ to
	\begin{equation*}
		D_1 := D \cup \bigcup_{Q\in \mathcal{G}} \skel(Q,1) \cap \mathcal{D}(i+1)
	\end{equation*}
	such that
	\begin{equation*}
		d(\iota(p),\iota(p')) \leq 5K^2 \beta\|p-p'\|_\infty
	\end{equation*}
	for each $p,p'\in D_1$ with $\|p-p'\|_\infty \leq l^i$.

	Now notice that,
	\begin{equation*}
		\forall Q \in \mathcal{G},\ D_1 \cap Q = \skel(Q,1).
	\end{equation*}
	Since
	\[(5K^2)\beta < \frac{(5K^2)}{2\sigma l^i} \leq \frac{1}{2l^i},\]
	we may apply \Cref{skel_iter} again, but with $\beta$ replaced by $(5K^2)\beta$ and with $m=1$.
	This gives an extension of $\iota$ to
	\begin{equation*}
		D_2 := D \cup \bigcup_{Q\in \mathcal{G}} \skel(Q,2) \cap \mathcal{D}(i+1)
	\end{equation*}
	such that
	\begin{equation*}
		d(\iota(p),\iota(p')) \leq (5K^2)^2 \beta\|p-p'\|_\infty
	\end{equation*}
	for each $p,p'\in D_2$ with $\|p-p'\|_\infty \leq l^i$.
	Thus, the hypotheses of \Cref{skel_iter} are satisfied again, with $\beta$ replaced by $(5K^2)^2\beta$ and with $m=2$.

	Since $\beta < 1/(2(5K^2)^nl^i)$, we repeat this for a total of $n$ times, extending $\iota$ to
	\begin{equation*}
		D_n := D \cup \bigcup_{Q\in \mathcal{G}} \skel(Q,n) \cap \mathcal{D}(i+1) = D \cup \left( \bigcup \mathcal{G} \cap \mathcal{D}(i+1) \right) = D'
	\end{equation*}
	such that
	\begin{equation*}
		d(\iota(p),\iota(p')) \leq (5K^2)^n \beta\|p-p'\|_\infty = \sigma \beta \|p-p'\|_\infty
	\end{equation*}
	for each $p,p'\in D'$ with $\|p-p'\|_\infty \leq l^i$.
\end{proof}

\subsection{Constructing the extension to all scales}
In \Cref{skel_ext} we demonstrated how to extend a function from some $D\subset \mathcal D(i)$ to $D'\subset\mathcal D(i+1)$.
We now iterate this extension over all $i\in\N$.

Note that, in \cref{iter_Lip,iter_B} of the following lemma, the radius of the ball in \cref{iter_B} and the upper bound in \cref{iter_Lip} are of the form $(\sigma l)^m$, up to some constant multiple.
In \eqref{def_N} we chose $\alpha$ so that $(\sigma l) = l^{\alpha}$, so both of these quantities are comparable to $l^{\alpha m}$.
In particular, \cref{iter_Lip} implies that each $\iota_i$ is $\alpha$-H\"older with the same H\"older constant.
The fact that these two quantities have the same power of $\alpha$ is an essential requirement for the construction given in this section to work;
looking ahead to \Cref{prop:surface}, this is exactly the condition required so that we can bound $\H^n_\infty(\iota([0,1]^n)\setminus C)$ in \eqref{first_cont_bound}.
Looking behind, we see that the two powers agree because of \Cref{choice_of_good}.
\begin{lemma}
	\label{iter_funct}
	Let $M\in \mathbb{N}$ with $M > \alpha/2$ , $0\leq L\leq \sigma^{M}$ and suppose that
	\[\iota_M \colon \mathcal{D}(M) \to C\]
	is $L$-Lipschitz.
	For each $i \geq M$ there exist $\mathcal G_i,\mathcal B_i\subset \mathcal Q(i)$ such that
	\begin{equation}
		\label{decomp_cube2}
		[0,1]^n= \bigcup \mathcal G_i \cup \bigcup_{M\leq j\leq i} \bigcup \mathcal B_i
	\end{equation}
	and a function
	\begin{equation}\label{dom_iota}
		\iota_{i} \colon \cor\left(\mathcal{G}_{i}\cup \bigcup_{M\leq j \leq i} \mathcal{B}_j\right) \to C
	\end{equation}
	such that:
	\begin{enumerate}
		\item \label{iter_ext}$\iota_{i+1}$ is an extension of $\iota_{i}$;
		\item \label{iter_child} $\child(\mathcal G_i) = \mathcal G_{i+1}\cup \mathcal B_{i+1}$
		\item \label{iter_G}For each $Q\in\mathcal G_i$, there exists $x\in G$ with
		\[\iota_i(\cor(Q)) \subset B(x,2L\sigma^{i-M} l^i);\]
		\item \label{iter_B} For all $p\in \cor(\mathcal B_{i})$,
		\begin{equation*}
			B(\iota_{i}(p),L \sigma^{i-M} l^{i}) \cap G =\emptyset,
		\end{equation*}
		\item \label{iter_Lip} For each $p,p'\in \dom \iota_{i}$ with $\|p-p'\|_\infty \leq l^{i-1}$,
	\begin{equation*}
		d(\iota_{i}(p), \iota_{i}(p')) \leq L\sigma^{i-M} \|p-p'\|_\infty.
	\end{equation*}
	Here and throughout, $\dom \iota_{i}$ denotes the domain of $\iota_i$.
	\end{enumerate}
\end{lemma}

\begin{proof}
	We prove the lemma by induction.
	To begin, for $\beta:=L$, let
	\[\mathcal Q(M)=\mathcal G_M\cup \mathcal B_M\]
	be given by \Cref{choice_of_good} for $\iota=\iota_M$ and $i=M$.
	Then $\iota_M,\mathcal{G}_M,\mathcal B_M$ satisfy the required conditions.
	
	Now suppose that $\iota_{i}, \mathcal G_{i}$ and $\mathcal B_M,\ldots,\mathcal B_{i}$ exist for some $i \geq M$.
	We first apply \Cref{skel_ext} to $\iota_{i}$ and $\mathcal G_i$ and begin by checking that the required hypotheses are satisfied.
	Since $L\leq \sigma^{M}$, \eqref{beta_i_bound} implies
	\[\beta := L\sigma^{i-M} \leq \sigma^{i} < \frac{1}{2\sigma l^{i}},\]
	and so \eqref{beta_req} holds.
	By induction hypothesis, \cref{iter_Lip} implies \eqref{skel_ext_input_lip} for this choice of $\beta$.
	Also, \cref{iter_G} implies $\mathcal G_i$ satisfies \eqref{C_near_ext_ass}.
	An application of \Cref{skel_ext} constructs an extension of $\iota_{i}$ to
	\begin{equation*}
		\iota_{i+1} \colon \cor\left(\child(\mathcal{G}_{i}) \cup \bigcup_{M\leq j\leq i} \mathcal{B}_j\right) \to C
	\end{equation*}
	such that, for each $p,p'\in \dom\iota_{i+1}$ with $\|p-p'\|_\infty \leq l^{i}$,
	\begin{equation}\label{local_bound2}
		d(\iota_{i+1}(p),\iota_{i+1}(p')) \leq \sigma \beta \|p-p'\|_\infty = L\sigma^{i+1-M} \|p-p'\|_\infty.
	\end{equation}

	\Cref{local_bound2} implies that $\mathcal Q=\child(\mathcal G_i)$ satisfies the hypothesis of \Cref{choice_of_good} with $\beta := L \sigma^{i+1-M}$.
	Let
	\begin{equation}\label{this_decomp_cube}\child(\mathcal G_{i})=\mathcal G_{i+1}\cup \mathcal B_{i+1}\end{equation}
	be the decomposition given by \Cref{choice_of_good}, so that \cref{iter_child} is satisfied and so that $\mathcal G_{i+1}$ satisfies \cref{iter_G} and $\mathcal B_{i+1}$ satisfies \cref{iter_B} for $i=i+1$.
	Note that \eqref{decomp_cube2} for $i$ combined with \eqref{this_decomp_cube} implies \eqref{decomp_cube2} for $i+1$.
	Thus the required $\iota_{i+1},\mathcal G_{i+1}$ and $\mathcal B_{i+1}$ exist and so the lemma is true by induction.
\end{proof}

We now prove, by induction on $i\in \N$, a local Lipschitz bound on each $\iota_i$ (with a Lipschitz constant that depends on $i$). Afterwards, this will be upgraded to uniform H\"older continuity.

\begin{lemma}
	\label{bound_over_scales}
	Let $M\in \mathbb{N}$ with $M > \alpha/2$ , $0\leq L\leq \sigma^{M}$ and suppose that
	\[\iota_M \colon \mathcal{D}(M) \to C\]
	is $L$-Lipschitz.
	For each $i\geq M$ let $\iota_i,\mathcal G_i$ and $\mathcal B_i$ be obtained from \Cref{iter_funct}.
	Then for any $i\geq M$, $j\in \N$ and $p, q\in \dom \iota_i$ with $p,q\in Q\in \mathcal Q(j)$,
	\begin{equation}\label{ancestor}
		d(\iota_i(p),\iota_i(q)) \leq 2L\sigma^{-M+1} (\sigma l)^j.
	\end{equation}
\end{lemma}

\begin{proof}
First note that, if $i=j$, \eqref{ancestor} follows from \Cref{iter_funct} \cref{iter_Lip} and if $j>i$ then $p=q$, so that \eqref{ancestor} holds.
To prove \eqref{ancestor} for the case $i>j\geq M$, we will prove that
\begin{equation}
	\label{anc_iter}
	d(\iota_i(p),\iota_i(q)) \leq L\sigma^{-M+1}\sum_{k=j}^{i-1} (l\sigma)^k \quad \forall p,q\in\dom\iota_i,\ p,q\in Q\in\mathcal Q(j)
\end{equation}
holds for all $i>M$ and all $i>j\geq M$ by induction on $i$.
\Cref{ancestor} then follows because the geometric series converges by \eqref{beta_i_bound}.
	
	Note that \eqref{anc_iter} holds for $i=M+1$ and $j=M$ since $\iota_M$ is $L$-Lipschitz.
	Now let $i> M+1$, and suppose that \eqref{anc_iter} holds for $i-1$ and all $i-1> j\geq M$.
	Let $i>j\geq M$ and $p, q\in\dom\iota_i$ with $p,q\in Q\in \mathcal Q(j)$.
	If neither of $p$ or $q$ belong to $\cor(\mathcal G_{i}\cup \mathcal B_i)$, then they both belong to $\dom\iota_{i-1}$.
	If $j=i-1$ then \Cref{iter_funct} \cref{iter_Lip} for $i-1$ implies \eqref{anc_iter} for $i$ since $\iota_{i}$ extends $\iota_{i-1}$.
	If $j<i-1$ then the induction hypothesis implies \eqref{anc_iter} for $i$.
	Therefore, we may suppose $p \in \cor(Q^*)$ for some $Q^* \in \mathcal G_{i}\cup \mathcal B_i$.
	
	Note that \Cref{iter_funct} \cref{iter_child} implies that there exists $Q'\in \mathcal G_{i-1}$ with $Q^*\subset Q'$.
	Let $p'\in \cor(Q')\subset \dom\iota_{i-1}$, so that $\|p-p'\|_\infty\leq l^{i-1}$.
	By \Cref{iter_funct} \cref{iter_Lip}
	\begin{equation}
		\label{one_scale}
		d(\iota_i(p),\iota_i(p')) \leq L\sigma^{i-M} l^{i-1} = L\sigma^{-M+1}(\sigma l)^{i-1}.
	\end{equation}
	Note also that, since $Q'\in \mathcal Q(i-1)$ and $Q\in \mathcal Q(j)$ with $j\leq i-1$, we must have $Q'\subset Q$.
	Thus, since $p',q\in Q$, the induction hypothesis implies
	\begin{equation}
		\label{many_scales_Q}
		d(\iota_{i-1}(p'),\iota_{i-1}(q)) \leq L \sigma^{-M+1} \sum_{k=j}^{i-2}(l\sigma)^k.
	\end{equation}
	Using the fact that $\iota_i$ extends $\iota_{i-1}$, we may combine \cref{one_scale,many_scales_Q} using the triangle inequality to get
	\[d(\iota_i(p),\iota_i(q)) \leq L\sigma^{-M+1} \sum_{k=j}^{i-1}(l\sigma)^k,\]
	completing the induction step.
\end{proof}

Recall from \eqref{sigmal_obs} that $\sigma l =l^\alpha$.
We now show that the limit of the $\iota_i$ is $\alpha$-H\"older. As mentioned before \Cref{iter_funct}, it is essential that the exponent of $l^i$ appearing in \cref{far_from_G} in the following lemma equals $\alpha$ in order to deduce \eqref{first_cont_bound}.
\begin{lemma}
	\label{scale_iter}
	Let $M\in \mathbb{N}$ with $M > 2/\alpha$, $0 \leq L \leq \sigma^{M}$ and let $\iota \colon \mathcal{D}(M) \to C$ be $L$-Lipschitz.
	There exist a closed set $S \subset [0,1]^n$ with $S\supset \mathcal{D}(M)$, an extension of $\iota$ to
	\begin{equation*}
		\iota\colon (S,\|\cdot\|_\infty)\to C\cap B(\iota(0),5L)
	\end{equation*}
	and, for each $i \geq M$, a set $\mathcal{B}_i \subset \mathcal{Q}(i)$ such that:
	\begin{enumerate}
		\item \label{is_holder} $\iota$ is $\alpha$-H\"older continuous with H\"older constant $5L l^{-\alpha}$.
		\item \label{cube_decomposition}\[[0,1]^n = S \cup \bigcup_{i \geq M} \bigcup \mathcal{B}_i\]
		\item \label{far_from_G} For each $i \geq M$ and each $p\in \cor(\mathcal{B}_i)$, $p\in S$ and
		\begin{equation*}
			B(\iota(p), L \sigma^{-M} l^{\alpha i})\cap G=\emptyset.
		\end{equation*}
	\end{enumerate}
\end{lemma}
	
\begin{proof}
	Set $\iota_M=\iota$ and for each $i\geq M$ let $\mathcal G_i,\mathcal B_{i}$ and $\iota_i$ be obtained from \Cref{iter_funct}.
	Set $S' = \cup_{i\geq M} \dom\iota_i$ and define $\iota$ on $S'$ by $\iota(x)=\iota_i(x)$ whenever $x\in \dom\iota_i$.
	Let $p,p'\in S'$, let $i\in\N$ be sufficiently large such that $p,p'\in\dom\iota_i$ and let $m\in \mathbb{N}$ be such that $l^{m+1} < \|p-p'\|_\infty \leq l^m$.

	If $m<M$, let $Q,Q'\in \mathcal{Q}(M)$ be such that $p\in Q$ and $p'\in Q'$ and let $q\in \cor(Q)$ and $q'\in \cor(Q')$ minimise the distance between $\cor(Q)$ and $\cor(Q')$, so that $\|q-q'\|_\infty \leq l^m$.
	Then by combining \eqref{ancestor}, the fact that $\iota$ is $L$-Lipschitz, and using the triangle inequality
	\begin{align}
	 d(\iota(p),\iota(p')) &\leq d(\iota_i(p),\iota_i(q)) + d(\iota(q),\iota(q')) + \d(\iota_{i}(q'),\iota_{i}(p'))\notag\\
	 &\leq 2L \sigma^{-M+1}(\sigma l)^M +Ll^m + 2L \sigma^{-M+1}(\sigma l)^M \notag\\
	 &\leq 4L\sigma  l^M + Ll^m\notag\\
	 &\leq 5L(\sigma l)^m.\label{holder_1}
	\end{align}
	
	If $m\geq M$, let $p\in Q \in \mathcal{Q}(m)$ and $p'\in Q'\in \mathcal{Q}(m)$.
	Since $\|p-p'\|_\infty \leq l^m$, $Q$ and $Q'$ have a common corner $q$ and hence $\|p-q\|_\infty, \|p'-q\|_\infty \leq l^m$.
	Thus the triangle inequality and \eqref{ancestor} give
	\begin{align}
		d(\iota(p),\iota(p')) &\leq d(\iota_{i}(p),\iota_i(q)) + d(\iota_{i}(q),\iota_{i}(p'))\notag\\
		&\leq 2L \sigma^{-M+1}(\sigma l)^m +2L \sigma^{-M+1}(\sigma l)^m\notag\\
		&\leq 4 L (\sigma l)^m,\label{holder_2}
	\end{align}
	using the fact that $M\geq 2$ in the final inequality.

	\Cref{holder_1,holder_2,sigmal_obs} show that
	\begin{equation}
		d(\iota(p),\iota(p')) \leq 5L l^{\alpha m}\leq 5L l^{-\alpha} \|p-p'\|_\infty^\alpha.
	\end{equation}
	That is, $\iota$ is $\alpha$-H\"older with constant $5Ll^{-\alpha}$.
	This also implies that $\iota$ maps into $B(\iota(0),5L)$.

	Let $S$ be the closure of $S'$ and extend $\iota$ to $S$ maintaining the H\"older bound.
	Since $C$ is closed, this extension maps into $C \cap B(\iota(0),5L)$.
	For each $i \geq M$, \eqref{decomp_cube2} implies
	\begin{equation}\label{decomp_cube}
		[0,1]^n = \bigcup \mathcal{G}_i \cup \bigcup_{j \geq M} \bigcup \mathcal{B}_j.
	\end{equation}
	For each $i\in \mathbb{N}$, $\cor(\mathcal{G}_i) \subset \dom\iota_i \subset S$ and so $\cup \mathcal{G}_i \subset B(S,l^i)$.
	Since $S$ is compact this implies
	\begin{equation*}
		\bigcap_{i\geq M} \bigcup \mathcal{G}_i \subset S.
	\end{equation*}
	Thus, intersecting \eqref{decomp_cube} over all $i\geq M$ gives \cref{cube_decomposition}.
			
	Finally, \cref{sigmal_obs} and \Cref{iter_funct} \cref{iter_B} and imply \cref{far_from_G}.
	\end{proof}

We now extend the map $\iota$ to $[0,1]^n$, assuming $X$ is a subset of $\ell_\infty$.
This will implicitly construct the metric space $\tilde E$ mentioned in the introduction.
Since we have constructed $\iota$ to be H\"older continuous, this can be done using \Cref{holder-extension}.
The H\"older condition is also used (in combination with \eqref{GTA_density}) to deduce \eqref{first_cont_bound} from \Cref{scale_iter} \cref{far_from_G}.
	\begin{proposition}\label{prop:surface}
	Let $M\in \mathbb{N}$ with $M>2/\alpha$, $0\leq L\leq \sigma^{M}$
	and let $\iota \colon \mathcal{D}(M) \to C$ be $L$-Lipschitz.
	Suppose that $X$ is a subset of $\ell_\infty$.
    There exists an $\alpha$-H\"older extension $\iota\colon [0,1]^n \to B(\iota(0),10L)$ with H\"older constant $5L l^{-\alpha}$
	such that
	\begin{equation}\label{first_cont_bound}\H^n_\infty(\iota([0,1]^n) \setminus C) \leq \frac{(50 \sigma^M l^{-\alpha})^n}{\eta} \H^n(B(\iota(0),10 L) \cap X\setminus G).\end{equation}
	\end{proposition}

	\begin{proof}
		First apply \Cref{scale_iter} to find a closed $S\subset [0,1]^n$ and extend $\iota$ to an $\alpha$-H\"older map
		\begin{equation}\label{def_before_ext}
		\iota \colon S \to B(\iota(0),5L) \cap C.
		\end{equation}
		Then apply \Cref{holder-extension} to extend $\iota$ to an $\alpha$-H\"older map
		\begin{equation*}
		\iota \colon [0,1]^n \to \ell_\infty
		\end{equation*}
		maintaining the H\"older constant of $5L l^{-\alpha}$.
		In particular, since $\mathcal{D}(M)\subset S$ and $M>2/\alpha$, this extension maps into $B(\iota(0),10L)$.
		
		Since $\iota(S)\subset C$, we deduce \eqref{first_cont_bound} by estimating the Hausdorff content of
		\[Y:=\iota\left(\bigcup_{i \geq M} \bigcup \mathcal{B}_i\right),\]
		where the $\mathcal B_i$ are as in \Cref{scale_iter} \cref{iter_B}.
		That is, for each $i \geq M$, $\mathcal B_i \subset \mathcal Q(i)$ and, for each $p\in\cor(\mathcal B_i)$, $p\in S$ and
		\begin{equation}\label{outside_G}
			B(\iota(p), L\sigma^{-M}l^{\alpha i})\cap G=\emptyset.
		\end{equation}

		Since $\iota$ is $\alpha$-H\"older with constant $5L l^{-\alpha}$,
		\begin{equation*}
			Y \subset \{ B(\iota(p), 5L l^{-\alpha}l^{i\alpha}) : i\geq M,\ p\in \cor(\mathcal B_i)\}.
		\end{equation*}
		That is, if
		\begin{equation*}
          \mathcal A=\{B(\iota(p),L\sigma^{-M}l^{\alpha i}): i \geq M,\ p\in \cor(\mathcal{B}_i)\},
		\end{equation*}
		then
		\begin{equation*}
			Y \subset \bigcup_{B\in \mathcal A} 5 \sigma^M l^{-\alpha} B.
		\end{equation*}
		Let $\mathcal A'$ be a disjoint $5r$-sub-cover of $\mathcal A$, so that the elements of $\mathcal A'$ are pairwise disjoint and
		\begin{equation*}\label{eq:cover-Y}
		Y \subset \bigcup_{B \in \mathcal A'} 25 \sigma^M l^{-\alpha} B.
		\end{equation*}
		In particular,
		\begin{equation}\label{hd_cont_est}
			\mathcal{H}^n_\infty(Y) \leq \sum_{B\in \mathcal A'} (50 \sigma^M l^{-\alpha} \rad B)^n
		\end{equation}
		By \cref{def_before_ext,outside_G},
		\begin{equation}\label{outside-g}
		\bigcup \mathcal A' \subset B(\iota(0),10L) \setminus G.
		\end{equation}
		For each $i \geq M$ and $p\in \cor(\mathcal{B}_i)$, $\iota(p)\in C$ and so \Cref{GTA} \cref{GTA1} implies
		\[\H^n(B(\iota(p),r) \cap X) \geq \eta r^n\]
		whenever $r \leq 1$.
		In particular, this is true for all balls in $\mathcal A'$.
		Therefore, by \eqref{outside-g} and the disjointness of the balls in $\mathcal A'$,
		\begin{align*}
		\eta \sum_{B\in \mathcal A'} \rad(B)^n &\leq \sum_{B\in \mathcal A'} \mathcal{H}^n(B \cap X)\\
		&\leq \mathcal{H}^n(B(\iota(0),10L) \cap X \setminus G).
		\end{align*}
		Combining this with \eqref{hd_cont_est} gives
		\[\H^n_\infty(Y) \leq \frac{(50 \sigma^M l^{-\alpha})^n}{\eta} \H^n(B(\iota(0)),10 L) \cap X\setminus G),\]
		as required.
	\end{proof}

	To conclude the section, we record the main conclusion without the fixed quantities of \Cref{notation}.
	Recall, for $m\in\mathbb N$ and $R>0$, that
	\[\mathcal D(R,m) = \{R2^{-m}(j_1,\ldots,j_n) \in [0,R]^n : j_1,\ldots, j_n \in\mathbb Z \}.\]
	\begin{theorem}
		\label{holder_construction}
		For any $K \geq 1$, $0<\gamma<1$ and $\eta>0$ there exists $N\in \mathbb N$ and $\gamma\leq \alpha<1$ such that the following is true.
	    For any closed $X\subset \ell_\infty$ and for any $R_0 >0$, let $(X,C,G)$ satisfy
		\begin{equation}\label{gta_N}GTA(\eta,K,2^{-N}/20,R_0).\end{equation}
		Then for any $M\in \mathbb{N}$ with $M>2/\alpha$, $0< r<R_0$, $0 \leq L \leq (5K^2)^{M}$ and any $L$-Lipschitz
		\[\iota \colon \mathcal D(r,MN)\subset \ell_\infty^n \to C,\]
		there exists an $\alpha$-H\"older extension
		\[\iota\colon [0,r]^n\subset \ell_\infty^n \to B(\iota(0),10L r)\]
		with H\"older constant $5L 2^{N\alpha}r^{1-\alpha}$ such that
			\begin{equation}\label{h_con_concl}\H^n_\infty(\iota([0,r]^n) \setminus C) \leq \Lambda \H^n(B(\iota(0),10L r) \cap X\setminus G),\end{equation}
		for some $\Lambda>0$ depending only upon $K,\gamma,M,n,\eta$.
	\end{theorem}

	\begin{proof}
		Given $K,\gamma,\eta$ as in the hypotheses, let $N\in\mathbb N$ and $\gamma\leq \alpha<1$ be given by \eqref{def_N}.
		For any $(X,C,G)$ satisfying \eqref{gta_N} and any $0<r<R_0$, let $C'= C/r$ and $G' = G/r$, considered as subsets of $X'=X/r \subset \ell_\infty$.
		Then $(X',C',G')$ satisfies $GTA(\eta,K,2^{-N}/20,1)$.
		Define $\iota' \colon \mathcal D(1,MN) \to C'$ by $\iota'(q) = \iota(rq)/r$.
		Then $\iota'$ is also $L$-Lipschitz.
		An application of \Cref{prop:surface} gives an $\alpha$-H\"older extension of $\iota'$ with H\"older constant $5L2^{N\alpha}$ satisfying
		\begin{equation}\label{h_con_conc2}\H^n_\infty(\iota'([0,1]^n) \setminus C') \leq \Lambda \H^n(B(\iota(0),10L) \cap X'\setminus G')\end{equation}
		for $\Lambda = (20 (5K)^{2M} 2^{\alpha N})^n/\eta$.
		This gives an $\alpha$-H\"older extension of $\iota$ defined by $\iota(q)=r\iota'(q/r)$ with H\"older constant $5L2^{N\alpha} r^{1-\alpha}$:
		\begin{align*}
			\|\iota(y)-\iota(z)\|_\infty &= r\|\iota'(y/r) - \iota'(z/r)\|_\infty\\
			&\leq r 5L2^{N\alpha}\|y/r-z/r\|_\infty^\alpha\\
			&= 5L2^{N\alpha} r^{1-\alpha}\|y-z\|_\infty^\alpha.
		\end{align*}
		\Cref{h_con_conc2} is equivalent to \eqref{h_con_concl}.
	\end{proof}
	
	Finally, we conclude the proof of \Cref{main_construction}

	\begin{proof}[Proof of \Cref{main_construction}]
		The hypotheses of \Cref{main_construction} imply the hypotheses of \Cref{holder_construction}.
		Let $N,\alpha$ be obtained from \Cref{holder_construction}
		and let $m_0=2N$.
		If $(X,C,G)$ satisfies \eqref{gta_MN} and $x\in G$ then by \Cref{bi-lip_param}, there exists a $(K+2^{-m})$-bi-Lipschitz $\xi\colon \mathcal D(r,m)\to C$ with $\xi(0)=x$.
		Since $(X,C,G)$ also satisfies \eqref{gta_N}, \Cref{holder_construction} with $M=2$ gives an $\alpha$-H\"older extension $\iota \colon [0,r]^n\to B(x,20Kr)$ with H\"older constant $10K2^{N\alpha}r^{1-\alpha}$ satisfying \cref{holder_bilip,holder_full_meas} of \Cref{main_construction}.
		Since $\gamma\leq \alpha<1$, $\iota$ is also $\gamma$-H\"older.
		Moreover, if $y,z\in [0,r]^n$ with $\|y-z\|_\infty \leq 2^{-m}r$ then
		\begin{align*}
		\|\iota(y)-\iota(z)\|_\infty &\leq 10K2^{N\alpha}r^{1-\alpha} \|y-z\|_\infty^\alpha\\ &\leq 10K2^{(N-m)\alpha} r^{1-\alpha+\alpha}\\
		&\leq 10K2^{-m\alpha/2}r\\
		&\leq 10K 2^{-m\gamma/2}r,
		\end{align*}
		giving \cref{holder_perturb}.
	\end{proof}
	
	\begin{remark}
		If, in the hypotheses of \Cref{holder_construction} or \Cref{main_construction}, $X$ is assumed to be a subset of a finite dimensional Banach space $(\mathbb{R}^m,\|\cdot\|)$, then the H\"older map $\iota$ can be constructed with values in $(\mathbb{R}^m,\|\cdot\|)$.
		One simply uses \Cref{holder-extension} to extend the map into $(\mathbb{R}^m,\|\cdot\|)$, increasing the H\"older constant by a factor that only depends on the norm $\|\cdot\|$.
	\end{remark}
	
	\begin{remark}
		At this point it would be reasonable to prove \Cref{main_marstrand} using \Cref{main_construction} and \Cref{useful_corollary}.
		However, there are various technical steps that are simply easier to deal with using $\dGHs$.
		Thus we delay the proof until after the proof of \Cref{main_thm}.
	\end{remark}

\section{Metrics between metric measure spaces}\label{sec:convergence}

\subsection{The pointed measured Gromov--Hausdorff metric}

Recall the definition of $\dG_{a,b}$ given in \Cref{dG}.
We use the following metric to describe pointed measured Gromov--Hausdorff convergence.
\begin{definition}
	\label{pmGH}
	For pointed metric measure spaces $(X,\mu,x)$ and $(Y,\nu,y)$ define
	\[\dpmGH((X,\mu,x),(Y,\nu,y))=\dG_{1,1}((X,\mu,x),(Y,\nu,y)).\]
\end{definition}

Note that $\dpmGH$ is bounded below by $\dpGH$ of the corresponding pointed metric spaces.

\begin{remark}
	If one is only interested in metrising pointed measured Gromov--Hausdorff convergence, it is more natural to use $\dKR_z^1$ rather than $\dKR_z$ in the definition of $\dpmGH$.
	Doing so would not change the results of this subsection, except for minor superficial changes in \Cref{eps_isom_pmGH}.
	We use $\dKR_z$ to give the desired relationship to $\dGHs$ in the next subsection.	
\end{remark}

We have the following relationship between $\epsilon$-isometries and $\dKR$.

\begin{lemma}
	\label{eps_isom_pushforward}
	Let $(Z,\zeta,z)$ be a pointed metric measure space $X,Y\subset Z$ Borel sets, $\mu\in \Mloc(X)$ and $\nu\in\Mloc(Y)$.
	For $0<\epsilon<1/2$ suppose that $f\colon X \cap B(z,1/\epsilon)\to Y$ is a Borel function that satisfies $\zeta(f(x),x)\leq\epsilon$ for all $x\in X \cap B(z,1/\epsilon)$.
	Then for any $0<r \leq 1/\epsilon-\epsilon$ and $L>0$,
	\begin{equation*}
	\vert\dKR^{L,r}_z(f_{\#}\mu,\nu) -\dKR^{L,r}_z(\mu,\nu)\vert \leq \epsilon L \mu(B(z,r+\epsilon)).	
	\end{equation*}
\end{lemma}
\begin{proof}
	Let $g\colon Z \to [-1,1]$ be $L$-Lipschitz with $\spt g \subset B(z,r)$.
	For any $x\in B(z,1/\epsilon)$,
	\begin{equation*}
		\vert g\circ f(x)-g(x)\vert \leq L \zeta(f(x),x) \leq L\epsilon
	\end{equation*}
	and so, since $r\leq 1/\epsilon -\epsilon$,
	\begin{equation*}
		\left\vert\int_{B(z,r+\epsilon)} g\circ f \d\mu - \int_{B(z,r+\epsilon)} g\d\mu\right \vert \leq L\epsilon\mu(B(z,r+\epsilon)).
	\end{equation*}
	However, $x\in X\cap B(z,1/\epsilon)\setminus B(z,r+\epsilon)$ then $x,f(x)\not\in \spt g$ and so
	\[\left\vert\int g\circ f \d\mu - \int g\d\mu\right \vert \leq L\epsilon\mu(B(z,r+\epsilon)).\]
	Therefore,
	\begin{align*}
		\int g\d(f_{\#}\mu - \nu) &= \int g\circ f \d\mu - \int g \d\nu \\
		&\leq L\epsilon\mu(B(z,r+\epsilon) + \int g \d(\mu-\nu).
	\end{align*}
	Taking the supremum over all such $g$ gives
	\[\dKR^{L,r}_z(f_{\#}\mu,\nu) \leq L\epsilon\mu(B(z,r+\epsilon)) + \dKR^{L,r}_z(\mu,\nu).\]
	Similarly,
	\[\dKR^{L,r}_z(\mu,\nu) \leq L\epsilon\mu(B(z,r+\epsilon)) + \dKR^{L,r}_z(f_{\#}\mu,\nu),\]
	as required.
\end{proof}

\begin{corollary}
\label{eps_isom_pmGH}
Let $(X,\mu,x)$ and $(Y,\rho,y)$ be pointed metric measure spaces and $0<\epsilon<1/2$.
If $\dpmGH((X,\mu,x),(Y,\nu,y))<\epsilon$, there exists a $2\epsilon$-isometry $f\colon(X,x) \to (Y,y)$ such that, for any $0<r \leq 1/\epsilon-\epsilon$ and $L>0$,
\begin{equation}
\label{close_equiv_push}
\dKR^{L,r}_y(f_{\#}\mu,\nu) \leq \epsilon+\epsilon L\mu(B(x,r+\epsilon)).
\end{equation}

Conversely, if there exists an $\epsilon$-isometry $f\colon (X,\mu,x) \to (Y,\nu,y)$ then
\[\dpmGH((X,\mu,x),(Y,\nu,y))\]
is bounded above by the maximum of $2\epsilon$ and
\[\inf\{1/r+ \dKR^{r,r}_y(f_{\#}\mu,\nu) + \epsilon r\mu(B(x,r+\epsilon)): 0<r\leq 2\epsilon\}.\]
\end{corollary}

\begin{proof}
Let $(Z,\zeta,z)$ be a pointed metric space for which there exist isometric embeddings $(X,x),(Y,y) \to (Z,z)$ with $\dHL_z(X,Y)<\epsilon$ and $\dKR_z(\mu,\nu)<\epsilon$.
\Cref{eps_isom_pGH} gives a  $2\epsilon$-isometry $f\colon (X,x)\to (Y,y)$ such that $\zeta(x',f(x'))\leq \epsilon$ for all $x'\in X\cap B(x,1/\epsilon)$.
For any $0<r\leq 1/\epsilon-\epsilon$, \Cref{eps_isom_pushforward} implies \eqref{close_equiv_push}.

Conversely, \Cref{eps_isom_pGH} gives a pointed metric space $(Z,\zeta,z)$ and isometric embeddings $(X,x),(Y,y) \to (Z,z)$ such that $\zeta(x',f(x'))\leq 2\epsilon$ for all $x'\in X\cap B(x,1/\epsilon)$ and $\dHL_z(X,Y)\leq 2\epsilon$.
\Cref{eps_isom_pushforward} implies, for any $0<r\leq 1/2\epsilon$,
\[\dKR^{r,r}_z(\mu,\nu) \leq \dKR^{r,r}_z(f_{\#}\mu,\nu) + \epsilon r \mu(B(x,r+\epsilon)).\]
Thus
\[\dKR_z(\mu,\nu) \leq 1/r + \dKR^{r,r}_z(f_{\#}\mu,\nu) + \epsilon r \mu(B(x,r+\epsilon)).\]
Taking the infimum over all $0<r\leq 1/\epsilon -\epsilon$ completes the proof.
\end{proof}

\begin{definition}
	\label{isom_mms}
	An isometry $(X,\mu,x)\to (Y,\nu,y)$ of pointed metric measure spaces is an isometry $(X,x)\to (Y,y)$ with $\mu = \nu$.
\end{definition}

\begin{corollary}
\label{pmGH_metric}
On the set $\Mpm$ of isometry classes of proper pointed metric measure spaces, $\dpmGH$ is a complete and separable metric.
\end{corollary}

\begin{proof}
By \Cref{G_completeness} we know $\dpmGH$ is a complete pseudometric.
It is easily verified that the set of finite pointed metric spaces with rational distances and rational combinations of Dirac masses is dense.
By \Cref{dpGH_metric}, any complete limit of proper pointed metric spaces is also proper.

Let
\[\dpmGH((X,\mu,x),(Y,\nu,y))=0.\]
By \Cref{eps_isom_pmGH}, for each $i\in\N$ there exists a $1/i$-isometry $f_i\colon (X,x)\to (Y,y)$ such that, for any $L,r>0$,
\begin{equation}\label{this_conv_KR}
\dKR_{L,r}((f_i)_{\#}\mu,\nu) \to 0.
\end{equation}
Fix $r>0$.
As in the proof of \Cref{dpGH_metric}, for each $j\in\N$ let $\mathcal N_j$ be a finite $1/j$-net of $B(x,r)$.
By taking a subsequence, we may suppose the $f_i$ converge pointwise to an isometry $\iota$ when restricted to $\cup_j\mathcal N_j$ whose image is a dense subset of $B(y,r)$.
We extend $\iota$ to an isometry $\iota\colon B(x,r)\to B(y,r)$.

For a moment fix $i,j\in \N$.
For any $w\in B(y,r)$ let $z\in \mathcal N_j$ with $w \in B(z,1/j)$.
Since $f_i$ is an $\epsilon_i$-isometry and $\iota$ is an isometry, 
\begin{align*}\
\rho(\iota(w),f_i(w)) &\leq \rho(\iota(w),\iota(z)) + \rho(\iota(z),f_i(z)) + \rho(f_i(z),f_i(w))\\
&\leq d(w,z) + \rho(\iota(z),f_i(z)) + d(w,z)+\epsilon_i\\
&\leq \rho(\iota(z),f_i(z)) + 2/j +\epsilon_i.
\end{align*}
Thus, for any $r$-Lipschitz $g\colon Y\to [-1,1]$ with $\spt g\subset B(y,r)$,
\begin{equation*}
	\int g\d((f_i)_{\#}\mu - \iota_{\#}\mu) \leq r\left(\sup_{z\in \mathcal N_j}\rho(f_i(z),\iota(z)) + 2/j +\epsilon_i\right).
\end{equation*}
Consequently,
\[\dKR^{r,r}_y((f_i)_{\#}\mu, \iota_{\#}\mu) \leq r\left(\sup_{z\in \mathcal N_j}\rho(f_i(z),\iota(z)) + 2/j +\epsilon_i\right)\]
and so
\[\limsup_{i\to\infty} \dKR^{r,r}_y((f_i)_{\#}\mu, \iota_{\#}\mu) \leq 2r/j.\]
Since this is true for all $j\in\N$, \eqref{this_conv_KR} gives $\dKR^{r,r}_y(\iota_{\#}\mu,\nu)=0$
for all $r>0$.
Thus \Cref{dKR_ms} implies $\iota_{\#}\mu=\nu$.
Consequently, $\dpmGH$ is a metric on $\Mpm$.
\end{proof}

\begin{corollary}\label{pmgh_compact}
	A set $\mathcal S \subset (\Mpm,\dpmGH)$ is pre-compact if and only if
\begin{enumerate}
		\item \label{pmgh_set}$\{(X,x): (X,\mu,x)\in\mathcal S\} \subset (\Mp,\dpGH)$ is pre-compact and
		\item \label{pmgh_meas} For every $r>0$,
$\{\mu(B(x,r)): (X,\mu,x)\in\mathcal S\}$ is bounded.
\end{enumerate}
\end{corollary}

\begin{proof}
Suppose that $\mathcal S$ satisfies \cref{pmgh_set,pmgh_meas} and, for each $i\in\N$, let $(X_i,\mu_i,x_i)\in \mathcal S$.
After passing to a subsequence, there exists a proper pointed metric space $(X,x)$ such that
\[\dpGH((X_i,x_i),(X,x))<1/(2i)\]
for each $i\in\N$.
By \Cref{eps_isom_pGH}, for each $i\in\N$ there exists a $1/i$-isometry $f_i\colon (X_i,x_i)\to (X,x)$.
\Cref{pmgh_meas} and \Cref{prokhorov} imply, after passing to a further subsequence, that $(f_i)_{\#}\mu_i$ converges to some measure $\mu$ on $X$.
\Cref{eps_isom_pmGH} then implies
\[\dpmGH((X_i,\mu_i,x_i),(X,\mu,x)) \to 0.\]

The converse statement is immediate.
\end{proof}

\subsection{Gromov--Hausdorff convergence of large subsets}
Gromov--Hausdorff convergence of pointed metric measure spaces is used throughout analysis.
However, the requirement that the underlying metric spaces must Gromov--Hausdorff converge is too rigid for some uses in geometric measure theory. 
For example, it requires that every sequence of null sets must Gromov--Hausdorff converge to some limit.
	
	For a concrete example, consider the following.
	\begin{example}
	\label{gh_example}
		Let $X=\{0,1\}$ with $d(0,1)=1$ and for each $i\in \N$ let $\mu_i(\{0\})=1$ and $\mu(\{1\})=1/i$.
		Then the sequence $(X,d,\mu_i,0)$ Gromov--Hausdorff converges to $(X,d,\delta_0,0)$ but \emph{not} to $(\{0\},\delta_0,0)$, which is the more natural choice.
		That is, Gromov--Hausdorff convergence may require one to consider a larger metric space than simply the support of the limit measure.
		
		For each $i\in\N$ let $X_i$ be an $i$-point metric space including a point $0$, equipped with the discrete metric $d_i$.
		Then the sequence $(X_i,d_i,\delta_0,0)$ has no Gromov--Hausdorff limit, even though a one point metric space with a Dirac measure is the natural candidate from a measure theoretic point of view.
		
		To construct a 1-rectifiable metric space with no Gromov--Hausdorff tangents at any point, one takes an isometric copy of an interval and scatters "dust" around it at all scales as in the previous example.
	\end{example}

Rather than Gromov--Hausdorff convergence we will use the following.
\begin{definition}\label{dGHs}
	For pointed metric measure spaces $(X,\mu,x)$ and $(Y,\nu,y)$ define
	\[\dGHs((X,\mu,x),(Y,\nu,y))=\dG_{0,1}((X,\mu,x),(Y,\nu,y)).\]
\end{definition}

Note that $\dGHs \leq \dpmGH$.
Also, for any Borel $x\in S\subset X$, \eqref{dKR_subset} implies
\begin{equation}
	\label{dGHs_subset}
	\dGHs((X,\mu,x),(S,\mu\vert_{S},x)) \leq \inf\{r>0 :\mu(B(x,1/r)\setminus S) <r\}.
\end{equation}

Although $\dGHs$ does not require the underlying sets to be close in the sense of $\dpGH$, we can recover large subsets which are.
\begin{lemma}
	\label{KR_implies_GH}
	Let $(Z,z)$ be a complete pointed metric space and $\mu,\nu\in \Mloc(Z)$.
	For any $r,\epsilon,\delta>0$ there exist compact sets
	\[K_\mu \subset B(z,r)\cap \spt\mu \quad \text{and} \quad K_\nu\subset B(z,r)\cap \spt\nu\]
	with
	\begin{equation}\label{x_sm_mea}\max\{\mu(B(z,r)\setminus K_\mu),\ \nu(B(z,r)\setminus K_\nu)\} < (1+\delta) \dKR^{1/\epsilon,r}_z(\mu,\nu),\end{equation}
	and
	\begin{equation}\label{k_mu_hd}\dHL_z(K_\mu,K_\nu) \leq \max\left\{\frac{1}{r-\epsilon},\epsilon\right\}.\end{equation}
\end{lemma}

\begin{proof}
	There exist compact
	\[K'_\mu \subset B(z,r)\cap\spt\mu \quad \text{and}\quad K'_\nu\subset B(z,r)\cap\spt\nu\]
	with
	\begin{equation}
	\label{big_meas}
	\max\{\mu(B(z,r)\setminus K'_\mu),\ \nu(B(z,r)\setminus K'_\nu)\}< \delta\dKR^{1/\epsilon,r}_z(\mu,\nu)
	\end{equation}
	Let
	\[\tilde K_\mu = (K'_\mu \cap U(z,r-\epsilon)) \setminus B(K'_\nu,\epsilon),\]
	a relatively open set in $K'_\mu$,
	and define $g\colon Z \to [0,1]$ by
	\[g(x)= \begin{cases}
		1 & z\in \tilde K_\mu\\
		1-\frac{1}{\epsilon}\min\{d(z,\tilde K_\mu),\epsilon\} & \text{otherwise}.
	\end{cases}\]
	Then $g$ is $1/\epsilon$-Lipschitz with $\spt g \subset B(z,r)$.
	By construction,
	\[\mu(\tilde K_\mu) \leq \int g\d\mu\]
	and
	\begin{equation*}\int g\d\nu \leq \nu(B(z,r)\setminus K'_\nu).
	\end{equation*}
	Therefore,
	\begin{align}\label{snd_big_meas}
		\mu(\tilde K_\mu) &\leq \dKR_z^{1/\epsilon,r}(\mu,\nu) + \nu(B(z,r)\setminus K'_\nu)\notag\\
		&\leq (1+\delta) \dKR_z^{1/\epsilon,r}(\mu,\nu),
	\end{align}
	by \eqref{big_meas}.
	Define
	\begin{equation}\label{kmu_alt}K_\mu := K'_\mu \setminus \tilde K_\mu = (K'_\mu\cap B(K'_\nu,\epsilon)) \cup (K'_\mu \setminus U(z,r-\epsilon)),\end{equation}
	a compact set, and, by \cref{big_meas,snd_big_meas},
	\begin{equation*}
	\mu(B(z,r)\setminus K_\mu) < (1+2\delta) \dKR_z^{1/\epsilon,r}(\mu,\nu).
	\end{equation*}
	Also, by \eqref{kmu_alt},
	\begin{equation}
		\label{this_local_HD}
		K_\mu\cap U(z,r-\epsilon) \subset B(K'_\nu,\epsilon).
	\end{equation}
	
	By the symmetric argument, the compact set
	\[K_\nu := (K'_\nu\cap B(K'_\mu,\epsilon)) \cup (K'_\nu \setminus U(z,r-\epsilon))\]
	satisfies
	\begin{equation*}
		\nu(B(z,r)\setminus K_\nu) < (1+2\delta) \dKR_z^{1/\epsilon,r}(\mu,\nu)
	\end{equation*}
	and
	\begin{equation}
		\label{this_local_HD2}
		K_\nu\cap U(y,r-\epsilon) \subset B(K'_\mu,\epsilon).
	\end{equation}
	
	Finally, if $x\in K_\mu\cap U(z,r-\epsilon)$, let $y\in K'_\nu$ with $x\in B(y,\epsilon)$.
	Then $y\in B(K'_\mu,\epsilon)$ and so $y\in K_\nu$.
	Thus
	\[K_\mu \cap U(x,r-\epsilon) \subset B(K_\nu,\epsilon).\]
	Similarly,
	\[K_\nu \cap U(y,r-\epsilon) \subset B(K_\mu,\epsilon).\]
	Thus $\dHL_z(K_\mu,K_\nu)\leq \max\{1/(r-\epsilon),\epsilon\}$.
	Therefore, $K_\mu,K_\nu$ have the required properties with $\delta$ replaced by $2\delta$.
	Since $\delta>0$ is arbitrary, this suffices.
\end{proof}

\begin{theorem}
	\label{d*_equiv_GH}
	For $\epsilon>0$ let $(X,\mu,x)$ and $(Y,\nu,y)$ be pointed metric measure spaces with
	\[\dGHs((X,\mu,x),(Y,\nu,y))<\epsilon.\]
	There exist compact
	\begin{equation}\label{def_K_mu}x\in K_\mu\subset \spt\mu\cap B(x,1/\epsilon) \quad \text{and}\quad y\in K_\nu \subset \spt\nu \cap B(y,1/\epsilon)\end{equation}
	such that
	\begin{equation}
	\label{pmgh_small_meas}	
	\max\{\mu(B(x,1/\epsilon)\setminus K_\mu), \nu(B(y,1/\epsilon)\setminus K_\nu)\} < \epsilon
	\end{equation}
	and
	\begin{equation}
	\label{pmgh_conc1}
	\dpmGH((K_\mu,\mu\vert_{K_\mu},x),(K_\nu,\nu\vert_{K_\nu},y))< 3\epsilon.	
	\end{equation}
	
	Conversely, if there exist compact $K_\mu,K_\nu$ as in \eqref{def_K_mu} satisfying \eqref{pmgh_small_meas} and
	\begin{equation}
	\label{pmgh_small}\dGHs((K_\mu,\mu\vert_{K_\mu},x),(K_\nu,\nu\vert_{K_\nu},y))<\epsilon
	\end{equation}
	then
	\begin{equation*}
	\label{pmgh_conc2}
	\dGHs((X,\mu,x),(Y,\nu,y)) < 3\epsilon.	
	\end{equation*}
	
	That is, $\dGHs$ is bi-Lipschitz equivalent to the metric defined by taking the infimum over all $0<\epsilon<1/2$ for which \cref{pmgh_small_meas,pmgh_conc1} hold.
\end{theorem}

\begin{proof}
If $\epsilon\geq 1/2$ there is nothing to prove.
Otherwise, by the definition of $\dGHs$, there exists a pointed metric space $(Z,z)$ and isometric embeddings $(X,x),(Y,y)\to (Z,z)$ with $\dKR_z(\mu,\nu)<\epsilon$.

Let $\delta>0$ be such that $(1+\delta)\dKR_z(\mu,\nu)<\epsilon$ and let $K_\mu$ and $K_\nu$ be obtained from \Cref{KR_implies_GH}, so that \eqref{pmgh_small_meas} is satisfied.
Then \eqref{k_mu_hd} implies
\[\dHL_z(K_\mu,K_\nu)\leq 2\epsilon.\]
Finally, \cref{pmgh_small_meas,dKR_subset} and the triangle inequality imply
\begin{align*}
\dKR_z^{1/\epsilon,1/\epsilon}(\mu\vert_{K_\mu},\nu\vert_{K_\nu}) &\leq \dKR_z^{1/\epsilon,1/\epsilon}(\mu\vert_{K_\mu},\mu) + \dKR_z^{1/\epsilon,1/\epsilon}(\mu,\nu) + \dKR_z^{1/\epsilon,1/\epsilon}(\nu,\nu\vert_{K_\nu})\\
&< 3\epsilon.
\end{align*}
Thus \eqref{pmgh_conc1} is established.

Conversely, by \eqref{dGHs_subset} and the triangle inequality,
\begin{align*}
	\dGHs((X,\mu,x),(Y,\nu,y)) &\leq \dGHs((X,\mu,x),(K_\mu,\mu\vert_{K_\mu},x))\\
	&\quad + \dGHs((K_\mu,\mu\vert_{K_\mu},x),(K_\nu,\nu\vert_{K_\nu},y))\\
	&\qquad + \dGHs((K_\nu,\nu\vert_{K_\nu},y),(Y,\nu,y))\\
	&< 3\epsilon.
\end{align*}
\end{proof}

Working with measures allows us to take limits of $\epsilon$-isometries (see also \cite[Proposition 3.2]{MR3477230}, which shows the compactness of the set of isometries).
\begin{lemma}
	\label{isom_limit}
	Let $(X,d,\mu,x)$ and $(Y,\rho,\nu,y)$ be pointed metric measure spaces.
	Suppose that, for each $i\in\N$, there exists a Borel set $x\in S_i\subset X$ with
	\begin{equation}\label{f_lrg}
		\mu(B(x,i)\setminus S_i) \to 0
	\end{equation}
	and a $1/i$-isometry $f_i\colon S_i\to S'_i\subset Y$such that
	\begin{equation}\label{fsurj}
	\dKR_y((f_i)_{\#}(\mu\vert_{S_i}),\nu)\to 0.
	\end{equation}
	Then there exists an isometry $\iota\colon (\spt\mu\cup\{x\},\mu,x)\to (\spt\nu\cup\{y\},\nu,y)$.
\end{lemma}

\begin{proof}
	By the inner regularity of $\mu$, we may suppose each $S_i$ is compact.
	For each $i\in\N$ and $w\in S_i$, let $g_i(w)=(w,f_i(w))\in X\times Y$ and define
	\[\lambda_i=(g_i)_{\#}(\mu\vert_{S_i}) \in \Mloc(X\times Y).\]
	By \Cref{prokhorov}, after passing to a subsequence, we may suppose that $(g_i)_{\#}(\mu\vert_{S_i})$ converges to some measure $\lambda \in \Mloc(X\times Y)$.
	Note that the pushforward of $\lambda_i$ to $X$ under the projection map is $\mu\vert_{S_i}$.
	By \eqref{f_lrg} the projection of $\spt\lambda$ to $X$ has full $\mu$ measure and by \eqref{fsurj} the projection of $\spt\lambda$ to $Y$ has full $\nu$ measure.
	
	If $(x^1,y^1)\in \spt\lambda$ then there exist $x^1_i\in S_i$ with $g_i(x^1_i)\to (x^1,y^1)$.
	Indeed, if not then there exist $r>0$ and $i_k\to \infty$ such that
	\[\lambda_{i_k}(B((x^1,y^1),r))=\mu_{i_k}(S_{i_k} \cap g_{i_k}^{-1}(B((x^1,y^1),r)))=0 \quad \forall k\in\N,\]
	contradicting the convergence of $\lambda_i$ to $\lambda$ using \eqref{inf_open_ball}.
	In particular, if $(x^1,y^1),(x^2,y^2)\in \spt\lambda$, there exist $x^1_i,x_i^2\in S_i$ with $(x^1_i,f_i(x^1_i)) \to (x^1,y^1)$ and $(x^2_i,f_i(x_i^2)) \to (x^2,y^2)$.
	Since each $f_i$ is a $1/i$-isometry, this implies
	\[\rho(y^1,y^2)=\lim_{i\to\infty} \rho(f_i(x^1_i),f_i(x_i^2)) = \lim_{i\to\infty} d(x_i^1,x_i^2)= d(x^1,x^2).\]
	That is, $\spt\lambda$ lies on the graph of an isometry.
	This isometry is defined $\mu$ almost everywhere and maps onto a set of full $\nu$-measure.
	Since $Y$ is complete, $\iota$ can be extended to an isometry of $\spt\mu$ to $\spt\nu$, as required.
\end{proof}

\begin{corollary}
	\label{d*_isom}
	Let $\M$ be the set of equivalence classes of all pointed metric measure spaces under the relation
	\[(X,\mu,x)\sim (Y,\nu,y) \text{ if } (\spt\mu\cup\{x\},\mu,x) \text{ is isometric to } (\spt\nu\cup\{y\},\nu,y).\]
	Then $(\M,\dGHs)$ is a complete and separable metric space.
\end{corollary}

\begin{proof}
\Cref{G_completeness} implies $\dGHs$ is a complete pseudometric on the set of all pointed metric measure spaces.
It is easily verified that the set of finite pointed metric spaces with rational distances and rational combinations of Dirac masses is dense.

We next show that $\dGHs$ is well defined on $\M$.
Indeed, if $(X,d,\mu,x)\sim (Y,\rho,\nu,y)$, let
\[\iota \colon \spt\mu\cup\{x\} \to \spt\nu\cup\{y\}\]
be an isometry of pointed metric measure spaces.
Define $\zeta$ on $X\sqcup Y$ by $\zeta=d$ on $X$, $\zeta=\rho$ on $Y$ and
\[\zeta(x,y)=\inf\{d(x,w) + \rho(f(w),y): w\in \spt\mu\cup \{x\}\}\]
whenever $x\in X$ and $y\in Y$.
A similar argument to \Cref{eps_isom_pGH} shows that $\zeta$ is a pseudometric on $X\sqcup Y$.
Indeed, if $z,z'\in X$ and $t\in Y$ then for any $w,w'\in \spt\mu\cup \{x\}$, the triangle inequality for $d$ and $\rho$ and the fact that $f$ is an isometry implies
\begin{align*}
	\zeta(z,z') &\leq d(z,w)+d(w,w') + d(w',z')\\
	&= d(z,w) + \rho(f(w),f(w')) + d(w',z')\\
	&\leq d(z,w) + \rho(f(w),t) + \rho(t,f(w')) + d(w',z').
\end{align*}
Taking the infimum over $w,w'\in \spt\mu\cup \{x\}$ gives $\zeta(z,z')\leq \zeta(z,t)+\zeta(t,z')$.
A similar argument for $z,z'\in Y$ and $t\in X$ gives the triangle inequality for $\zeta$.
Note that $\zeta(x,f(x))=0$ for all $x\in \spt\mu\cup \{x\}$.

Let $(Z,\zeta)$ be the metric space obtained from $(X\sqcup Y,\zeta)$ and let $z\in Z$ be the point corresponding to $x=y$.
Then $(X,x)$ and $(Y,y)$ isometrically embed into $(Z,z)$ and, since $f$ is an isometry of metric measure spaces, the pushforwards of $\mu$ and $\nu$ under these embeddings agree.
In particular, $\dKR_z(\mu,\nu)=0$ in $Z$ and hence
\[\dGHs((X,d,\mu,x),(Y,\rho,\nu,y))=0.\]
That is, $\dGHs$ is well defined on $\M$.

By combining \Cref{eps_isom_pmGH}, \Cref{isom_limit} and \Cref{d*_equiv_GH}, any pointed metric measure spaces satisfying
\[\dGHs((X,\mu,x),(Y,\nu,y))=0\]
are isometric.
Hence $\dGHs$ is a metric on $\M$.
\end{proof}

When considering $(X,\mu,x)\in \M$, the set $X\setminus (\spt\mu\cup\{x\})$ often plays no role.
In such cases, to simplify notation, we will write $(\mu,x)$ for the element of $\M$.

\begin{remark}\label{gigli_rmk}
	Sturm \cite{sturm} defines a metric on the set of metric measure spaces equipped with a probability measure with finite variance.
	Greven, Pfaffelhuber and Winter \cite{MR2520129} define the \emph{Gromov--Prokhorov} metric on the set of compact metric measure spaces equipped with a probability measure.
	The definitions of the metrics are analogous to that of $\dGHs$, but use the $L_2$ transportation distance and the Prokhorov metric respectively in place of $\dKR$ and there are no requirements imposed on a distinguished point.
	In this setting, one can slightly modify the definition of $\dGHs$ to not make any requirements on a distinguished point, or modify the definitions of \cite{sturm} or \cite{MR2520129} to consider a distinguished point.
	After doing so, since $\dKR$, the $L_2$ transportation distance and the Prokhorov metric all metrise weak* convergence in the respective settings, the metrics all define the same topology whenever they are defined.
	
	Gigli, Mondino and Savar\'e \cite{MR3477230} define a notion of convergence on the set of pointed metric measure spaces by $(X_i,\mu_i,x_i) \to (X,\mu,x)$ if there exist a complete and separable metric space $Z$ and isometric embeddings $\spt\mu_i,\spt\mu\to Z$ with $x_i\to x$ and $\mu_i\to \mu$ in $\Mloc(Z)$.
	By \Cref{G_completeness} we see that this definition agrees with convergence with respect to $\dGHs$, using \Cref{same_basepoint} to ensure that the distinguished points are mapped to the same point.
	Several equivalent definitions of this convergence for \emph{non-zero} measures are given in \cite[Theorem 3.15]{MR3477230}.
	
	Shortly after a first version of the present paper appeared I became aware of Pasqualetto and Schultz \cite{schultz}.
Amongst other results, they prove a relationship between pointed measured Gromov--Hausdorff convergence and the convergence of \cite{MR3477230}.
This relationship corresponds to the relationship between $\dpmGH$ and $\dGHs$ given by \Cref{d*_equiv_GH}.
\end{remark}

\begin{theorem}\label{dGHs_compact}
	A set $\mathcal S\subset (\M,\dGHs)$ is pre-compact if and only if, for every $\epsilon>0$ and $(\mu,x)\in \mathcal S$, there exists a compact $K_{\mu,x}\subset B(x,1/\epsilon)$ with
		\begin{equation}
		\label{large_gromov}
		\mu(B(x,1/\epsilon) \setminus K_{\mu,x}) \leq \epsilon	
		\end{equation}
		such that
		\[\{(K_{\mu,x},\mu\vert_{K_{\mu,x}},x): (\mu,x) \in\mathcal S\} \subset (\Mpm,\dpmGH)\]
		is pre-compact.
\end{theorem}

\begin{proof}
	Let $(\mu_i,x_i)$ be a sequence in $\mathcal S$ and, for $i,j\in\N$ let $K_{i}^{j}\subset B(x_i,j)$ be as in \eqref{large_gromov} for $\epsilon=1/j$.
	Let $\mathcal K(i,j)=(K_i^j,(\mu_i)\vert_{K_i^j},x_i)$.
	First note that \eqref{dGHs_subset} implies
	\begin{equation}\label{also_conv_Xi}
	\dGHs(\mathcal K(i,j),(\mu_i,x_i))\leq 1/j \quad \forall i,j\in\N 
	\end{equation}
	and hence
	\begin{equation}
	\label{unif_large_gromov}
	\dGHs(\mathcal K(i,j),\mathcal K(i,j')) \leq 2/j \quad \forall i\in\N,\ \forall j\leq j'\in\N.
	\end{equation}
	
	By assumption, for each $j\in\N$ there exists a $\mathcal K_j\in \Mpm$ and a subsequence $i^j_k$ such that
	\[\dpmGH(\mathcal K(i_k^j,j),\mathcal K_j) \to 0 \text{ as } k\to\infty.\]
	By taking a diagonal subsequence, we may suppose that
	\[\dpmGH(\mathcal K(i,j),\mathcal K_j) \to 0 \quad \forall j\in\N.\]
	\Cref{unif_large_gromov} then implies
	\begin{equation*}
	\dGHs(\mathcal K_j,\mathcal K_{j'}) \leq 2/j \quad \forall j\leq j'\in\N.
	\end{equation*}
	Since $(\M,\dGHs)$ is complete, $\mathcal K_j$ converges to some $(\mu,x)$.
	\Cref{also_conv_Xi} and the triangle inequality then imply $(\mu_{i},x_{i})$ converges to $(\mu,x)$.
	
	Conversely, suppose that $\mathcal S$ is pre-compact but the conclusion does not hold.
	Note that, for any $r>0$, the set $\{\mu(B(x,r)): (\mu,x)\in \mathcal S\}$ must be bounded.
	Therefore, there exist $\delta,\epsilon>0$ and, for each $i\in\N$, a $(\mu_i,x_i)\in\mathcal S$ such that, for any $K\subset \spt\mu_i\cap B(x_i,1/\epsilon)$ with
	\[\mu(B(x_i,1/\epsilon)\setminus K)\leq \epsilon,\]
	$K$ contains $i$ points separated by distance at least $\delta$.
	After passing to a subsequence, we may suppose that $(\mu_i,x_i)$ converges to some $(\mu,x)$.
	By \Cref{G_completeness}, there exists a complete and separable pointed metric space $(Z,z)$ and isometric embeddings $(\spt\mu_i\cup\{x_i\},x_i)\to (Z,z)$ with $\mu_i\to \mu$.
	\Cref{prokhorov} implies that there exists a compact $K\subset Z$ with $\mu_i(B(z,r)\setminus K)\leq \epsilon$ for all $i$, contradicting our initial assumption.
\end{proof}

\begin{remark}
	\Cref{dGHs_compact} implies \cite[Proposition 7.1]{MR2520129} and \cite[Corollary 3.22]{MR3477230}.
The properties of $\dGHs$ and its relationship to $\dpmGH$ allow for a simpler proof.
\end{remark}

\begin{lemma}
\label{lem_doub}
	For $0<\delta<1$, $0<\epsilon<1/2$, and $k\geq 1$ let $R>0$ satisfy
	\begin{equation}\label{ref-bound}
	R>(4R(\epsilon/\delta)^{\frac{1}{k}})^{-1}.
	\end{equation}
	Suppose that $(X,\mu,x)$ is a $2^k$-doubling pointed metric measure space with $
	\mu(B(x,R))\geq \delta$ and
	suppose that $S\subset X$ satisfies
	\[\mu(B(x,2R)\setminus S)<\epsilon.\]
	Then
	\begin{equation}
	\label{doubling_pmGH}
	\dpmGH((X,\mu,x),(S,\mu\vert_S,x))\leq 4R(\epsilon/\delta)^{\frac{1}{k}} +\epsilon.
	\end{equation}
\end{lemma}

\begin{proof}
	For $0<r<R$ suppose that $w\in B(x,R)\setminus B(S,r)$.
	Then
	\[\mu(B(w,r)) \leq \mu(B(x,R+r)\setminus S) < \epsilon\]
	and so
	\[\delta\leq \mu(B(x,R)) \leq \mu(B(w,2R))\leq (2^k)^m \mu(B(w,r)) < \epsilon 2^{km},\]
	for $m= \lceil\log_2 2R/r\rceil$.
	Therefore
	\[r<4R(\epsilon/\delta)^{\frac{1}{k}}\]
	and so
	\[B(x,R) \subset B(S,4R(\epsilon/\delta)^{\frac{1}{k}}).\]
	Combining this with \eqref{ref-bound} gives
	\[B(x,(4R(\epsilon/\delta)^{\frac{1}{k}})^{-1}) \subset B(S,4R(\epsilon/\delta)^{\frac{1}{k}}),\]
	so that
	\[H_x(X,S) \leq 4R(\epsilon/\delta)^{\frac{1}{k}}.\]
	Since $\mu(B(x,2R)\setminus S)<\epsilon$, \eqref{doubling_pmGH} follows from \eqref{dKR_subset}.
\end{proof}

We also note the following partial converse to the fact that $\dGHs\leq \dpmGH$.
\begin{corollary}\label{cor_doub}
For $0<\delta<1$ and $k\geq 1$ let $\mathcal S$ be the equivalence classes of those $2^k$-doubling pointed metric measure spaces $(X,\mu,x)$ with $\mu(B(x,1))\geq \delta$.
Then, when restricted to $\mathcal S$,
\begin{equation}
\label{doub_conc}
\dpmGH \leq 9\delta^{-\frac{1}{k}} \dGHs^{\frac{1}{2k}}.	
\end{equation}
\end{corollary}

\begin{proof}
For $0<\epsilon<1/2$ suppose that $(X,\mu,x),(Y,\nu,y)\in \mathcal S$ with
\[\dGHs((X,\mu,x),(Y,\nu,y))<\epsilon.\]
By \Cref{d*_equiv_GH}, there exist compact $K_\mu\subset B(x,\epsilon^{-1/2k})$ and $K_\nu\subset B(y,\epsilon^{-1/2k})$ satisfying
\begin{equation}
	\label{this_small_conc_meas}
	\max\{\mu(B(x,\epsilon^{-1/2k})\setminus K_\mu),\ \nu(B(y,\epsilon^{-1/2k})\setminus K_\nu)\} < \epsilon
\end{equation}
and \cref{pmgh_conc1}.
\Cref{this_small_conc_meas} and \Cref{lem_doub} (applied with $R=\epsilon^{-1/2k}$, which satisfies \eqref{ref-bound} for $\delta\leq 1$) imply that
\[\dpmGH((X,\mu,x),(K_\mu,\mu\vert_{K_\mu},x)) \text{ and } \dpmGH((Y,\nu,y),(K_\nu,\nu\vert_{K_\nu},y))\]
are at most $4\delta ^{-\frac{1}{k}}\epsilon^{\frac{1}{2k}}+\epsilon$.
Therefore \eqref{pmgh_conc1} and the triangle inequality imply \eqref{doub_conc}.
\end{proof}

	\section{Tangent spaces of metric measure spaces}\label{sec_tangents}
	
	With the notion of convergence defined in \Cref{sec:convergence}, we may define a tangent space following Preiss \cite{MR890162}.	
	Given $(\mu,d,x)\in \M$ with $x\in\spt\mu$ and $r>0$, let
	\[T_{r}(\mu,d,x) := \left(\frac{\mu}{\mu(B(x,r))},\frac{d}{r}, x\right)\]
	
	\begin{definition}\label{def:gh-tangent}
		A $(\nu,\rho,y) \in \M$ is a \emph{tangent measure} of $(\mu,d,x)\in \M$ if there exist $r_k\to 0$ such that
		\[\dGHs(T_{r_k}(\mu,d,x),(\nu,\rho, y))\to 0.\]
		The set of all tangent measures of $(\mu,d,x)$ will be denoted by $\Tan(\mu,d,x)$.
	\end{definition}

	\begin{remark}
		\label{eqiv_tangents}
		If $X=\ell_2^m$ for some $m\in\N$, any tangent measure is also supported on $\ell_2^m$.
		In this case, $\Tan(\mu,x)$ consists of the equivalence classes of the tangent measures of Preiss.
		
		By \Cref{cor_doub}, $\Tan(\mu,x)$ agrees with the usual definition of Gromov--Hausdorff tangent spaces for doubling metric measure spaces.
	\end{remark}
	
	\begin{proposition}\label{tangents_exist}
		Let $X$ be a metric space and let $\mu\in \Mloc(X)$ be asymptotically doubling.
		For \muae $x\in \spt\mu$ and every $r_0>0$,
		\begin{equation}\label{con_subseq}
		\{T_r(\mu,x): 0<r<r_0\} \text{ is pre-compact.}
			\end{equation}
		In particular, for any $x\in \spt\mu$ satisfying \eqref{con_subseq}, $\Tan(\mu,x)$ is a non-empty compact metric space when equipped with $\dGHs$ and
		\begin{equation}\label{contra}
		\forall \delta>0,\ \exists r_x>0 \text{ s.t. } \dGHs(T_{r}(\mu,x),\Tan(\mu,x)) \leq\delta\ \forall 0<r<r_x.
		\end{equation}
	\end{proposition}
	
	\begin{proof}
		By \Cref{doubling_decomposition}, we know that we may cover $\mu$ almost all of $\spt\mu$ by
		\[X_m := \{x\in \spt\mu : \mu(B(x,2r)) \leq m\mu(B(x,r))\ \forall 0<r<1/m\}\]
		with $m\in\N$.
		Fix an $m\in \N$.
		By \Cref{thm:lebesgue-density}, $\mu$ almost every $x\in X_m$ is a Lebesgue density point of $X_m$.
		Fix such an $x$.
		It suffices to prove the conclusion for $x$.
		To do so, fix $R\geq 2$.

		First note that, if $R/(2m) \leq r<r_0$ then
		\[\frac{\mu(B_{d/r}(x,R))}{\mu(B_d(x,r))}=\frac{\mu(B_d(x,Rr))}{\mu(B_d(x,r))} \leq \frac{\mu(B_d(x,Rr_0))}{\mu(B_d(x,R/(2m)))} < \infty.\]
		Here the subscripts on the balls indicate the metrics used to define the balls, for $d$ the original metric in $X$.
		Also, if $0<r< R/(2m)$ then
		\[\frac{\mu(B_{d/r}(x,R))}{\mu(B_d(x,r))} = \frac{\mu(B_{d}(x,R r))}{\mu(B_d(x,r))} \leq m^{4\log_2 R}.\]
		Thus
		\[\{\nu(B(x,R)) : (\nu,x) \in T_r(\mu,x),\ 0<r<r_0\}\]
		is bounded.

		Secondly, for any $\epsilon>0$, 	since $x$ is a Lebesgue density point of $X_m$, there exists $r_1>0$ such that
		\[\frac{\mu(B_d(x,r)\cap X_m)}{\mu(B_d(x,r))} < \epsilon\]
		for all $0<r<r_1$.
		Let $r_2=\min\{r_1,R/(2m)\}$.
		Also, let $K\subset B(x,Rr_0)$ satisfy
		\[\mu(B_d(x,Rr_0)\setminus K) \leq \epsilon.\]
		Then for any $r_2 \leq r<R_0$,
		\[\frac{\mu(B_{d/r}(x,R)\setminus K)}{\mu(B_d(x,r))} = \frac{\mu(B_d(x,Rr)\setminus K)}{\mu(B_d(x,r))}\leq \frac{\epsilon}{\mu(B_d(x,r_2))}.\]
		Since $K$ is compact, the metric spaces $(K,d/r)$ with $r_2 \leq r<R_0$ are uniformly totally bounded.
		On the other hand, by Lemma \ref{lem:locally-doubling}, if $0<r<r_2$ then
		\[B_{d/r}(x,R)\cap X_m= B_{d}(x,Rr) \cap X_m\]
		is covered by $m^{-4\log_2 \epsilon}$ balls of the form $B_{d/r}(w_i,\epsilon)$.
		Thus
		\[\{B_{d/r}(x,R) : 0<r<r_0\}\]
		is uniformly totally bounded.
		Therefore, \Cref{dGHs_compact} proves \eqref{con_subseq}.
		
		Now let $x\in \spt \mu$ satisfy \eqref{con_subseq}.
		By applying \eqref{con_subseq} to an arbitrary sequence $r_i\to 0$ we see that $\Tan(\mu,x)$ is non-empty. To see that it is compact, for each $j\in \N$ let $(\nu_j,y_j) \in \Tan(\mu,x)$ and let $0<r_j < 1/j$ be such that
		\[\dGHs(T_{r_j}(\mu, x), (\nu_j,y_j))< 1/j.\]
		By \eqref{con_subseq} there exists a subsequence $r_{j_k}\to 0$ and a $(\nu,y)\in \M$ such that
		\[\dGHs(T_{r_{j_k}}(\mu,x),(\nu,y))\to 0.\]
		In particular, $(\nu,y)\in \Tan(\mu,x)$ and, by the triangle inequality, $(\nu_{j_k},y_{j_k}) \to (\nu,y)$,
		as required.
		Finally, given $\delta>0$, the existence of such an $r_0$ is given by the contrapositive to \eqref{con_subseq}.
	\end{proof}

	\begin{lemma}
		\label{nhood_compact_closed}
		Let $(X,d,\mu)$ be a metric measure space and suppose that $\mathcal S\subset \M$.
		Then for any $R,\delta>0$,
		\[C_{R,\delta}(\mathcal S) := \{x\in \spt\mu: \dGHs((T_{r}(\mu,d,x),\mathcal S) \leq \delta\ \forall 0<r< R\}\]
		is a closed subset of $X$.
		In particular, if $S\subset X$ is a Borel set of points satisfying \eqref{con_subseq} and $\mathcal C\subset \M$ is closed,
		\[\mathcal G(\mathcal C):=\{x\in S : \Tan(\mu,d,x)\subset \mathcal C\}\]
		is Borel.
	\end{lemma}

	\begin{proof}
		Let $x\in X$ and $x_j\in C_{R,\delta}(\mathcal S)$ with $x_j \to x$.
		First note that, since $\mu\in\Mloc(X)$, all but countably many $r>0$ satisfy
		\begin{equation*}\mu(B(x,r)\setminus U(x,r))=0.\end{equation*}
		For such an $r>0$, $T_{r}(\mu,d,x_j) \to T_r(\mu,d,x)$.
		Indeed, \cref{sup_closed_ball,inf_open_ball} imply that $\mu(B(x_j,r)) \to \mu(B(x,r))$ and consequently \Cref{same_basepoint} implies
		\begin{equation*}
			\dpmGH\left(\left(\frac{\mu}{\mu(B(x_j,r))},\frac{d}{r} ,x_j \right), \left(\frac{\mu}{\mu(B(x,r))},\frac{d}{r}, x \right)\right)\to 0
		\end{equation*}
		Therefore
		\begin{equation}\label{this_subs_cl}\dGHs(T_{r}(\mu,x),\mathcal S) \leq \delta.\end{equation}
		For arbitrary $0<r<R$, let $r_j\downarrow r$ satisfy \eqref{this_subs_cl}.
		Then $\mu(B(x,r_j))\to \mu(B(x,r))$ and so
		\begin{equation*}
			\dpmGH\left(\left(\frac{\mu}{\mu(B(x,r_j))},\frac{d}{r_j},x \right),\left(\frac{\mu}{\mu(B(x,r))},\frac{d}{r},x \right)\right)\to 0
		\end{equation*}
		(for any $R>0$, the identity is a $(r_j-r)R$-isometry on $B(x,R)$).
		Consequently \eqref{this_subs_cl} holds in this case too and hence $C_{R,\delta}(\mathcal S)$ is closed.
		
		If $x$ satisfies \eqref{con_subseq} and $\Tan(\mu,d,x)\subset \mathcal C$ then by \eqref{contra},
		\begin{equation}\label{contra2}
			\forall \delta>0,\ \exists r_x>0 \text{ s.t. } \dGHs(T_{r}(\mu,x),\mathcal C) \leq\delta\ \forall 0<r<r_x.
		\end{equation}
		Conversely, if $\mathcal S$ is closed, \eqref{contra2} implies $\Tan(\mu,x)\subset \mathcal C$.
		Thus,
		\[\mathcal G(\mathcal C)= S\cap \bigcap_{\delta\in \mathbb Q^+}\bigcup_{R\in \mathbb Q^+} C_{R,\delta}(\mathcal C),\]
		a Borel set.
	\end{proof}
	
	We now prove three results that correspond results in Preiss \cite{MR890162}.
	The following statement, and its proof, is essentially that of \cite[Theorem 2.12]{MR890162} reformulated in our setting. See also \cite{MR3334235,MR2865538} for adaptations to Gromov--Hausdorff tangent spaces and the appendix of \cite{MR3990192} for an adaptation to tangents of sets of finite perimeter in RCD spaces.
	\begin{proposition}
		\label{change-basepoint}
		Let $(X,d,\mu)$ be an asymptotically doubling metric measure space. Then $\mu$-a.e.\ $x\in \spt\mu$ has the following property:
		\begin{equation}\label{shift_base}
			\forall(\nu,y) \in \Tan(\mu,x) \text{ and } \forall z\in \spt\nu,\ (\nu,z)\in \Tan(\mu,x).
		\end{equation}
	\end{proposition}

	\begin{proof}
		For each $p\in\N$ let $A_{p}$ be the set of $x\in \spt\mu$ for which there are $(\nu_x,y_x)\in \Tan(\mu,x)$
		and $b_x\in \spt\nu_x$ such that
		\[\dGHs((\nu_x,b_x), T_{r}(\mu, x))> \frac{1}{p} \quad \forall 0<r<1/p.\]
		Note that, if $x\not\in \cup_{p\in\N}A_{p}$, then $x$ satisfies the required conclusion.

		Suppose that, for some $p\in\N$, $\mu(A_{p})>0$.
		Then by the separability of $(\M,\dGHs)$, there exists a $E\subset A_{p}$ with $\mu(E)>0$ such that
		\[\dGHs((\nu_x,b_x),(\nu_y,b_y))< 1/2p\]
		for every $x,y\in E$. By \Cref{thm:lebesgue-density} there exists a Lebesgue density point $x$ of $E$.
		Let $r_k\to 0$ be such that $T_{r_k}(\mu, x) \to (\nu_x,y_x)$.
		By \Cref{G_completeness} there exists a complete and separable pointed metric space $(Z,\zeta,z)$ and isometric embeddings $(\spt\mu,d/r_k,x),(\spt\nu,y_x) \to (Z,\zeta,z)$ with
		\[\frac{\mu}{\mu(B(x,r_k))} \to \nu_x.\]
		Let $\mu_k$ be the isometric copy of $\mu\in \Mloc(X,d/r_k)$, let $E_k$ be the isometric copy of $(E,d/r_k)$ and let $c_k\in E_k$ be such that
		\[\zeta(b_x,c_k)<\zeta(b_x,E_k) + 1/k.\]

		Note that
		\begin{equation}
			\label{preiss-star}
			\zeta(b_x,E_k)\to 0.
		\end{equation}
		Indeed, since $x$ is a density point of $E$, if $R=2\zeta(y_x,b_x)$ then
		\begin{equation}\label{dp_consq}\lim_{k\to \infty} \frac{\mu_k(E_k\cap B(x,R))}{\mu_k(B(x,R))} = \lim_{k\to \infty} \frac{\mu(E\cap B(x,R r_k))}{\mu(B(x,R r_k))} = 1.\end{equation}
		On the other hand, suppose that there exists $\delta>0$ such that $B(b_x,\delta)\cap E_k = \emptyset$ for infinitely many $k$.
		\Cref{inf_open_ball,sup_closed_ball} imply
		\begin{align*}
			\liminf_{k\to \infty} \frac{\mu_k(E_k\cap B(x,R))}{\mu_k(B(x,R))} &\leq 1- \limsup_{k\to\infty} \frac{\mu_k(B(b_x,\delta))}{\mu_k(B(x,R))}\\
			&\leq 1- \frac{\nu_x(B(b_x,\delta/2))}{\nu_x(B(y_x,2R))}<1,
		\end{align*}
		since $b_x\in \spt \nu_x$.
		This contradicts \eqref{dp_consq}.
		
		Now, \eqref{preiss-star} implies $\zeta(b_x,c_k)\to 0$.
		This and \Cref{same_basepoint} imply
		\[\dGHs\left(\left(\frac{\mu}{\mu(B(x,r_k))},\frac{d}{r_k},c_k\right),(\nu_x,b_x)\right) \to 0.\]
		Therefore, there exists $k$ with $r_k<1/p$ such that
		\[\dGHs\left(\left(\frac{\mu}{\mu(B(x,r_k))},\frac{d}{r_k},c_k\right), (\nu_x,b_x)\right)<1/2p.\]
		Thus, since each $c_k\in E_k$,
		\begin{align*}
			1/p &< \dGHs\left(\left(\frac{\mu}{\mu(B(x,r_k))},\frac{d}{r_k},c_k\right), (\nu_{c_k},b_{c_k})\right)\\
			&\leq \dGHs\left(\left(\frac{\mu}{\mu(B(x,r_k))},\frac{d}{r_k},c_k\right),(\nu_x,b_x)\right) + \dGHs((\nu_x,b_x), (\nu_{c_k},b_{c_k})) \\
			&\leq 1/2p + 1/2p,
		\end{align*}
		a contradiction.
		Therefore we must have $\mu(A_{p})=0$ for all $p\in\N$.
	\end{proof}

	\begin{lemma}\label{tangents_doubling}
		Let $(\mu,x)\in \M$ with $x\in\spt\mu$ and
		\begin{equation}
			\label{as-doubling-control}
			M_x:=\limsup_{r\to 0}\frac{\mu(B(x,2r))}{\mu(B(x,r))}.
		\end{equation}
		For every $(\nu,y) \in \Tan(\mu,x)$ and $r>0$,
		\[\nu(B(y,2r)) \leq M_x \nu(B(y,r))\]
		and
		\[1\leq \nu(B(y,1))\leq M_x.\]
		In particular, if $x$ satisfies \eqref{shift_base}, every $(\nu,y)\in \Tan(\mu,x)$ is $M_x$-doubling.
	\end{lemma}

	\begin{proof}
		By \Cref{G_completeness,inf_open_ball,sup_closed_ball}, for any $(\nu,y)\in \Tan(\mu,x)$ there exist $r_i\to 0$ such that, for any $R>0$ and $\lambda>1$,
		\[\nu(B(y,2R)) \leq \liminf_{i\to\infty} \frac{\mu(B(x,\lambda 2Rr_i))}{\mu(B(x,r_i))}\]
		and
		\[\nu(B(y,\lambda R)) \geq \limsup_{i\to\infty} \frac{\mu(B(x,\lambda Rr_i)}{\mu(B(x,r_i))}.\]
		Consequently,
		\begin{equation*}
			\frac{\nu(B(y,2R))}{\nu(B(y,\lambda R))}\leq M_x.
		\end{equation*}
	Since this is true for all $\lambda>1$ we also have
	\begin{equation}
		\label{doubling-control}
		\frac{\nu(B(y,2R))}{\nu(B(y,R))}\leq M_x.
	\end{equation}
	
	Also note that, by \eqref{sup_closed_ball},
	\[\nu(B(y,1)) \geq \limsup_{i\to\infty}\frac{\mu(B(x,r_i))}{\mu(B(x,r_i))}=1,\]
	and by \eqref{inf_open_ball},
	\[\nu(B(y,1))\leq \nu(U(x,2)) \leq \liminf_{i\to\infty} \frac{\mu(U(x,2))}{\mu(B(x,1))} \leq \liminf_{i\to\infty} \frac{\mu(B(x,2))}{\mu(B(x,1))} \leq M_x.\]
	\end{proof}

	\begin{corollary}
		\label{tangent-density-equal}
		Let $(\mu,d,x) \in \M$ and let $S\subset \spt\mu$ be $\mu$-measurable with $x$ a density point of $S$.
		Then
		\[\Tan(\mu\vert_S,x) = \Tan(\mu,x).\]
	\end{corollary}

\begin{proof}
	Combine \eqref{dGHs_subset} and the definition of $\Tan$.
\end{proof}

\section{Proof of the main theorems} \label{proof_of_main}
We first demonstrate how tangent measures supported on elements of $\biLip(K)$ lead to a decomposition of the space into sets with good tangential approximation.

Define the projection $\pi_1\colon \Mpm\to \Mp$ by
\[(X,\mu,x)\mapsto (X,x)\]
and $\pi_2\colon \Mpm\to \M$ by
\[(X,\mu,x) \mapsto (\mu,x).\]
Both of $\pi_1,\pi_2$ are 1-Lipschitz.

For $M\geq 1$ let
\[S_1(M) =\{(X,\mu,x)\in\Mpm : M\geq \mu(U(x,1)),\ \mu(B(x,1))\geq 1/M\}\]
and $S_2(M)$ be the set of those $(X,\mu,x)\in\Mpm$ with
\begin{equation}\label{nearly_doub}\mu(U(y,r)) \leq M \mu(B(y,2r)) \quad \forall y\in X \text{ and } r>0.\end{equation}
For $K\geq 1$ let
\[\mathcal K(K,M) = \pi_1^{-1}(\biLip(K)) \cap S_1(M) \cap S_2(M).\]

\begin{lemma}\label{biLip_compact_M}
	For any $K,M\geq 1$,
	\begin{enumerate}
		\item \label{Mtop1}$\biLip(K)$ is a compact subset of $(\Mp,\dpGH)$
		\item \label{Mtop2}$S_1(M)$ and $S_2(M)$ are closed subsets of $(\Mpm,\dpmGH)$
		\item \label{Mtop3}$\mathcal K(K,M)$ is a compact subset of $(\Mpm,\dpmGH)$ and hence a compact subset of $(\M,\dGHs)$.
	\end{enumerate}
\end{lemma}

\begin{proof}
	First note that the elements of $\biLip(K)$ are proper and so do belong to $\Mp$.
	Moreover, for any $r>0$, the set
	\[\{B(x,r): (X,x)\in \biLip(K)\}\]
	is uniformly totally bounded and so, by \Cref{gromov_compactness}, $\biLip(K)$ is pre-compact.
	
	To see that $\biLip(K)$ is closed, for each $i\in\N$ let $(X_i,x_i)\in \biLip(K)$ and $(X,x)\in\Mp$ with
	\[\dpGH((X_i,x_i),(X,x))<1/i.\]
	For each $i\in\N$ let $f_i\colon (X_i,x_i)\to (X,x)$ be a $2/i$-isometry given by \Cref{eps_isom_pGH}.
	Also let $\psi_i \colon \ell_\infty^n \to X_i$
	be surjective and $K$-bi-Lipschitz and define
	\[\phi_i=f_i\circ \psi_i \colon \ell_\infty^n\to X.\]
	
	Fix $R>0$.
	For any $y,z\in B(0,R)\subset \ell_\infty^n$,
	\begin{align}
		d(\phi_i(y),\phi_i(z)) &= d(f_i(\psi_i(y)),f_i(\psi_i(z)))\notag\\
		&\leq d(\psi_i(y),\psi_i(z)) +2/i\notag\\
		&\leq K\|y-z\|_\infty +2/i.\label{phi_lip}
	\end{align}
	Thus the $\phi_i$ are equicontinuous on $B(0,R)$.
	Since the $\phi_i$ map into $B(x,KR+2)$, a compact set, the Arzel\`a--Ascoli theorem gives a $\phi\colon B(0,R) \to X$ such that, after passing to a subsequence, $\phi_i\to \phi$ uniformly on $B(0,R)$.
	Since this is true for all $R>0$, by taking a diagonal subsequence, we see that there exists a $\phi\colon \ell_\infty^n\to X$ such that $\phi_i\to \phi$ uniformly on bounded sets.
	A similar estimate to \eqref{phi_lip} shows that $\phi$ is $K$-bi-Lipschitz.
	Since each $\phi_i$ is surjective and each $f_i$ is a $2/i$-isometry, $\phi$ is surjective.
	Thus $\biLip(K)$ is closed, proving \cref{Mtop1}.
		
	\Cref{Mtop2} follows from \Cref{G_completeness} and \cref{inf_open_ball,sup_closed_ball}.

	\Cref{Mtop3} follows since $\mathcal K(K,M)$ is pre-compact and closed in $(\Mpm,\dpmGH)$ by \Cref{pmgh_compact} and \cref{Mtop1,Mtop2}.
\end{proof}

For $K\geq 1$ let
\[\biLip(K)^*=\{(\mu,x)\in \M : \spt\mu \in \biLip(K)\}.\]

\begin{proposition}
	\label{AR_tangent_decomp}
	Let $X$ be a complete metric space and $S\subset X$ a $\mathcal{H}^n$-measurable set with $\H^n(S)<\infty$.
	Suppose that, for $\H^n$-a.e.\ $x\in S$, $\Theta^n_*(S,x)>0$ and there exists a $K_x\geq 1$ such that
	\begin{equation*}
		\Tan(\H^n\vert_S, x) \subset \biLip(K_x)^*.
	\end{equation*}
	
	There exist compact $C_i\subset S$ with
	\[\mathcal H^n\left(S\setminus \bigcup_{i\in \mathbb N} C_i\right)=0\]
	and, for each $i\in\N$, an $\eta_i>0$ such that, for each $x\in C_i$,
	\begin{equation}\label{AR}
		\eta_i r^n \leq \mathcal{H}^n(B(x,r)\cap S) < 2(2r)^n \ \forall 0<r<\eta_i
	\end{equation}
	and
	\begin{equation}\label{AR_tangents} \Tan(\mathcal H^n\vert_S,x) \subset \mathcal K(\eta_i^{-1},\eta_i^{-1}).
	\end{equation}
	
	Further, for each $i\in\mathbb N$ and any $\epsilon>0$, there exists countably many closed $G_j\subset C_i$ with
	\[\mathcal H^n\left(C_i\setminus\bigcup_{j\in\N} G_j\right)=0,\]
	such that each $(C_i,G_j)$ satisfies $GTA(\eta_i,\eta_i^{-1},\epsilon,\min\{\eta_i,1/j\})$.
\end{proposition}

\begin{proof}
	For any $i\in \mathbb N$, \Cref{doubling_decomposition} implies that
	\[S_i := \{x\in S : r^n/i \leq \mathcal{H}^n(B(x,r)\cap S) < 2(2r)^n \ \forall 0<r<1/i\}\]
    is $\H^n$-measurable.
	By assumption, the $S_i$ monotonically increase to almost all of $S$ as $i\to \infty$.

	Fix $i\in\mathbb N$.
	For $j\in \mathbb N$, let
	\[S_i^j:=\{x\in S_i : \Tan(\mathcal H^n\vert_S,x) \subset \mathcal K(j,j)\}.\]
	If $x\in S_i$ satisfies \eqref{shift_base}, \Cref{tangents_doubling} implies that every $(\nu,y)\in \Tan(\H^n\vert_S,x)$ is $4^{n+1} i$-doubling with $1\leq \nu(B(y,1))\leq 4^{n+1}i$.
	Therefore, if $\max\{K_x,4^{n+1}i\}\leq j$ then $x\in S_i^j$ and so $\H^n(S\setminus \cup_{i,j} S_i^j)=0$.
	\Cref{biLip_compact_M} implies that $\mathcal K(j,j)$ is compact and so \Cref{nhood_compact_closed} implies each $S_i^j$ is $\H^s$-measurable.
	Consequently, each can be decomposed as a countable union of compact sets (up to a set of measure zero).
	Re-indexing completes the proof of the first decomposition.

	Now fix $i\in\N$ and let $C=C_i$ and $\eta=\eta_i$.
	Since $\mathcal H^n(C)<\infty$, \Cref{hd_density} implies
	 $\Theta^{*,n}(C,x)\leq 1$ for $\mathcal H^n$-a.e.\ $x\in C$ and since $\mathcal H^n(S\setminus C)<\infty$, $\Theta^{n}_*(C,x) \geq \eta$ for $\mathcal H^n$-a.e.\ $x\in C$.
	Thus $\mathcal H^n\vert_{C}$ is asymptotically doubling.
	Therefore, by \Cref{tangent-density-equal}, \Cref{thm:lebesgue-density,AR_tangents},
	\begin{equation}\label{tangents_subset}
		\Tan(\mathcal H^n\vert_{C},x)\subset \mathcal K(\eta^{-1},\eta^{-1})\quad \text{for } \mathcal H^n\text{-a.e. } x\in C.
	\end{equation}
	
	For any $\epsilon>0$ the sets
	\[C_{j,\epsilon}:=C_{1/j,\epsilon}(\mathcal K(\eta^{-1},\eta^{-1}))\]
    defined in \Cref{nhood_compact_closed} are closed and $\cup_j C_{j,\epsilon}$ contains $C$.
	For any $j\in\N$, any $x\in C_{j,\epsilon}$ and any $0<r<1/j$, let $(\nu,y)\in \mathcal K(\eta^{-1},\eta^{-1})$ with
	\[\dGHs(T_r(\H^n\vert_C,x), (\nu,y)) < 2\epsilon.\]
	Apply \Cref{d*_equiv_GH} to obtain a $K\subset C
	\cap B(x,r/2\epsilon)$ and $K_\nu\subset B(y,1/2\epsilon)$ with
	\begin{equation}\label{this_large_pmgh}\max\left\{\frac{\H^n(C\cap B(x,r/2\epsilon)\setminus K)}{\H^n(C\cap B(x,r))}, \nu(B(y,1/2\epsilon)\setminus K_\nu)\right\} < 2\epsilon\end{equation}
	and
	\begin{equation}\label{this_close_pmgh}\dpmGH(T_r(K,\H^n\vert_K,x),(K_\nu,\nu\vert_{K_\nu},y))< 6\epsilon.\end{equation}
	Note that \eqref{nearly_doub} implies $(Y,\nu)$ is $2\eta^{-1}$-doubling.
	\Cref{this_large_pmgh} allows us to apply \Cref{lem_doub} with $k=\log_2 2\eta^{-1}$, $R=\epsilon^{-\frac{1}{2k}}$ and $\delta=1$ to obtain
	\[\dpmGH((K_\nu,\nu\vert_{K_\nu},y),(\spt\nu,\nu,y)) \leq 9\epsilon^{\frac{1}{2k}}.\]
	The triangle inequality and \eqref{this_close_pmgh} then give
	\[\dpmGH(T_r(K,\H^n\vert_K,x),(\spt\nu,\nu,y)) \leq 18\epsilon^{\frac{1}{2k}}.\]
	In particular,
	\[\dpGH((K,d/r,x),\biLip(\eta^{-1}))\leq 18\epsilon^{\frac{1}{2k}}.\]
	Also, \cref{this_large_pmgh,AR} imply
	\[\H^n(B(x,r)\setminus K) <2\epsilon \H^n(B(x,r)) \leq 4\cdot4^n \epsilon r^n = \eta 4( 4 (\epsilon/\eta)^{1/n} r)^n.\]
	That is, $C_{j,\epsilon}$ satisfies $GTA(\eta,\eta^{-1},\epsilon^*,\min\{\eta,1/j\})$, for some $\epsilon^*$ depending upon $\eta$ such that $\epsilon^*\to 0$ as $\epsilon \to 0$.
	Since $\epsilon>0$ was arbitrary, setting $G_j=C_{1/j,\epsilon'}$ for each $j\in\N$ and appropriate $\epsilon'$ completes the proof.
\end{proof}

To prove that \eqref{main_bliptan} implies \eqref{main_rect} in \Cref{main_thm}
we require the following result \cite[Theorem 1.1]{perturbations}.
It concerns the set of all 1-Lipschitz functions $f\colon X \to \ell_2^n$, equipped with the supremum norm.
\begin{theorem}\label{perturbations}
	Let $X$ be a complete metric space and $m\in\N$.
	Suppose that $S\subset X$ is purely $n$-unrectifiable with $\mathcal{H}^n(S)<\infty$ and satisfies $\Theta_{*}^n(S,x)>0$ for $\mathcal{H}^n$-a.e.\ $x\in S$.
	The set of all 1-Lipschitz $f \colon X \to \ell_2^m$ with $\mathcal{H}^n(f(S)) = 0$ is residual and hence dense.
\end{theorem}

\begin{corollary}\label{useful_corollary}
	Let $X$ be a complete metric space and let $S\subset X$ be $\mathcal{H}^n$-measurable with $\mathcal{H}^n(S)<\infty$ and $\Theta_{*}^n(S,x)>0$ for $\mathcal{H}^n$-a.e.\ $x\in S$.
	For $L\geq 1$ and $\delta>0$
	suppose that
	\begin{enumerate}
		\item \label{usef_meas} $\iota\colon [0,r]^n \to X$ is continuous with
		\begin{equation*}
		\mathcal{H}^n_{\infty}(\iota([0,r]^n) \setminus S) < \mathcal{H}^n([\delta r,(1-\delta)r]^n)/L.
		\end{equation*}
		\item \label{usef_pert} $f \colon \iota([0,r]^n) \to [0,r]^n\subset \ell_2^n$ is $L$-Lipschitz with		\begin{equation*}
		\|f(\iota(x))-x\|_2 < \delta r \quad \forall x\in \partial [0,r]^n.
	\end{equation*}
	\end{enumerate}
	Then $S$ is not purely $n$-unrectifiable.
\end{corollary}

\begin{proof}
	Suppose that $S$ is purely $n$-unrectifiable and let $\epsilon>0$ be such that
	\begin{equation*}
		\mathcal{H}^n_\infty(\iota([0,r]^n)\setminus S) < \mathcal{H}^n([(\delta+\epsilon)r,(1-(\delta+\epsilon))r]^n)/L.
	\end{equation*}
	Let $g\colon \iota([0,r]^n) \to \ell_2^n$ be given by \Cref{perturbations} such that
	\begin{itemize}
		\item $g$ is $L$-Lipschitz;
		\item $\|f-g\|_\infty <\epsilon r$;
		\item $\mathcal{H}^n(g(S))=0$.
	\end{itemize}
	In particular,
	\begin{equation}
		\label{g_pert_boundary}
		\|g(\iota(x))-x\|<(\delta+\epsilon)r \quad \forall\ x\in \partial [0,r]^n
	\end{equation}
	and, since $\mathcal{H}^n_\infty=\mathcal{H}^n$ in $\ell_2^n$,
	\begin{align}
		\mathcal{H}^n(g(\iota([0,r]^n))) &\leq \mathcal{H}^n(g(\iota([0,r]^n)\setminus S)) + \mathcal{H}^n(g(S))\notag\\
		&\leq L \mathcal{H}^n_\infty(\iota([0,r]^n)\setminus S) + 0 \notag\\
		&< \mathcal{H}^n([(\delta+\epsilon)r,(1-(\delta+\epsilon))r]^n).\label{img_small_meas}
	\end{align}
	Note that \eqref{g_pert_boundary} implies that
	\begin{equation*}
	    g(\iota([0,r]^n)) \supset [(\delta+\epsilon)r,(1-(\delta+\epsilon))r]^n
	\end{equation*}
	(this is a consequence of Brouwer's fixed point theorem; See \cite[Lemma 7.1]{perturbations} for a proof involving the unit ball).
	Consequently,
	\begin{equation*}
		\mathcal{H}^n(g(\iota([0,r]^n))) \geq \mathcal{H}^n([(\delta+\epsilon)r,(1-(\delta+\epsilon))r]^n),
	\end{equation*}
	contradicting \eqref{img_small_meas}
\end{proof}

Recall the definition of $\biLip(K)^*$ preceding \Cref{AR_tangent_decomp}.
    \begin{theorem}\label{3_implies_1}
    Let $(X,d)$ be a complete metric space, $n\in \mathbb{N}$ and $E\subset X$ a $\mathcal{H}^n$-measurable set with $\H^n(E)<\infty$.
    Suppose that, for $\H^n$-a.e.\ $x\in E$, $\Theta^n_*(E,x)>0$ and there exists a $K_x\geq 1$ such that
    \begin{equation}\label{flat_tangents}
        \Tan(X,d,\H^n\vert_E,x) \subset \biLip(K_x)^*.
    \end{equation}
    Then $E$ is $n$-rectifiable.
	\end{theorem}
	
	\begin{proof}
		By the Kuratowski embedding, we may identify $X$ with an isometric copy contained in $\ell_\infty$.

	Suppose that $E$ is not $n$-rectifiable.
	Then by \Cref{rect_decomp} there exists a compact purely $n$-unrectifiable $S\subset E$ with $ \mathcal{H}^n(S)>0$.
	By \Cref{hd_density}, $\Theta_*^n(S,x)>0$ for $\mathcal{H}^n$-a.e.\ $x\in S$.
	In particular, $\H^n\vert_S$ is asymptotically doubling and so by \Cref{tangent-density-equal},
	\begin{equation*}
        \Tan(\H^n\vert_{S},x) \subset \biLip(K_x)^*
    \end{equation*}
	for $\mathcal H^n$-a.e.\ $x\in S$.
	Then $S$ satisfies the hypotheses of \Cref{AR_tangent_decomp}.
	Since $\H^n(S)>0$, \Cref{AR_tangent_decomp} implies the existence of a compact $C\subset S$ with $\H^n(C)>0$ and a $0<\eta<1$ such that, for any $\epsilon>0$, there exist countably many compact $G_j\subset C$ with $\H^n(C\setminus G_j)=0$ such that each $(C,G_j)$ satisfies $GTA(\eta,\eta^{-1},\epsilon,\min\{\eta,1/j\})$.
	
	Fix $0<\gamma<1$ and $0<\delta<\eta/4\sqrt{n}$ and let $m_0\in \N$ be given by \Cref{main_construction} for these parameters and $K=\eta^{-1}$.
	Let $m\geq m_0$ be such that
	\begin{equation}
		\label{eps_m_choice}
		20 \sqrt{n}\eta^{-2} 2^{-m\gamma/2} + \sqrt{n}2^{-m} < \delta
	\end{equation}
	and set $\epsilon=4^{-m}$.
	Since $\H^n(C)>0$, there exists $j\in \N$ for which $G_j$ obtained in the previous paragraph satisfies $\H^n(G_j)>0$.
	Set $G=G_j$ and $R_0=\min\{\eta,1/j\}$, so that $(C,G)$ satisfies $GTA(\eta,\eta^{-1},4^{-m},R_0)$.
	Let $\Lambda>0$ be obtained from \Cref{holder_construction}.
	Fix $x$ a density point of $G$ with $\Theta^{*,n}(X,x)\leq 1$ and let $0<r<R_0$ be such that
	\begin{equation}\label{high_den_x}\Lambda \H^n(B(x,20\eta^{-1} r)\cap X\setminus G) < \delta r^n.\end{equation}
	Then \Cref{main_construction} gives an
	$\iota\colon [0,r]^n \subset \ell_\infty^n \to B(x, 20\eta^{-1}r) \subset \ell_\infty$
	satisfying \Cref{main_construction} \cref{holder_bilip,holder_perturb,holder_full_meas}.

	Since $\iota\vert_{\mathcal D(r,m)}^{-1}$ is $2\eta^{-1}$-Lipschitz and takes values in $[0,r]^n$,
	an application of the second conclusion of \Cref{holder-extension} extends it to a $\sqrt{n}2\eta^{-1}$-Lipschitz map $f\colon \iota([0,r]^n) \to [0,r]^n\subset \ell_2^n$.
	By \Cref{main_construction} \cref{holder_perturb}, for any $p,q\in [0,r]^n$ with $\|p-q\|_\infty\leq 2^mr$ we have
	\begin{align*}
		\|f(\iota(p))-f(\iota(q))\|_2 &\leq \sqrt{n}2\eta^{-1} \|\iota(p)-\iota(q)\|_\infty\\
		&\leq 20 \sqrt{n}\eta^{-2} 2^{-m\gamma/2} r.
	\end{align*}
	In particular, for any $p\in [0,r]^n$, pick $q\in \mathcal D(r,m)$ with $\|p-q\|_\infty \leq 2^{-m}r$.
	Then
	\begin{align*}
		\|f(\iota(p))-p\|_2 &\leq \|f(\iota(p)) -f(\iota(q))\|_2 + \|f(\iota(q))-p\|_2\\
		 &\leq 20 \sqrt{n}\eta^{-2} 2^{-m\gamma/2} r + \|q-p\|_2\\
		 & \leq 20 \sqrt{n}\eta^{-2} 2^{-m\gamma/2} r + \sqrt{n}2^{-m} r\\
		 &< \delta r
	\end{align*}
	by \eqref{eps_m_choice}.
	Thus, $f$ satisfies \Cref{useful_corollary} \cref{usef_pert}.
	
	Combining \Cref{main_construction} \cref{holder_full_meas} and \eqref{high_den_x} implies that $f$ satisfies \Cref{useful_corollary} \cref{usef_meas}.
	An application of \Cref{useful_corollary} shows that $S$ is not purely $n$-unrectifiable, a contradiction.
    \end{proof}
    
	Finally we prove that \eqref{main_rect} implies \eqref{main_utan} in \Cref{main_thm}.

	\begin{theorem}
		\label{tan_of_rect}
		Let $(X,d)$ be a complete metric space, $n\in\mathbb N$ and $E\subset X$ an $n$-rectifiable set with $\mathcal H^n(E)<\infty$.
		For any $x\in E$ that satisfies the conclusion of \Cref{kirchheim} and \Cref{hd_density} \cref{upper_density},
		\begin{equation}
		 \label{convergence_to_tan}
		 T_{r}(X,d,\mathcal H^n\vert_E,x) \to (\mathbb R^n,\|\cdot\|_x,\mathcal H^n/2^n,0)
		\end{equation}
		as $r\to 0$.
		In particular, for $\mathcal H^n$-a.e.\ $x\in E$,
		\[\Tan(X,d,\mathcal H^n\vert_E,x) = \{(\mathbb R^n,\|\cdot\|_x,\mathcal H^n/2^n,0)\}.\]
	\end{theorem}
	
	\begin{proof}
	 Let $x\in E$ be as in the hypotheses and $f_x,\|\cdot\|_x,C_x$ as in the conclusion of \Cref{kirchheim}.
	 For each $r>0$, let
	 \begin{equation*}
	  \epsilon_r = \sup\left\{\left\vert1-\frac{\|f_x(y)- f_x(z)\|_x}{d(y,z)}\right\vert : y\neq z\in C_x\cap B(x,\sqrt{r})\right\},
	 \end{equation*} 
	 so that $\epsilon_r\to 0$ as $r\to 0$.
	 Also let $f_r=f_x/r$ and
	 \[C_r=C_x\cap B(x,\sqrt{r}).\]
	 
	 First note that, for any $y,z\in C_r$ with $d(y,z)/r \leq \min\{r^{-1/2},\epsilon_r^{-1/2}\}$,
	\begin{align}
	 \left\vert \frac{d(y,z)}{r} - \|f_r(y)-f_r(z)\|_x\right\vert &= \frac{\vert d(y,z)-\|f_x(y)-f_x(z)\|_x\vert}{r}\notag\\
	 &\leq \frac{\epsilon_r d(y,z)}{r}\notag\\
	 &\leq \sqrt{\epsilon_r}\label{kirch_eps_embed}.
	\end{align}
    
	Now observe that $f_r\colon (C_r,d/r)\to (\mathbb R^n,\|\cdot\|_x)$ is $L_r$-bi-Lipschitz, for
	 \[L_r = \frac{1+\epsilon_r}{1-\epsilon_r^2} \to 1.\]
	Therefore
    \begin{equation}
     \label{close_to_hd}
     \frac{1}{L_r^n}\mathcal H^n_x\vert_{f_r(C_r)} \leq (f_r)_{\#}(\H^n_r\vert_{C_r}) \leq L_r^n \H^n_x\vert_{f_r(C_r)},
    \end{equation}
    where $\H^n_r$ denotes $\H^n$ on the metric space $(C_r,d/r)$ and $\mathcal H^n_x$ denotes $\H^n$ on $(\mathbb R^n,\|\cdot\|_x)$.
    However,
    \begin{equation*}
     (f_r)_{\#}(\H^n_r\vert_{C_r}) = (f_r)_{\#}\left(\frac{\H^n\vert_{C_r}}{r^n}\right) = \frac{(f_r)_{\#}(\H^n_d\vert_{C_r})}{r^n},
    \end{equation*}
    where $\H^n_d$ denotes $\H^n$ on $(X,d)$.
    Also observe, for any $R>0$, by \cref{hd_meas_ball,k_density_2}.
	\begin{align}
		\label{comp_small_hd}\H^n_x(B(0,R)\setminus f_r(C_r)) &\leq \H^n_x(B(0,R) \setminus f_r(C_r\cap B(x,Rr/L_r)))\\
		&\leq (2R)^n - \H^n_x(f_r(C_r\cap B(x,Rr/L_r)))\notag\\
		&\leq (2R)^n - \frac{L_r^n}{r^n} \H^n_d(C_r\cap B(x,Rr/L_r))\notag\\
		&\to (2R)^n - (2R)^n =0\notag
	\end{align}
	Let $\|\cdot\|_\mathcal M$ denote the total variation of elements of $\mathcal M((B(0,R),\|\cdot\|_x))$.
	Then \cref{comp_small_hd,close_to_hd} imply
	\begin{align*}
		\left\|\frac{(f_r)_{\#}(\H^n_d\vert_{C_r})}{r^n} - \H^n_x\right\|_{\mathcal M} &\leq \left\|\frac{(f_r)_{\#}(\H^n_d\vert_{C_r})}{r^n} - \mathcal H^n_x\vert_{f_r(C_r)}\right\|_{\mathcal M}\\
		&\qquad +\|\mathcal H^n_x\vert_{f_r(C_r)}- \H^n_x\|_{\mathcal M}\\
		&\leq  (L_r^n-1)\H^n_x(B(0,R)) + \H^n_x(B(0,R)\setminus f_r(C_r))\to 0.
	\end{align*}
	That is,
	\begin{equation}\label{strong_conv}\frac{(f_r)_{\#}(\mathcal H^n_d\vert_{C_r})}{(2r)^n} \to \frac{\H^n_x}{2^n} \quad \text{strongly on } B(0,R).\end{equation}

	\Cref{strong_conv} has two consequences.
	Firstly, it shows that, for any $R,\delta>0$,
	\[B(f_r(C_r),\delta)\supset B(0,R)\]
	for sufficiently small $r>0$.
	Together with \eqref{kirch_eps_embed}, this shows that there exist $\epsilon'_r\to 0$ such that
	 \begin{equation}\label{k_eps_isom}
	 f_r\colon (C_r,d/r) \to (\mathbb{R}^n,\|\cdot\|_x) \quad \text{is an } \epsilon'_r\text{-isometry.}
	 \end{equation}
	\Cref{strong_conv} also implies that
	 \begin{equation}\label{kirch_meas_conv}
	 	\frac{(f_r)_{\#}(\mathcal H^n\vert_{C_r})}{r^n} \to \mathcal H^n_x \quad \text{in } \Mloc(\mathbb R^n,\|\cdot\|_x).
	 \end{equation}
	 \Cref{eps_isom_pmGH,k_eps_isom,kirch_meas_conv} then imply
	 \[\dpmGH(T_{r}(C_r,d,\mathcal H^n\vert_{C_r},x), (\mathbb R^n,\|\cdot\|_x,\mathcal H^n/2^n,0)) \to 0\]
	 and so \cref{k_density_2,d*_equiv_GH} imply \eqref{convergence_to_tan}.
	\end{proof}
	
	\Cref{3_implies_1,tan_of_rect,k_density} complete the proof of \Cref{main_thm}.
	To prove \Cref{main_marstrand}, note that, by the translation and scale invariance of $\mathbb{R}^n$, one can replace \eqref{conv_to_tan_metric} by
	\begin{equation}\label{conv_to_tan2}\frac{\dpGH((Y_r\cap B(0,r),0),(E_r,x))}{r}\to 0\end{equation}
	and obtain an equivalent hypothesis.
	\Cref{conv_to_tan2,exclude_small_metric} imply \eqref{flat_tangents} (by using \Cref{d*_equiv_GH}) and so
	\Cref{3_implies_1} implies \Cref{main_marstrand}.

	\bibliography{references}
\end{document}